\documentclass[11pt]{amsart}
\usepackage[utf8]{inputenc}

\DeclareUnicodeCharacter{00A0}{ }

\usepackage{etex}
\usepackage[margin=1.52in]{geometry}
\usepackage[parfill]{parskip}
\usepackage{setspace}
\usepackage{amsmath,amsfonts,amscd,amssymb,bm,amsthm,  mathrsfs}
\usepackage{tikz-cd}
\usetikzlibrary{positioning,arrows,calc}
\usepackage{movie15}
\usepackage{stmaryrd,enumerate}
\usepackage{graphicx}
\usepackage{adjustbox}
\usepackage{comment}
\usepackage[mathscr]{eucal}
\usepackage[pdfencoding=auto,colorlinks,,
            linkcolor=red,
            anchorcolor=blue,
            citecolor=green]{hyperref}
\usepackage{bookmark}

\usepackage[all]{xy}

\usepackage{mathtools}
\usepackage[backend=bibtex, style=numeric, maxnames=8]{biblatex}\bibliography{Pi1.bib}

\usepackage[all]{xy}
\usepackage{tikz}
\definecolor{myred}{RGB}{183,18,52}
\definecolor{myyellow}{RGB}{254,213,1}
\definecolor{myblue}{RGB}{0,80,198}
\definecolor{mygreen}{RGB}{0,155,72}
\newcommand{\mM}{\mathcal{M}}
\newcommand{\mS}{\mathcal{S}}

\def\eD{\EuScript{D}}

\def\eM{\EuScript{M}}

\def\eS{\EuScript{S}}

\def\eW{\EuScript{W}}

\def\eY{\EuScript{Y}}

\newcommand{\JJ}{\mathcal{J}}

\newcommand{\mJ}{\mathcal{J}}
\newcommand{\mT}{\mathcal{T}}

\newcommand{\mC}{\mathcal{C}}

\newcommand{\mX}{\mathcal{X}}

\newcommand{\mY}{\mathcal{Y}}
\newcommand{\mZ}{\mathcal{Z}}

\newcommand{\mR}{\mathcal{R}}
\newcommand{\mA}{\mathcal{A}}
\newcommand{\mG}{\mathcal{G}}

\newcommand{\mE}{\mathcal{E}}
\newcommand{\mL}{\mathcal{L}}
\newcommand{\mQ}{\mathcal{Q}}
\newcommand{\mB}{\mathcal{B}}
\newcommand{\mD}{\mathcal{D}}

\newcommand{\aA}{\mathbb{A}}

\newcommand{\DD}{\mathbb{D}}
\newcommand{\EE}{\mathbb{E}}
\newcommand{\ZZ}{\mathbb{Z}}

\newcommand{\CC}{\mathbb{C}}
\newcommand{\PP}{\mathbb{P}}
\newcommand{\QQ}{\mathbb{Q}}
\newcommand{\RR}{\mathbb{R}}
\newcommand{\TT}{\mathbb{T}}
\newcommand{\CP}{\mathbb{C}P}
\newcommand{\RP}{\mathbb{R}P}
\newcommand{\id}{\mbox{id}}

\def\arbar{\overline{Art}}

\def\G{\Gamma}
\newcommand\aut{\operatorname{Aut}}

\newcommand{\Conf}{\mbox{Conf}}
\newcommand{\Diff}{\mbox{Diff}}

\newcommand{\Aut}{\mbox{Aut}}

\newcommand{\PSL}{\mbox{PSL}}

\newcommand{\cov}{\mbox{cov}}
\newcommand{\mon}{\text{mon}}
\newcommand{\inte}{\text{int}}

\newcommand{\w}{\omega}

\newcommand{\coker}{\mbox{coker}}

\newcommand{\ov}{\overline}

\newcommand{\wt}{\widetilde}

\newcommand{\bP}{\mathbb{P}}

\newcommand{\red}{\textcolor{red}}

\newtheorem*{theorem*}{Theorem}

\newtheorem{thm}{Theorem}[section]
\newtheorem{dfn}[thm]{Definition}
\newtheorem{cor}[thm]{Corollary}
\newtheorem{lma}[thm]{Lemma}
\newtheorem{qes}[thm]{Question}
\newtheorem{prp}[thm]{Proposition}

\newtheorem{rmk}[thm]{Remark}
\newtheorem{conj}[thm]{Conjecture}

\begin{document}
\title{Symplectic Torelli groups of rational surfaces}
\author{Jun Li, Tian-Jun Li, and Weiwei Wu}
\begin{abstract}
  We call a symplectic rational surface $(X,\w)$ \textit{positive} if $c_1(X)\cdot[\w]>0$.  The positivity condition of a rational surface is equivalent to the existence of a divisor $D\subset X$, such that $(X, D)$ is a log Calabi-Yau surface.  

  The cohomology class of a symplectic form can be endowed with a \textit{type} using the root system associated to its Lagrangian spherical classes.  In this paper, we prove that the symplectic Torelli group of a positive rational surface is trivial if it is of type $\aA$, and is a sphere braid group if it is of type $\DD$.

  As an application, we answer affirmatively a long-term open question that Lagrangian spherical Dehn twists generate the symplectic Torelli group $Symp_h(X)$ when $X$ is a positive rational surface. We also prove that all symplectic toric surfaces have trivial symplectic Torelli groups.   Lastly, we verify that Chiang-Kessler's symplectic involution is Hamiltonian, answering a question of Kedra positively.

  Our key new input is the recent study of almost complex subvarieties due to Li-Zhang \cite{LZ15} and Zhang \cite{Zha17, Zhang16}.  Inspired by these works, we define a new \textit{coarse stratification} for the almost complex structures for positive rational surfaces.  We also combined symplectic field theory and the parametrized Gromov-Witten theories for our applications.

 \end{abstract}

\maketitle
\setcounter{tocdepth}{2}
\tableofcontents

\section{Introduction}

\subsection{Background and problems} 
\label{sub:background_and_problems}

It has been long noticed that the topology of symplectomorphism groups is connected to the space of compatible almost complex structures of a symplectic manifold.  This relation can be made very explicit in dimension $4$. Starting from Gromov's work, a series of foundational works \cite{Abr98,AM00,McDacs,Anj02,LP04,McDuffBlowup,AGK09} have made substantial progress in understanding the symplctomorphism groups on $S^2\times S^2$ and $\CP^2\#k\ov\CP^2$ when $k\le2$.

However, even for slightly larger rational surfaces, the space of almost complex structures turns out extremely hard to study.   When the symplectic form is monotone, the space of compatible almost complex structure behaves nicely, and the weak homotopy type of $Symp (X,\omega)$ was fully computed in \cite{Eva} when $X=\CP^2\#k\ov\CP^2$ for $3\le k\le5$.  When the symplectic form is not monotone, this problem becomes considerably harder.  See \cite{AP13} for the case of $k=3$.

The focus of the current paper is the lowest dimension of the homotopy group for $Symp(X)$, i.e. its connected components.  This is also called the \textbf{symplectic mapping class group}.  The homological action of symplectomorphism groups of a rational surface has been classified in \cite{LW12} (see also \cite{She10}), which is given by reflections along homology classes of embedded Lagrangian spheres.  Denote $Symp_h(X,\w)$ as the homologically trivial part of the symplectomorphism group, and we call $\pi_0(Symp_h(X,\w))$ the \textbf{symplectic Torelli group}.

Seidel and Tonkonog's results showed that the symplectic Torelli group is infinite for a large class of hypersurfaces in projective space \cite{Se99,Seithesis,Ton15}.  A recent result by Sheridan and Smith \cite{SS17} discovered a symplectic $K3$ surface for which the symplectic Torelli group is infinitely generated.  In the same paper, they summarized several open questions that has been known for a long time in folklore, and the following one is particularly relevant to our current result.

\begin{qes}[Donaldson]
 	Is the symplectic Torelli group generated by squared Dehn twists along Lagrangian spheres?
 \end{qes}

To date, many constructions of interesting group embeddings into the Torelli groups are known, see \cite{Sei08}, \cite{KS02}, \cite{ST04}, \cite{Kea12}, \cite{NBparametrizedGW} for example.  Most of these non-trivial subgroups are related to braid groups, which reflect features of the monodromies of certain moduli spaces.  Donaldson's question in rational surfaces has long expected a positive answer, partly due to its connection to the fundamental group of the configuration space of points.  The only cases where Donaldson's question is completely understood are the small rational surfaces $S^2\times S^2$ and $\CP^2\#k\ov\CP^2$ for $k\le5$ (see a detailed account in the introduction of \cite{Sei08,Eva}), mostly as a direct consequence of a complete understanding of the weak homotopy types of the symplectomorphism groups.  In \cite{LLW22}, the authors computed the symplectic mapping class group of the non-monotone $\CP^2\#5\ov\CP^2$, and an affirmative answer to Donaldson's question for these cases should be deduced from the result therein as well.

There is another seemingly unrelated question that originated from the last decade.	

\begin{qes}[Kedra]
	Does there exist a Hamiltonian action by a cyclic group, which does not extend to an $S^1$-Hamiltonian action?
\end{qes}

So far, the only example of a homologically trivial cyclic action that cannot be extended to a Hamiltonian $S^1$-Hamiltonian action is given by Chiang and Kessler \cite{CKfinite}.  Their example is a symplectic involution on a rational surface diffeomorphic to $\CP^2\#6\ov\CP^2$ with a non-monotone symplectic form, but it seems difficult to conclude whether this involution is Hamiltonian with their construction.  In addition, even though there are abundance of examples of cyclic actions from algebraic geometry, most of them seem to be either clearly non-Hamiltonian or hard to prove either way.

\subsection{Main results} 
\label{sub:main_results}

Recall that a symplectic rational surface is either a symplectic $S^2\times S^2$ or $\CP^2\#k\ov\CP^2$, $k\ge1$.

\begin{dfn}\label{d:positiveRational}
    We call a symplectic form $\w\in\Omega^2(X,\RR)$ $c_1$-\textbf{positive}, if $[c_1]\cdot[\w]>0$.  A symplectic rational surface $(X,\w)$ is called \textbf{positive} if $\w$ is $c_1$-positive.
\end{dfn}

Note that even though the positivity condition concerns only with the symplectic class, the following result of McDuff \cite{McD96} shows $Symp(X,\w)$ is irrelevant to the choice of the form.

\begin{thm}\label{t:DusaUniqueness}
    Cohomologous symplectic forms on a rational surface are diffeomorphic and hence the symplectomorphism type of a rational surface $(X,\w)$ is determined by $[\w]$.
\end{thm}

The positivity condition of $\w$ first appeared in the work of Friedman and Morgan \cite{FM88} as the K\"ahler cone of a specific class of integrable complex structures, which will be revisited later.  Positive rational surfaces were studied extensively in a series of works \cite{LZ15,Zha17}, where they are called $c_1$-\emph{nef}.  $c_1$-positive symplectic structures behave nicely in terms of their correspondence with pseudo-holomorphic curves, and hence many results therein have natural analogs from algebraic geometry.  In particular, Li-Zhang and Zhang showed that the $c_1$-positive sphere cone coincides with the $c_1$-positive symplectic cone for any almost complex structure $J$ compatible with some symplectic form \cite{LZ15,Zha17}.  This class of symplectic forms is also broad enough to include quite interesting examples of rational surfaces, namely, the following two classes.

\begin{enumerate}[(1)]
  \item $S^2\times S^2$ or $\CP^2\#k\ov\CP^2$ for $k\le9$ with an arbitrary symplectic form \cite{LL01}, \cite{Bir99pakstab}.  This includes all symplectic del Pezzo surfaces (i.e. monotone rational surfaces).
  \item A symplectic rational surface $X$ that admits symplectic divisor $D$, where $[D]$ is dual to the anti-canonical class.  This is the case when $(X,D)$ is a symplectic log Calabi-Yau pair.
In fact, as a direct consequence of a result of Dorfmeister-Li, which says $(X,\w)$ is a positive rational surface if and only if there is a symplectic divisor $D$, such that $(X,D)$ is a log Calabi-Yau pair.  For example, take an embedded elliptic curve $D'$ in $\CP^2$ as a symplectic divisor, and perform $n$ small size symplectic blow-ups on $D'$, we get a proper transform $D\subset X=\CP^2\#n\ov\CP^2$.  From \cite[Theorem 2.7]{DoL10} (see also \cite[Corollary 5.12]{EF21}), the symplectic cone of $(X,D)$ is given by $\{d\in H^2(X)| d^2>0, d\cdot[E]>0\text{ for all }E\text{ exceptional, and }d\cdot[D]>0\}$.  This is exactly the whole $c_1$-positive cone.
\end{enumerate}

Following \cite{LLW22} and \cite{ALLP}, we show in Section \ref{sec:symplectic_cone} that, for any symplectic form $\omega$ on a rational surface, the homology classes of Lagrangian $(-2)$-spheres form a root system $\Gamma(X,\w)$, called the \textbf{Lagrangian system}.  This is related to the stratification of the  symplectic cone.  Borrowing again from \cite{FM88}, one may define \textit{wall structures} by the sign of pairing against $-2$ classes, and the stratification can be defined by the corner structures of these walls.  We will tag these strata according to the Dynkin diagram of the root system formed by Lagrangian spheres, which was first suggested by Shevchishin in a preprint \cite{She10} assuming $\w$ is reduced.

Therefore, we divide the symplectic classes of a rational surface into type $\mathbb{A}$, $\mathbb{D}$ and $\mathbb{E}$ (see Definition \ref{d:typeForms} and a concrete numerical description of these forms).  Notice that a generic symplectic form is of type $\aA$, and type $\DD_k, \EE_k$ has codimension $k$ in the reduced cone (and hence symplectic cone).   Note that type $\aA$ forms is an open dense region of the reduced cone in Definition \ref{d:reduced}. Type $\DD_k$ forms are Cremona equivalent to  $(1| a,  \underbrace{\frac{1-a}{2}, \frac{1-a}{2}, \frac{1-a}{2}, \cdots \frac{1-a}{2}}_k, m_{k+2} \cdots)$ where $a<\frac13,$ and $\frac{1-a}{2}> m_{k+2}$ (see Definition \ref{d:slice} for the notation). Below is our main calculation for the symplectic Torelli group.

\begin{thm}\label{genSMC}
Given a symplectic rational surface $(X,\w)$ where $\w$ is $c_1$-positive, then

 \begin{itemize}
    \item  $\pi_0(Symp_h(X,\w))={\{1\}}$ if $\w$ is of type $\mathbb{A}$;
    \item $\pi_0(Symp_h(X,\w))=PB_k(S^2)$, the pure braid group of m strands on the sphere if $\w$ is of type $\DD_k$.  
  \end{itemize}

\end{thm}

 The main difficulty proving Theorem \ref{genSMC} has deep roots in the well-known problem of finding a good notion of genericity condition in $\mathbb{P}^2$, which entangles with the notorious Nagata conjecture.  Given $p_1,\cdots,p_r\in\CP^2$ at very general positions and $m_1,\cdots, m_r\in\ZZ^+$ for $r>9$, Nagata conjecture predicts that, any curve $C$ in $\CP^2$ which passes through each of the points $p_i$ with multiplicity $m_i$ satisfies 
 \[
   \deg C>\frac{1}{\sqrt{r}}\sum_{i=1}^r m_i.	
 \]
 
The connections between Nagata conjecture and symplectic geometry was established from a series of pioneering work, see for example \cite{MP94,Bir99pakstab,}, inspiring Biran's extension of Nagata conjecture to arbitrary polarized surfaces \cite{Bir01}.  Namely, given any irreducible curve $C$ in $\CP^2$ satisfying the assumptions of Nagata conjecture, one may first blowup all the $p_i$'s, and consider the proper transform $\wt C$ in a complex rational surface $\CP^2\#r\ov\CP^2$.  If the integrable complex structure obtained this way is compatible with a symplectic class $[\w_\epsilon]=(1|\frac{1}{\sqrt{r}}-\epsilon,\cdots,\frac{1}{\sqrt{r}}-\epsilon)$ for any $\epsilon$, the fact that $\w(\wt C)>0$ would imply Nagata conjecture.  Note that Biran proved $(1|\frac{1}{\sqrt{r}}-\epsilon,\cdots,\frac{1}{\sqrt{r}}-\epsilon)$ is indeed represented by a symplectic form in his famous paper \cite{Bir99pakstab}, so this problem mainly concerns the compatibility of this form with the complex structure obtained by blow-ups at a very general position.

To approach the topology of symplectomorphism groups on rational surfaces, McDuff made an important observation extending Kronheimer's fibration (see Section \ref{inflation}).  Her approach clarifies that, understanding the change of a symplectomorphism group under a symplectic deformation $\{\w_t\}_{t\in[0,1]}$ is equivalent to the study of the change of $\mA_{\w_t}$, which is the space of almost complex structures which are compatible with some symplectic structure isotopic to $\w_t$.  If an almost complex structure $J$ is compatible with a path of (non-cohomologous) forms $\{\w_t\}_{t\in [0,1]}$, this particular symplectic deformation can be regarded as a $J$-\textit{inflation}.  From this perspective, one would have proved Nagata conjecture, if they find a $J$-inflation for the integrable complex structure $J$ obtained by blowing up $r$ points of very general positions, whose ending symplectic class $[\w_1]$ is near Biran's full packing form.  This problem indicates that, it is a highly non-trivial task to understand the compatible cone for even a very general complex structure on a rational surface.  Therefore, understanding the change of $\mA_\w$ for blow-up numbers greater than $9$ under an arbitrary symplectic deformation seems out of reach for current technology.  Moreover, even finding a reasonable notion of ``genericity'' where the symplectic cone can be understood is a non-trivial problem.

Friedman and Morgan's notion of good generic complex structures provided an important insight to this problem.  A \textit{good generic surface} is one that admits a smooth elliptic curve $F$ representing the anti-canonical class, and the only irreducible curves with $C^2<0$ are exceptional divisors and $F$.  Such surfaces must be rational, and their K\"ahler cone is exactly the $c_1$-positive cone (\cite[Proposition 3.4]{FM88}), even though they are in general \textit{not} very general in Nagata's sense.

Our solution to the problem is largely inspired by this line of thoughts by focusing exclusively on exceptional curves, and considering an almost complex structure \textit{generic} if all exceptional classes have embedded representatives.  This is similar to the notion of \textit{generic} complex structures in Friedman-Morgan's sense, but we do not eliminate curves with self-intersections less than $-1$.  Extending this definition, we define a stratification which is \textit{coarse} in two senses: it only concerns exceptional curves, and it does not stratify almost complex structures with codimensions higher than $4$.

  We are able to prove that the compatible cone of generic almost complex structures in our sense contains the $c_1$-positive cone as a consequence of results from \cite{LZ15,Zha17}.  For almost complex structures in the second stratum, we are still able to obtain inflation for a large part of it, which allows us to obtain the invariance of the Torelli group when the type of $\w$ is fixed.

Although our method cannot obtain a complete description of symplectomorphism groups for type $\mathbb{E}$ forms, this is largely expected: Seidel explained in \cite{Sei08} that the fundamental group of the moduli space of marked cubic surfaces is in fact very large, by results of Allcock et al \cite{ACT98,Al00}; and for del Pezzo surfaces of degree 2, it is related to the moduli space of non-hyperelliptic curves of genus 3.  Investigations of symplectomorphism groups of this type seem to require completely different sorts of techniques.  Despite of that, our deformation technique does reduce them to the three base cases, i.e. the monotone blow-ups of 6, 7, 8 times on $\CP^2$ (see Lemma \ref{l:blowdown} and Remark \ref{rem:typeE}).  

Moreover, we are still able to prove the following theorem, which does not depend on the type of $\w$, answering Donaldson's question for positive rational surfaces.

\begin{thm}\label{t:generation}
   Given a positive rational surface $(X,\w)$, the symplectic Torelli group is generated by Lagrangian Dehn twists.
\end{thm}

As is pointed out in \cite{LW12}, Corollary \ref{c:LagSph} could be considered as a form of symplectic counterpart of a classical theory by M. Noether, which asserts that a plane Cremona map can be decomposed into a series of ordinary quadratic transformations.  See \cite{CremonaMapBook} for the history of this problem.


\subsection{Further applications and remarks} 
\label{sub:further_applications}

There is an immediate consequence of Theorem \ref{genSMC} regarding toric surfaces.

\begin{cor}[=Theorem \ref{toric}]
The symplectic Torelli group is trivial for any symplectic toric surfaces.
\end{cor}

The next corollary observes that Chiang and Kessler's example is indeed a positive rational surface of type $\mathbb{A}$.  From Theorem \ref{genSMC}, we conclude that

\begin{cor}
The Chiang-Kessler involution $\iota_{CK}: X_{CK}\to X_{CK}$ is Hamiltonian isotopic to identity, hence yields a Hamiltonian $\ZZ_2$-action which does not extend to a Hamiltonian $S^1$-action.
\end{cor}

In \cite[Corollary 1.2, Lemma 2.4]{BLW12}, Borman and the second and third authors proved that $Symp_h(X)$ acts transitively on the space of Lagrangian spheres ($-2$ symplectic spheres, resp.) with a fixed homology class for any symplectic rational surface $X$.  Theorem \ref{genSMC} and \ref{t:generation} imply that

\begin{cor}\label{c:LagSph}
    Given a positive rational surface $(X,\w)$, any two homologous Lagrangian spheres are Hamiltonian isotopic up to a sequence of Lagrangian Dehn twists.

	If $(X,\w)$ is of type $\mathbb{A}$, then any two homologous Lagrangian spheres are Hamiltonian isotopic.  In particular, they represent the same object in the Fukaya category.
\end{cor}

The results above lead to more questions.  Although Chiang-Kessler's involution $\iota_{CK}$ is non-extendable, their proof relies on very specific knowledge from Karshon's work of $S^1$ Hamiltonian four manifolds \cite{Ka99}, which does have an extension to higher dimensions at the moment.  Therefore, the following question is still difficult to answer, which seems to require a new type of Floer theoretic obstruction to extending a finite group action to an $S^1$-action.

\begin{qes}
   Given any symplectic manifold $M$, does $M\times X_{CK}$ admit a Hamiltonian involution that does not extend to a Hamiltonian $S^1$-action?
 \end{qes}

There is another connection between our results and a recent work by Hacking and Keating.  In \cite{HK20}, the mirror symmetry of a class of (complex) log Calabi-Yau surfaces $(Y,D)$ with a distinguished complex structure is considered.  Their mirror Weinstein manifold $W$ can be constructed as a Lefschetz fibration using the data of $(Y,D)$.  Suppose we also have a compactification such that $W=\eY\setminus \eD$, where $(\eY,\eD)$ is a log CY pair.\footnote{In fact, one may construct another log CY pair $(Y',D')$ from the dual cusp, such that $W$ is an open analytic subset of $Y'\setminus D'$.  But the symplectic form of $W$ cannot extend to a K\"ahler form of $Y'$ (see \cite[Remark 1.4]{HK20}).  It is also a conjecture of Hacking and Keating that such divisorial compactifications only exist when $(Y,D)$ is negative semi-definite, i.e. $D$ is a cycle of $-2$ spheres.}   Consider the space $\eS_\eD$ of symplectic configurations isotopic to $\eD$.  Assume that $\pi_0(Symp_c(W))$ is infinitely generated and $\pi_0(Symp_h(\eY))$ is finitely generated, the homotopy exact sequence
\[
   \wt\pi_1(\eS_\eD)\to \pi_0(Symp_c(W))\to\pi_0(Symp_h(\eY))\to\wt\pi_0(Symp_c(\eS_\eD))
 \]

 implies that $\eS_\eD$ has an infinite type topology.  Here, $\wt\pi_{i}(\eS_\eD)$ for $i=0,1$ is an extension of the fundamental group of $\eS_\eD$ by the extra generators given by automorphisms of a neighborhood of $\eD$, which is expected to be finitely generated.  This phenomenon regarding the topology of the space of symplectic curves seems new to the literature, and clearly contrasts sharply against linear systems.

When $D$ is negative semi-definite, it is known that $W=X\setminus E$, where $X$ is a del Pezzo surface, and $E$ is a smooth symplectic torus with anti-canonical class.  Moreover, there is an upcoming result by Hacking and Keating \footnote{Private communications.} that $Symp_h(W)$ is infinitely generated as long as there exists one $-2$ curve in $Y\setminus D$ (this last condition is easy to verify for any negative semi-definite case.  Compare also \cite[Section 7.2]{HK21} for the negative definite case).  Given this result, the above discussions apply to $(Y,D)$, in which $D$ is negative semi-definite and the length $4\le k\le9$.  The mirror of $(Y,D)$ is $(X,E)$, where $X=\CP^2\#(9-k)\ov\CP^2$ is known to have a finitely generated Torelli group.  Therefore, Hacking and Keating's observation implies the space of anticanonical symplectic torus has an infinitely generated fundamental group.  To extend this to $k=1,2,3$, we need the following extension of Theorem \ref{genSMC}.

\begin{qes}
   Is the symplectic Torelli group finitely generated for a type $\mathbb{E}$ positive rational surface $(X,\w)$?
\end{qes}

We do not know whether there are mirror Weinstein manifolds that compactify to rational surfaces of more than $8$ blow-ups (presumably these will be mirrors of length $k\ge9$ if exist), and there seems no obvious way to prove similar infinite generation results without mirror symmetry. 

\begin{qes}
   Given a rational surface $(X,\w)$ which is not monotone.  Does the space of symplectic tori of anti-canonical class have the topology of infinite type? 
\end{qes}

\subsection{Outline of the paper}

In Section \ref{sec:symplectic_cone}, we prove that any $c_1$-positive symplectic form is a positive linear combination of exceptional classes.  Section \ref{inflation} defines the coarse stratification and proves all the related inflation results.  In particular, we show that the Torelli group is invariant under two distinguished directions of symplectic deformations (Corollary \ref{01stab}), which will span the whole $c_1$-positive cone.

To prove Theorem \ref{genSMC} for type $\mathbb{A}$ forms, we use the so-called minimal inflation which shrinks the area of the smallest exceptional divisor.  This in turn proves the Torelli group is invariant even after blowing-down such a divisor, regardless of the type of $\w$ (when $\w$ is of type $\mathbb{D}$ or $\mathbb{E}$, we ask this minimal exceptional divisor to stay out of the relevant root system).  This concludes the case of type $\mathbb{A}$ forms and reduces type $\mathbb{D}$ or $\mathbb{E}$ to the case when the Lagrangian root system is irreducible.

The proof of type $\mathbb{D}$ case is more complicated.  To use the Hopfian property of the braid group, we divide the theorem into two surjectivity statements Lemma \ref{Dsurjbraid} and Proposition \ref{prp:surjective}.  The first statement is proved using a new symplectic skeleton by understanding the action of $Symp_h(X,\w)$ on its space, similar to \cite{Eva}, \cite{LLW22}.  The proof of Proposition \ref{prp:surjective} involves more new ingredients, and depends on an analysis of ampleness of certain divisors.  We construct a map $\beta: \pi_0(Symp_h(X,\w))\to \Conf_{n-1}(\CP^2)/\PSL(2,\CC)$ using the moduli space of certain pseudo-holomorphic curves, and prove that it precomposes another homomorphism $\alpha:\Conf_{n-1}(\CP^2)\to \pi_0(Symp_h(X,\w))$ to yield the canonical projection.  The map $\alpha$ is constructed using algebraic geometry.  

Lastly, to prove Theorem \ref{t:generation}, we consider the Torelli group as the cokernel of $\pi_1(\Diff(X))\to\pi_1(\mA_\w,\w)$.  Since we already know that this map is surjective when $\w$ is of type $\mathbb{A}$, for other types of symplectic classes, we perturb it to a type $\mathbb{A}$ form, and study the change of $\pi_1(\mA_\w,\w)$.  The deformation from $\w$ to a type $\mathbb{A}$ form glues in a series of codimension $2$ divisors in $\mA_\w$, and the extra generators all come from the meridians of these divisors.  We consider a specific construction of loops in almost complex structures, and prove that this loop is a meridian by considering a parametrized Gromov-Witten invariant using SFT.

The authors noted that Alekseeva \cite{Alekseeva21} announced an upcoming article, joint with Shevchishin, proving Theorem \ref{genSMC} for pure type $\mathbb{D}$ forms, as well as \ref{t:generation} for pure type $\mathbb{D}$ and $\mathbb{E}$ forms for a range of classes contained in the $c_1$-positive classes.  Their methods seem related but essentially different from ours.  It would be interesting to compare the two methods.

The last problem of parametrized Gromov-Witten invariant has very close relations to a circle of questions via family Seiberg-Witten invariants in \cite{SmirnovFlop,SmirnovK3,Linloop22,SmirnovTori22,SS17elliptic}.  The authors expect the parametrized Gromov-Witten invariant we used in this paper should correspond to the family Seiberg-Witten invariant used in \cite{SmirnovDehn} under a family version of Taubes's \textit{GW=SW} correspondence. \\

{\bf Acknowledgements:} J.Li is partially supported by AMS-Simons Travel Grant.  T-J. Li is supported by NSF Grant DMS1611680. W.Wu is supported by Simons Collaboration Grants for Mathematicians \#524427. The authors are particularly grateful to Ailsa Keating, who offered numerous suggestions and corrections to an earlier draft.  We would also like to thank Silvia Anjos, Paul Biran, Olguta Buse, Octav Cornea, Jonny Evans, Richard Hind, Nitu Kitchloo, Dusa McDuff, Martin Pinsonnault, Leonid Polterovich, and Gleb Smirnov for helpful conversations.

\section{The symplectic cone and Lagrangian root systems}\label{sec:symplectic_cone}
Recall that  the symplectic cone $\mC_X$ of an oriented, smooth manifold $X$ is the open cone 
of classes in $H^2(X;\RR)$ that are  represented by a symplectic form compatible with the orientation.
For a rational surface, a consequence of  Theorem \ref{t:DusaUniqueness} is that Diff$^+(X)$ acts on $\mC_X$. 
Notice that diffeomorphic symplectic forms have homeomorphic symplectomorphism groups. 
Thus we can restrict our attention to a fundamental region of $\mC_X$ under the action of Diff$^+(X)$. 
We describe such a region and its $c_1-$positive portion explicitly in this section.
This is done in two steps, considering first the subcone $\mC_{X, K}$ in 2.1 and then the reduced cone in 2.2. 
Finally, we introduce the Lagrangian root system $\mL_\w$ in 2.3. 

\subsection{The  $K-$symplectic cone, $W(X)$, $\Gamma_K$ and $P_K^+$}

Let $K\in H^2(X;\mathbb Z)$ be the symplectic canonical class of some orientation-compatible symplectic form on $X$ and define the $K$-symplectic cone as
\begin{equation}\label{def:K-symplectic cone}
\mC_{X,K}= \{e\in \mC_X~|~e=[\w] \hbox{ with }  K_{\w}=K\}.
\end{equation}
 Introduce the set of exceptional classes
\[\mE =\{E\in H_2(X,\ZZ)~|~E\cdot E=-1\text{~and~} E \text{~is represented by a smooth sphere} \},\]
the subset of $K$-exceptional spherical classes
$\mE_{K}= \{E\in \mE~|~E\cdot K=-1\}$,
and the subset of $\w$-exceptional spherical classes 
\[\mE_{\w}= \{E\in \mE~|~E\text{~is represented by an~}\w\text{-symplectic sphere}\}.\]
By the adjunction formula, $\mE_{\w}\subset \mE_{K_\w}$. It follows from Lemma 3.5 in~\cite{LL01} that  
$\mE_{\w}=\mE_{K_\w}$,
and  Theorem 4 of the same paper yields a characterization of the $K$-symplectic cone as
\begin{equation} \label{Kcone}
\mC_{X,K}= \left\{e\in \mathcal P_X~|~ e \cdot E > 0  \text{~for all~} E\in \mE_K\right\}.
\end{equation}
Since Diff$^+(X)$ acts transitively on the set of symplectic canonical classes, it suffices to focus to the subcone $\mC_{X, K}$ for one choice of $K$. 
 


Let $X$ be a rational surface not homeomorphic to $S^2\times S^2$. A basis $\{A_1,\ldots,A_{k+1}\}$ of $H_2(X;\ZZ)$ is said to be standard if i) each class $A_i$ is represented by a smoothly embedded sphere, and ii) the intersection pairing is represented, in this basis, by the diagonal matrix diag$(1,-1,\ldots,-1)$. Note that any identification $X\simeq X_k:= {\CC P^2}\# k\overline{\CC P^2}$ determines a standard homology basis $\{H, E_1, E_2, \cdots, E_{k}\}$. 
For the product $X_{S^2}=S^2\times S^2$, a standard basis consists of two classes $B$ and $F$ where $B$ and $F$ are both represented by a smoothly embedded sphere with self-intersection zero and $B\cdot F=1$. 
A rational surface $X$ together with a choice of a standard basis of $H_2(X,\ZZ)$ is called a \emph{framed surface}. 

A framed rational surface is equipped with a distinguished choice of $K$. In the case $X=X_n$, $K=-3H+E_1+\cdots+E_n$, and $K=-2B-2F$ for $S^2\times S^2$. 
From now on, in order to simplify the exposition, we will often implicitly assume that such a framing is chosen with the associated $K$.

We next give a brief account of the Cremona group for framed rational surfaces.  In the context of algebraic geometry, the relevant Kac-Moody algebra and generalized root system were first used by Looijenga \cite{Lo81} to the best of authors' knowledge, and the corresponding Coxeter groups are studied by Vinberg in a very general context \cite{Vin67,Vin84}, and this was used by Nikulin \cite{Ni87} in the study of $K3$ surfaces.  We will follow the exposition of \cite{Lo81,She10,GaoHZ} tailored to the symplectic context.

When $X=X_n$,  by taking a standard basis, we may identify $H_2(X,\RR)$ with the Lorentzian space $\RR^{1,n}$.  Consider the group $W_n=W(X)$ of automorphisms of 
$H_2(X, \RR)$ generated by reflections along the $-2$ classes $l_i, i=0, n-1$, where
\begin{equation}\label{e:Kroot}
     l_0=H-E_1-E_2-E_3, \quad l_1=E_1-E_2,\quad \cdots,\quad l_{n-1}= E_{n-1}-E_{n},
\end{equation}
 The group $W(X)$ is called the Cremona group of the framed rational surface $X$.

 In \cite[Proposition 4.7]{LW12}, the second and third author proved that $W_n$ is the homological action of diffeomorphisms that preserves the standard anti-canonical class $K$. 
 Let  $W_n'$ be the extension of $W_n$ by adding one more generator given by the reflection of $E_n$.
 It is proved in \cite{LL02,GaoHZ} that $W_n'$ is an index 2 subgroup in the homological actions of $\Diff(X)$ on $H_2(X)$.  
 
 From the general theory of Coxeter groups, $W(X)$ can be represented as the Weyl group associated to a Kac-Moody Lie algebra.  Its corresponding (generalized) root system consists of the orbit of $l_1=E_1-E_2$ (together with $l_0$ when $n=2$) under the $W(X)$-action.  Elements in this orbit are $-2$ classes $d\in H_2(X,\ZZ)$ which can be represented by an embedded $-2$ sphere and satisfies $d\cdot K=0$ from \cite{LL02}.   We note that  this set  is precisely the set $\mathcal S^{-2}_K$ in \cite{ALLP}. It is noted there that $\mathcal S^{-2}_K$ is an infinite set when $n\geq 9$. 
 
 We call this the $K$-\emph{root system} $\Gamma_{K}$, and the simple roots are given by $l_i$, $i\ge0$.  
Note that  $\Gamma_K$ is not a classical root system when $n\geq 9$. 

 \begin{figure}[tb]
   \centering
   \includegraphics[scale=0.8]{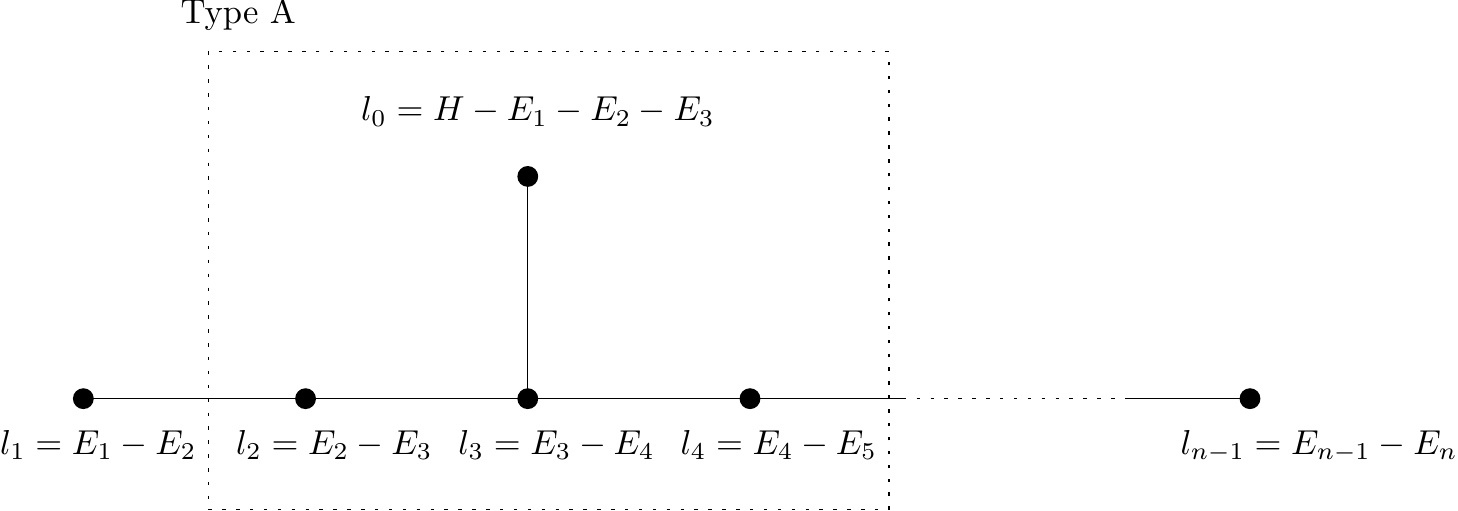}
   \caption{$K$-root system of $\CP^2\#n\ov\CP^2$ }
   \label{fig:Kroot}
 \end{figure}

In this paper, we are primarily interested in   $P_K^+$, the $c_1$-\textbf{positive} portion of  $\mC_{X, K}$, i.e. $d\in P_K^+$ if $d\in \mC_{X, K}$ and satisfies  $d\cdot (-K)>0$.  

\begin{thm}\label{pcomb}
 Let $X=X_n$ with $n\geq 1$. Any class $d\in P_K^+$ can be written as a finite $\RR_+$-combination of exceptional classes in $\mE_K$.  

\end{thm}

To prove this we recall the notion of a   reduced class.

\begin{dfn}[Reduced classes, cf.\cite{LL01,Gaozhaoruled}]\label{d:reduced}
Fix a standard basis of $X_k$. A nonzero class $\nu H-\sum_{i=1}^k m_i E_i\in H^2(X_k;\RR)$ is called \emph{reduced with respect to this basis} if
\begin{align*}
\nu&>0,\text{~for~}k=0;\\
\nu&\geq m_1\geq 0,\text{~for~}k=1;\\
\nu&\geq m_1+m_2\text{~and~}m_1\geq m_2\ge0, \text{~for~}k=2;\\
\nu&\geq m_1+m_2+m_3 \text{~and~} m_1\geq m_2 \geq \cdots \geq m_k\ge0, \text{~for~}k\geq3. 
\end{align*}
For the product $X_{S^2}=S^2\times S^2$ with the standard basis $\{B=[S^2\times pt], F=[pt \times S^2]\}$, a  class $bB+fF$ is called reduced if $b\geq f\geq 0$. 
\end{dfn}
  We need the following lemma which might be of independent interest.

 \begin{lma}\label{l:sphereComb}
 Let $X=X_n$ with $n\geq 1$.    If $\alpha\in H_2(X,\ZZ)$ has $\alpha^2>0$ and is represented by a symplectic embedded sphere, then $\alpha$ is Cremona equivalent to one of the following model classes

    \begin{itemize}
       \item $2H$,
       \item $kH-(k-1)E_1$, or
       \item $kH-(k-1)E_1-E_2$.
     \end{itemize}
 \end{lma}
 
\begin{proof}
   From \cite[Theorem 3]{Ki93}, \cite[Proposition 4.3]{LL02}, $\alpha$ is $W_n'$-equivalent to one of the model classes above.  All three model classes are reduced, and \cite[Theorem 2.1(2)]{GaoHZ} shows that such a representative of its $W_n'$-orbit is unique. 

   On the other hand, \cite[Lemma 3.7]{Lo81} showed that $\alpha$ is $W_n$-equivalent to a reduced class.  Clearly, this reduced class is also a reduced representative of its own $W_n'$-orbit.  Therefore, this reduced class must be one of the model classes by Zhao-Gao-Qiu's uniqueness.
\end{proof}

\begin{proof}[Proof of Theorem \ref{pcomb}]
 By \cite[Proposition 5.21]{LZ15}, a given $d\in P_K^+$ is a finite combination of positive square sphere classes $\sigma_i$.  It is clear that each model classes listed in Lemma \ref{l:sphereComb} can be represented by a positive linear combination of certain classes in $\mE_K$.  Therefore, if $\tau(\sigma_i)$ is a model class represented by a $\RR_+$-combination of exceptional classes $e_{ij}$ for some $\tau\in W_n$, then $\sigma$ can be represented by a $\RR_+$-combination by $\tau^{-1}(e_{ij})$.
 \end{proof}



\subsection{$V(X)$ and  $N\mR_n$}

\begin{dfn}\label{d:slice}
For a framed rational surface $X$, a symplectic form on $X$  is called reduced if its class is reduced. Such a class is called a reduced symplectic class. 

For $X=X_k$, the normalized reduced symplectic slice $V_k=V(X_k)\subset \mC_{X_k}$ is the subset of reduced symplectic classes with $\nu=1$. We represent such a class by the vector $(1\,|\,m_1, \cdots, m_k)$ or simply by $(m_1, \cdots, m_k)$.

For $X=X_{S^2\times S^2}$, the normalized reduced symplectic slice $V_{S^2\times S^2}=V(X_{S^2\times S^2})\subset \mC_{X_{S^2\times S^2}}$ is the subspace of reduced symplectic classes with $f=1$.
\end{dfn}

To describe $V(X_n)$ explicitly, consider the closed polytope $\ov\mR_n$ in $H_2(X;\RR)$ with the following $n+1$ vertices:
$$M=-\frac{1}{3}K,\hskip 2mm O=(1|0,\cdots,0),\hskip 2mm
 A=(1|1,0,\cdots,0),\hskip 2mm
G_3=(1|\frac{1}{2}, \frac{1}{2}, 0,\cdots,0),$$
$$G_4=(1|\frac{1}{3}, \frac{1}{3}, \frac{1}{3}, 0,\cdots,0),
  \cdots ,
G_n=(1|\frac{1}{3},\cdots, \frac{1}{3}, 0).
$$

\begin{lma} For $X=X_n$,  a class $\sigma$ is normalized and reduced if and only if $\sigma\in\ov\mR_n$
\end{lma}

This is independent of $n$, and was checked in 
\cite{LL16}).  Let $\mR_n:=\ov\mR_n\setminus\{m_n=0\}$.

\begin{rmk}\label{-kn}
	Notice that, for $i\geq 4$,  $G_i$ can be identified with $\frac13(-K)$ on $X_{i-1}$.  
	By abusing notation, we will occasionally use $G_1, G_2, G_{n+1}$ for $O, A, M$, eg. in Lemma \ref{convexNRn}.
\end{rmk}

We now summarize the properties of the normalized reduced symplectic slice $V(X)$ in the next proposition (cf. Proposition 2.11 in \cite{ALLP}). 
\begin{prp} \label{redtran}
Let $X$ be a framed rational surface with the associated class $K$.  Then
\begin{enumerate}
\item  $V(X)$ is a  convex subset of $\mC_{X, K}$.  
 For $X=X_n$, $V(X)=\mR_n\cap \mC_{X, K}$. 
\item 
A  reduced class is symplectic if and only if it has a positive square. So for $X=X_n$,  $V(X)=\{v\in \mR_n|v\cdot v>0\}$. 
\item  Every class in $\mC_X$ is equivalent to a unique normalized reduced symplectic class under the action of { $\Diff^+(X)\times \RR^+$}, where $\RR^+$ acts by rescaling. In other words,  $V(X)$ is a fundamental domain of $\mC_X$ under the action of { $\Diff^+(X)\times \RR^+$}. 
\item  $V(X)$ is also a fundamental domain of $\mC_{X, K}$ under the action of   $W(X)\times \RR^+$ (see \cite{MS12}).  
\end{enumerate}
\end{prp}

The conditions of $c_1$-positivity and reducedness are well-compatible.  Namely, the \textbf{$c_1$-positive reduced cone} 
still forms a fundamental domain of the $c_1$-positive cone under Cremona transforms.  

\begin{dfn}\label{d:c1positivecone}
Let    $N\mR_n:=P_K^+\cap\mR_n$ be the set of \textbf{normalized} $c_1$-\textbf{positive reduced symplectic classes}.

       $S\mR_n$ is the $c_1$-\textbf{positive reduced symplectic cone}, consisting of classes $d\in P_K^+$ which are reduced.
   Here, $S\mR_n$ is a fundamental chamber of $P_K^+$ under the $W_n$ action.   
\end{dfn}

$P_K^+$ has  a simplicial structure (\cite{FM88} and \cite[Lemma 5.24]{LZ15}). It follows  that   $N\mR_n$ inherits a simplicial structure.
Here we directly describe the  simplicial structure of  $N\mR_n$.

\begin{lma}\label{convexNRn}
If $n\le 8$, $N\mR_n=\mR_n$ is the convex polytope generated by the  $n+1$ vertices $G_i, 1\leq i\leq n+1$, only containing the vertex $G_{n+1}$.

$N\mR_9$ is $\mR_9$  is the convex polytope generated by the  $10$ vertices $G_1, ..., G_{10}$, not containing any of the vertices.  	

	When $n\ge 10$,	$N\mR_n$ is a convex polytope with    $10(n-8)$ vertices, not containing any of the vertices.  Those vertices are given by all the vertices of $N\mR_{n-1}$ with 10  intersection points of the 10 edges of $G_{n+1}G_i, 1\leq i\leq 9$ and the hyperplane of $-K$ as the new vertices. 

\end{lma}

\begin{proof} 
We first show that $$ \mR_n \cap \{d|  d\cdot(-K)>0 \} \subset \{d|d^2>0\}.$$     To see this, we first check all the vertices of $ \mR_n $.  It is easy to see that   
	
(i)	$ (-K)\cdot G_i>0$ and  $G_i^2>0$ for $1\leq i\leq 9$,
	
(ii)	 $ (-K)\cdot G_i\le 0$ and  $G_i^2\le0 $  for $i\geq 10$. 
	 
	  Since  $ N\mR_n \cap \{d|  (-K)\cdot d>0 \}$ is convex,  $ N\mR_n \cap \{d|  (-K)\cdot d>0 \} \subset \{d|d^2>0\}$ by lightcone lemma.  

What we have shown is that, for all but the  new facet  given by $d\cdot(-K)=0$,  each face of $N\mR_n$ is the restriction of a facet of $\mR_n$ with the condition   $d\cdot(-K)  \ge 0$.  The new face given by $d\cdot (-K)=0$ automatically satisfies $d^2\geq 0$, and hence the non-linear boundary $d^2=0$ is outside of the $c_1$-positive symplectic cone.

Then we give the explicit description of $N\mR_n$.  Note that when $n\le 8$,  $N\mR_n$ is the same as $\mR_n$ and hence the conclusion holds.  Then we do induction:  1) for $n=9$, the hyperplane $\{d|  d\cdot (-K)=0 \} $ intersects with $\mR_n$ only at the vertex $G_{10}$, and hence the conclusion holds for $N\mR_9$. 

2) For $n=10$,   $N\mR_9$ is now a facet of $N\mR_{10}$, hence $N\mR_{10}$ has those 10 vertices.  There are 10 more vertices: the hyperplane $\{d|  d\cdot (-K)=0 \} $ intersects with $\mR_{10}$  on the edges $G_{11}G_i, 1\le i\le 10$. 


3)  For  $n>10$, if $N\mR_{n-1}$ is the convex hull as described, then  $N\mR_{n}$ is the convex hull of  $N\mR_{n-1}$  with the intersection points of  the hyperplane $\{d|  d\cdot (-K)=0 \} $  with the 10 edges $G_{n+1}G_i, 1\le i\le 10$. 
\end{proof}

The general formula, using the notation in Remark \ref{-kn}, for the $n$ new vertices for $n \geq 10$ is as follows:

On the $G_{n+1}O$ edge with $O= ( 1|\underbrace{0,0,\cdots, 0}_{n})$:\hskip 5mm  $(1|\underbrace{\frac{3}{n}, \cdots, \frac{3}{n}}_{n})$.

On the $G_{n+1}A$ edge with $A= (1|1,  \underbrace{0, \cdots, 0}_{n-1})$:\hskip 5mm $(1| \frac{n-7}{n-3}, \underbrace{ \frac{2}{n-3},\cdots \frac{2}{n-3}}_{n-1} )$.

On the $G_{n+1}G_3$ edge with $G_3=(1|\frac{1}{2}, \frac{1}{2}, \underbrace{ 0, \cdots, 0}_{n-2})$:\hskip 5mm $(1| \frac{n-5}{ 2(n-3)},  \frac{n-5}{2(n-3)},  \underbrace{\frac{2}{n-3}, \cdots, \frac{2}{n-3}}_{n-2} )$.



 On the $G_{n+1}G_i$ edge,   $4 \le i  \le 10$:\hskip 5mm   $(1|\underbrace{\frac{1}{3}, \cdots, \frac{1}{3}}_{i},  \underbrace{\frac{9-i}{3(n-i)}, \cdots, \cdots,  \frac{9-i}{3(n-i)}}_{n-i} )$.



 \subsection{Lagrangian root system $\mL_{\w}$}
\begin{dfn}
For a symplectic 4-manifold $(X, \w)$, let $\mL_\w$ be the set of the homology classes of Lagrangian $(-2)$-spheres.
\end{dfn}

 We showed that in Lemma 2.23  in \cite{LL16} that $\mL_\w$ is a root system as long as it is a finite set.

\begin{thm} \label{Lag root system}
For a symplectic rational surface $(X, \w)$, $\mL_\w$ forms a (classical) root system, which we call the \textbf{Lagrangian root system} of $(X, \w)$. 
\end{thm}

By Lemma \ref{redtran} (3) it suffices to prove this for a reduced symplectic form $\w$. Therefore Theorem \ref{Lag root system} follows from the following lemma, especially part (iv). 

\begin{lma}\label{l:-2class}
   \begin{enumerate}[(i)]
     \item For a reduced symplectic form $\w$, any element $D\in \Gamma_{K}$  satisfying $H\cdot D\ge0$ is a positive integral linear combination of the simple roots $l_i$ (this coincides with the definition of positive roots in \cite{Lo81}).  In particular, $\w(D)\ge0$ for such classes.
     \item Each $D$ with $H\cdot D\ge0$ can either be represented by a Lagrangian sphere or a symplectic sphere.  We call $D$ a \textbf{positive root} of $\Gamma_{K}$ if $D\cdot H>0$.  

    \item The homology classes of Lagrangian $(-2)$-spheres form a root subsystem $\mL_\w\subset\Gamma_{K}$ in the $K$-root system, called the \textbf{Lagrangian root system}.  

          \item The Lagrangian root system $\mL_\w$ is a direct sum of ADE type root system. (\cite[Theorem 4'(iii)]{She10})
          
        \end{enumerate}
\end{lma}

\begin{proof}
    If we write $[D]=aH-\sum_i b_iE_i$ for $a>0$ $b_i\ge b_{i+1}\ge0$, it was shown in \cite[Proposition 4.10]{LW12} that $a<b_1+b_2+b_3$ always holds, and there is a Cremona reflection

  \begin{multline}\label{e:reduction}
    \Gamma_{123}([D])=[D]+([D]\cdot(H-E_1-E_2-E_3))(H-E_1-E_2-E_3)\\=(2a-b_1-b_2-b_3)H-\sum c_iE_i,
  \end{multline}

 such that $0\le 2a-b_1-b_2-b_3<a$ unless $[D]=H-E_1-E_2-E_3$.  Repeating this process, one may reduce $D$ to $H-E_\alpha-E_\beta-E_\gamma$ for some $\alpha,\beta,\gamma$, which is a positive integral linear combination of the simple roots.  Since $D\cdot (H-E_1-E_2-E_3)=a-b_1-b_2-b_3<0$, \eqref{e:reduction} splits off a positive component of $H-E_1-E_2-E_3$ from $D$ in each step of the reduction. 

 If $a=0$, it is easy to check that $D=E_i-E_j$, hence a positive integral combination of simple roots. This implies (i).

 Since all $l_i$ pair $\w$ non-negatively from the reducedness of $\w$, $\w\cdot D\ge0$ from (i).  Classes with positive $\w$-areas are represented by $-2$ symplectic spheres by \cite[Proposition 5.16]{DLW18}; and classes with zero $\w$-area are represented by a Lagrangian sphere by \cite[Theorem 1.4]{LW12}, concluding (ii).

  To see (iii), suppose $[D]$ is represented by a Lagrangian sphere, then $[D]\in\Gamma_{K}$.  All simple roots involved in the reduction process \eqref{e:reduction} must have zero area, and can be represented by Lagrangian spheres by (ii).  This implies $\mL_\w$ is a root subsystem of $\Gamma_{K_0}$, and its simple roots forms a subset of simple roots $\{l_i\}_{i\ge0}$ of the $K$-root system.  Therefore, if it is not a direct sum of ADE type, its Coxeter graph must contain that of the affine root system $\wt E_8$, i.e. $l_i\in\mL_\w$ for $0\le i\le8$.   This implies $\w(E_i)=\frac{1}{3}\w(H)$ for $1\le i\le 9$, which makes $[\w]^2\le0$, a contradiction.  (iv) therefore follows.

\end{proof}

\begin{dfn}\label{d:typeForms}
    We say a  symplectic form $\w$ is of \textbf{type} $\mathbb{D}$ or $\mathbb{E}$ if $\mL_\w$ contains a type $\mathbb{D}$ or $\mathbb{E}$ factor, respectively.   Otherwise,  $\w$ is of \textbf{type $\aA$}.
\end{dfn}

    We also denote $\Gamma^+(X,\w)\subset\Gamma_{K}$ as the homology classes which are represented by embedded $-2$ symplectic spheres (see Lemma \ref{l:-2class}).


We  have a concrete description  for a normalized reduced symplectic class:
\begin{itemize}
 \item $[\w]$ is of type $\EE_6,\EE_7,\EE_8$ (or simply type $\EE$), if  
 	  $$[\w]=(1|   \underbrace{ \frac13, \frac13, \cdots \frac13}_k, m_{k+1} \cdots),$$ where $  m_{k+1}<\frac13,    k =6,7,8,$ respectively.  These classes forms a codimension $k$ in the reduced cone.
 
 \item $[\w]$ is of type $\DD_k$ (or simply type $\DD$), if $$[\w]=(1| a,  \underbrace{\frac{1-a}{2}, \frac{1-a}{2}, \frac{1-a}{2}, \cdots \frac{1-a}{2}}_k, m_{k+2} \cdots),$$ where $ \frac{1-a}{2}> m_{k+2}$, and either  $\frac13 < a <1 $  and $ k\geqslant 4$; or $a=\frac13$ and  $k=4$.  These classes form a codimension $k$ subset in the reduced cone.
 
  \item  $[\w]$ is of type $\aA$ for all other possibilities. Equivalently, $\w$ is of type $\mathbb{A}$ if and only if at least one of $l_i$ has positive $\w$-area for $i=0,2,3,4$ (See Figure 1). 
\end{itemize}


\begin{lma}\label{conecv}
  A reduced symplectic class of type $\mathbb{D}_{n-1}$ of $\CP^2\#n\ov\CP^2$ can be written as  
  $$[\w_a]:=(1| a,  \underbrace{\frac{1-a}{2}, \frac{1-a}{2}, \frac{1-a}{2}, \cdots \frac{1-a}{2}}_{n-1}),\hskip 3mm \frac13< a<1.$$ 

  When $n\le 9$, $[\w_a]\in S\mR_n$.  When $n>9$, $[\w_a]\in S\mR_n$ if and only if $a>\frac{n-7}{n-3}$.  
\end{lma}

\begin{proof}





	Note that this follows from the proof of Lemma \ref{convexNRn}, and it is a special case of the intersection point on the edge $-K_{n}, A$ being $( \frac{n-7}{n-3},  \frac{2}{n-3}, ...,   \frac{2}{n-3} )$.


\end{proof}

\section{Inflation and stability for $Symp_h(X,\w)$}\label{inflation}

\subsection{Kronheimer-McDuff's sequence and $J$-holomorphic curves} 
\label{sub:kroheimer_mcduff_s_sequence_and_}

 Let $\mS_\w$ denote the space of all symplectic forms that are isotopic to a fixed symplectic form $\w$, and $\mJ_\w$ denote the space of almost complex structures that are compatible with $\w$.  Let
$$\wt\mA_\w:=\{(\w',J): \w'\in\mS_\w, J\in\mJ_{\w'}\}$$
 and
 $$\mA_\w:=\{J: J\in\mJ_{\w'}, \w' \text{ is isotopic to }\w\}.$$

 The following observation is due to McDuff.

\begin{lma}\label{l:projection}
    The projections $\pi_\w: \wt\mA_\w\to \mS_\w$ and $\pi_J: \wt\mA_\w\to \mA_\w$ are both homotopy equivalences.
\end{lma}

Consider Kronheimer's fibration
\begin{equation}\label{e:Kro}
    Symp_0(X,\omega)\to \Diff_0(X)\to \mS_\w,
\end{equation}

where $\Diff_0(X)$ is the identity component of the diffeomorphism group, and $Symp_0(X_n,\w)$ is the group of symplectomorphisms which are smoothly isotopic to identity.  From \cite[Theorem 0]{LLW22}, $Symp_0(X,\w)\cong Symp_h(X,\w)$.  Since $\Diff_0(X)$ doesn't depend on $\w$, the change of homotopy groups of $Symp_h(X_k,\w)$ under a symplectic deformation is controlled by $\mS_\w$, which is homotopic to $\mA_\w$.  Combining Lemma \ref{l:projection} with the original Kronheimer fibration, we have the following form of the Krohemier-McDuff fibration, well defined up to homotopy
\begin{equation} \label{homotopy fibration}
Symp_h(X,\w)\to \Diff_0(X) \to \mS_{\w}\sim \mA_{\w}.
 \end{equation}

Recall from section 4 of \cite{ALLP} the following version of inflation lemma:
\begin{thm}\label{t:inflation}
   Given a compatible pair  $(J,\w)$ which admits a $J$-holomorphic curve $Z$, there exist a family of symplectic forms $\w_t$ compatible with  $J$ such that $[\w_t]= [\w]+ t PD(Z), t\in [0,\lambda)$ where $\lambda= \infty$ if $Z\cdot Z\ge0$ and $\lambda= \frac{\w(Z)}{(-Z\cdot Z)}$ if $Z\cdot Z<0$.

\end{thm}

Readers might be concerned that Theorem \ref{t:inflation} only stated the existence of a symplectic form $\w_t$ which is compatible with $J$ and has the correct cohomology class for each $t>0$. A priori, the theorem cannot guarantee smooth (or even continuous) dependence since we applied the theorem in \cite{LZ15}.  However, the situation can be easily improved as follows.

\begin{cor}\label{cor:smoothALLP}
    $\{\w_t\}_{t\in[0,\lambda)}$ in Theorem \ref{t:inflation} can be chosen smoothly.
\end{cor}

\begin{proof}

When $\lambda=\infty$, one may form a family $\{\wt\w_t\}_{t\in[N-1,N]}$ by a linear interpolation between $\w_{N-1}$ and $\w_N$, since the $J$-compatible cone is convex.  Concatenating these families for each $N\in\mathbb{N}$ and smoothing at the corners, one indeed obtains a smooth family of $J$-compatible forms $\{\w_t\}_{t\ge0}$.  One may use this interpolation method similarly when $Z\cdot Z<0$.

\end{proof}

\begin{rmk}\label{rem:J}
      Corollary \ref{cor:smoothALLP} is still slightly weaker than the traditional inflation method, which yields a family $\w+t\eta$ for some $\eta\in\Omega^2(M)$ supported near $Z$.  It is clear that one could use the above interpolation method between $\w_0$ and $\w_N$ to obtain some two-form $\eta_N$ in this flavor, but we do not know whether one could obtain a good control so that $\eta_N$ would converge in any regularity.
  Thankfully, the above family seems to suffice for most applications at hand.
\end{rmk}

\begin{lma}\label{rem:AtoS}
   \begin{enumerate}
     \item 
     Consider 
     \begin{multline}
      \mA^{t}_\w:=\{J: J\text{ is tamed by some }\w'\in\mS_\w,\\
      \text{ and } J\in\mJ_{\w''}\text{ for some symplectic form }\w''\}.\end{multline} 

      Then $\mA^{t}_\w=\mA_\w$.  

     \item 
     The isomorphism
     \begin{equation}\label{e:pin}
          \rho_n: \pi_n(\mA_\w)\xrightarrow{\sim}\pi_n(\mS_\w)
     \end{equation}
     can be given by the following explicit description.  Assume that $\alpha: (S^n,*)\to (\mA_\w,*)$ and $\beta: (S^n,*)\to (\mS_\w,*)$ satisfies $\alpha(x)$ tames $\beta(x)$, then $\rho_n([\alpha])=[\beta]$.  
   \end{enumerate}

\end{lma}

\begin{proof}
\begin{enumerate}
  \item 
      Note that $\w''$ doesn't need to be related to $\w$ or $\w'$ in any way.  It is obvious that $\mA_\w\subset\mA^{t}_\w$.  For $J\in\mA^{t}_\w$, we know $J\in\mA_{\w'}$ from \cite[Corollary 1.4]{LZ09} for some $\w'$ \emph{cohomologous} to $\w$.  $\w$ and $\w'$ then are linearly isotopic because they are in the same tame cone, so $J\in\mA_\w$ by definition.

  \item We first reduce the statement to the case when $\alpha(x)$ is compatible with $\beta(x)$.  Take a based homotopy equivalence $\gamma:\mA_\w\to\wt\mA_\w$ which is realized by a section of $\pi_J:\wt\mA_\w\to \mA_\w$, then $\pi_J\circ\gamma\circ\alpha(x)$ is a symplectic form isotopic to $\w$ and compatible with $\alpha(x)$.  Since $\pi_J\circ\gamma\circ\alpha(x)$ and $\beta(x)$ are both inside the tame cone of $\alpha(x)$, there is a linear interpolation which isotopes $\beta$ to $\pi_J\circ\gamma\circ\alpha$ as based maps from $S^n$.

  Therefore, we assume without loss of generality that $\alpha(x)$ is compatible with $\beta(x)$ for all $x\in S^n$.  Note that $\rho_n$ factorizes through $\wt\mA_\w$.  The map $\wt\alpha: (S^n, *)\to (\wt\mA_\w,*)$ given by $\wt\alpha(x):=(\alpha(x),\beta(x))$ is clearly a lift over $\alpha$ of the Serre fibration $\pi_J$.  Therefore, $\gamma\circ\alpha$ is homotopic to $\wt\alpha$ since the fiber of $\pi_J$ over each $\alpha(x)$ is contractible.  This implies that under the composition of homotopy equivalences $\mA_\w\xrightarrow{\gamma}\wt\mA_\w\xrightarrow{\pi_\w}\mS_\w$, the homotopy classes are mapped via $[\alpha]\mapsto[\wt\alpha]\mapsto[\beta]$, giving $\rho_n$.
\end{enumerate}

\end{proof}

\begin{rmk}\label{rem:}
    Below, we will sometimes use $\mA_\w$ as the equivalent but more flexible tame counterpart above (especially in Section \ref{sec:kroheimer_mcduff_fibration_ball_swappings_and_dehn_twists}).  A recent result by Tan-Wang-Zhou-Zhu \cite{TWZZ} asserts that $\mA^{t}$ can also be defined simply as the almost complex structure tamed by a symplectic form (not assuming the existence of compatible symplectic form $\w''$), and the above discussion remains valid.

\end{rmk}

Below, we collect some existence results of $J$-holomorphic curves for $J\in\mA_\w$ to prepare us for the inflation procedures later.

\begin{lma}[\cite{Pin08}, Lemma 1.2]\label{minemb}
  Let $H, E_1,\cdots, E_n$ be the standard basis of $H_2(X,\ZZ)$, where the $[\w]=(1|m_1,\cdots,m_n)$.
The class $E_n$ has the smallest area among all exceptional sphere classes in $X$, and hence has an embedded $J$-holomorphic representative for any $J\in \mJ_{\w}$. 
\end{lma}

\begin{lma}[Lemma 4.1 in \cite{Zha17}, Lemma 3.2 in \cite{Chen20}]\label{a<0}Let $M=\mathbb CP^2\#k\overline{\mathbb CP^2}$.
If $S=aH+\sum b_iE_i$ with 
$a<0$ is represented by an irreducible curve, then
\begin{itemize}

\item $S=-nH+(n+1)E_1-\sum_{k_j\neq 1}E_{k_j}$ up to a Cremona transform.

\item Or $S=f^*S'$, where $f$ is a Cremona transformation and $S'$ is a class with $a'>0$.
\end{itemize}

In particular, if the form is reduced, and $S$ is embedded with $S^2<0$, then $S=-nH+(n+1)E_1-\sum_{k_j\neq 1}E_{k_j}$ (without doing Cremona transforms) and it has to be a spherical class.

\end{lma}

Recall that a class $A$ is called $J$-\textbf{nef} if $A\cdot [\Sigma]\ge0$ for any irreducible $J$-curve $\Sigma$.  We have the following Nakai-Moishezon type result.

\begin{lma}\label{extray}
$J$-nef spherical classes have non-empty irreducible moduli, and hence always have embedded representatives.

 In particular, for any rational surface with $\chi>4$, $\w\in (1|m_1,\cdots, m_n)\in N\mR_n$,
  then for any $J$ taming $\w$, $H-E_1$ has an $J$-embedded representative.

  \end{lma}

\begin{proof}
 The first statement is   Proposition 6 in \cite{LZ15}.

 From Proposition 3.5 in \cite{Zha17}, for any irreducible $J$-curve $\Sigma$ with $[\Sigma]=aH-\sum_i b_iE_i$ where $a\ge0$, we either have $a>b_i$, or $[\Sigma]=E_i-\sum E_j$.  Lemma \ref{a<0} listed all irreducible curve classes for $a<0$.  $H-E_1$ pairs non-negative with all classes listed above, hence is $J$-nef for any $J$ taming $\w$.
\end{proof}

Lastly, we need to consider the existence of embedded $J$-curves when $J$ admits $(-2)$ spheres.  The following statement is a special case of \cite[Theorem 1.2.7]{MO15}.

\begin{thm}\label{thm:1} Let $(M,\w)$ be a symplectic $4$-manifold with an embedded symplectic sphere $\Sigma$, where $[\Sigma]^2=-2$.  Suppose $A\in H_2(M;\ZZ)$ satisfies:

\begin{enumerate}[1)]
  \item $Gr(A)\ne 0$;
  \item $A\cdot E\ge 0$ for every $E\in \mE-  \{A\}$, where $\mE$ is the set of exceptional classes;
  \item $A\cdot S_i\ge 0$ for  $1\le i\le s$.
\end{enumerate}

Let $\mJ(\Sigma)\subset \mJ_\w$ be the subset of almost complex structures such that $\Sigma$ is $J$-holomorphic when $J\in\mJ(\Sigma)$. Then there is a residual subset $\mJ_{emb}(\Sigma,A)\subset\mJ(\Sigma)$
such that  $A$ is represented by  an embedded $J$-holomorphic curve.  $\mJ_{emb}(\Sigma,A)$ contains an open dense subset of $\mJ(\Sigma)$, and is itself an open dense subset if $c_1(A)>0$.

Moreover, any two elements $J_0, J_1\in \mJ_{emb}(\Sigma, A)$ can be joined by a path
$J_t, t\in [0,1],$  in
$\mJ(\Sigma)$  for which there is a smooth family
 of embedded $J_t$-holomorphic $A$-curves.

\end{thm}

\begin{proof}
  It is clear from our assumption that $A$ is $\Sigma$-good in the sense of McDuff and Opshtein.  $A$ clearly satisfies the assumption (v) in \cite[Theorem 1.2.7]{MO15}: if $A=l[\Sigma]$, $c_1=0$ and $A^2<0$, whose Gromov-Witten dimension is clearly negative, contradicting the fact that $Gr(A)\neq0$.

  The only claim that requires caution is the open and dense property of $\mJ_{emb}(\Sigma,A)$.  The subset in $\mJ_{emb}(\Sigma,A)$ which makes $A$-curve regular is residual from \cite[Proposition 3.3.3]{MO15})  (the proof of this case involves difficult virtual dimension arguments but not by-hand constructions of $J$-curves).  Each regular $J$ has an open neighborhood where the embedded $A$-curve deforms smoothly hence yielding embedded curves as well.  The union of such neighborhoods clearly yields an open dense subset in $\mJ_{emb}(\Sigma,A)$.  If $c_1(A)>0$, then every embedded $A$-curve is transversal by automatic transversality \cite{HLS95}.  This implies $\mJ_{emb}(\Sigma,A)$ itself is open and dense.
\end{proof}

\subsection{A coarse stratification of $\mA_\w$}\label{s:stratification}

Recall the following result in
\cite{AP13} Appendix B.1 (cf.  \cite{LL16} Proposition 2.6), for spherical classes $\sigma$ such that $\sigma^2\le-1$.

\begin{prp}\label{stratum}
Let $(X,\omega)$ be a symplectic 4-manifold.
 Suppose $\mC$ is a finite subset of $H_2(X,\ZZ)$.  Denote $\mJ_{\mC}\subset\mJ_{\omega}$ is a subset
characterized by the existence of a configuration of somewhere injective
$J$-holomorphic rational curves $C_{1}\cup C_{2}\cup\cdots\cup C_{N}$ of negative self-intersections, with $\mC=\{ [C_{1}], [C_{2}],\cdots , [C_{N}]\}$.
Then $\mJ_{\mC}$ is a co-oriented Fr\'echet submanifold
of $\mJ_{\omega}$ of (real) codimension $2N-2c_{1}([C_{1}]+\cdots+ [C_{N}])$. 

The same conclusions hold for $\mJ_\mC^t\subset\mJ_\w^t$, where the almost complex structures are only tamed by $\w$.
\end{prp}

In particular, for any symplectic 4-manifold  $(X,\omega)$, if $J$ admits a configuration as above, we say \textbf{$J$ is of codimension $2N-2c_1([C_1]+\cdots+[C_N])$}.  

Note that Anjos and Pinsonnault stated the result for embedded spheres, but their proof works as long as each $C_i$ is somewhere injective.  The Fr\`echet submanifolds $\mJ_\mC$ are usually \textit{not} closed, but they are relatively closed as a general consequence of Fredholm theory.  This means for any point $p\in\mJ_\mC$, there is an open neighborhood $V_p\subset\mJ_\w$, such that $\mJ_\mC\cap V_p$ is a closed submanifold of $V_p$, which is also a chart of $\mJ_\mC$.

 In the special case when $E\in\mE_K$, $\mJ_E$ is an open dense subset in $\mJ_\w$.  And if $D$ is a single embedded (-2) symplectic sphere class, $\mJ_D$ has codimension $2$ in $\mJ_\w$.

\begin{lma}\label{l:Edecomposition}
  \begin{enumerate}[(1)]
    \item Assume for a given $J\in\mJ_\w$, $E$ has a $J$-stable curve representation where $E=\sum r_i C_i, r_i>0$ and $C_i$ are irreducible components.  Then at least one of these components have $C_i^2<0$.
    \item If a class $A^2\le-1$ is represented as a $J$-rational curve $C_A$ which is not embedded.  Then its virtual dimension $vd(A)\le-4$.  The equality holds if and only if $C_A$ is immersed with a single transversal self-intersection, or has a single cusp point whose Milnor number equals $1$.
    \item Suppose a class $A^2\ge0$ is represented by an irreducible rational curve $C_A$ for some $J\in\mJ_\w$, or $A$ is an exceptional class.  Then $A\cdot E\ge0$ for any exceptional class $E\neq A$. 
  \end{enumerate}

\end{lma}

\begin{proof}

\begin{enumerate}[(1)]
  \item We assume the contrary that $C_i^2\ge0$ for all $i$.   This leads to a contradiction: $-1=E\cdot E = (\sum r_i C_i )\cdot (\sum r_i C_i )= \sum r^2_i C^2_i+ \sum r_ir_j C_i\cdot C_j  \ge0. $
  \item From the adjunction formula \cite[Theorem 1.6.2]{IS99}, $c_1(A)\le A^2+2-2\delta(u)$ for a rational curve, where $\delta$ is the sum of the number of immersed double points and the Milnor number of all cusp points.  Therefore, $vd(A)=2c_1(A)+4-6\le -4\delta(u)\le -4$. \footnote{Note that this is the virtual dimension of the moduli space of all rational curves of class $A$ without restricting the type of singularities that could possibly appear.} 

  \item When $A^2\ge0$, since $E$ has a non-vanishing Gromov-Witten invariant, it has a connected $J$ stable representative.  Then every component of $E$ intersects $C_A$ non-negatively: it's either $C_A$ itself, or by positivity of intersections if it's not $C_A$.

  When $A$ is exceptional, this is known to \cite[(ii)]{MS12}.
\end{enumerate}

 \end{proof}

Take an exceptional class $E$ and $\w$, and denote $\mD(E,\w)$ as the collection of homology classes $D^2=-2$ of embedded $\w$-symplectic spheres, such that $D\cdot E<0$.  We have the following \textbf{coarse stratification} of $\mJ_\w$.

\begin{lma}\label{estable2}
	 Fix an exceptional class $E$.  Then $\mJ_{\w}$ can be decomposed into a disjoint union of the following three subsets:\\
1) $\mJ^{0}_\w(E):=\mJ_{E}$.  This is an open and dense subset of $\mJ_{\w}$.\\
  2) $\mJ^2_\w(E):=\bigcup_{D\in\mD(E,\w)} \mJ_{D}$ is a finite union of (possibly disconnected and mutually intersecting) codimension 2 submanifolds in $\mJ_\w$, which is locally closed.\\
  3)  $\mJ^{4}_\w(E):=\mJ_{\w}\setminus (\mJ^0_\w(E) \cup \mJ^2_\w(E))$ is a closed subset in $\mJ_\w$.  It is the union of closed subsets in finitely many locally closed submanifolds of $\mJ_\w$, and each of these submanifolds has codimension at least $4$.

Therefore, $\mJ^2_\w(E)\subset\mJ_w^0(E)\cup \mJ_w^2(E)$ is a closed subset, and any continuous loop $\gamma:S^1\to \mJ_\w$ or disk $\sigma: D^2\to\mJ_\w$ can be perturbed to be disjoint from $\mJ_\w^4(E)$ and transversal to $\mJ_\w^2(E)$.  

\end{lma}

\begin{proof}
	It follows from Gromov-Taubes theory that $\mJ_E$ is open and dense, noting that an embedded exceptional curve has automatic transversality \cite{HLS95}.  From \ref{stratum} we also know that each $\mJ_{D}$ is of codimension 2.  To see that $\mD(E,\w)$ has finite cardinality, we assume  $\w$ is reduced without loss of generality.  We claim that if $D\in\mD(E,\w)$, then $D\cdot H\le E\cdot H$.  

  Assume the contrary that $D\cdot H>E\cdot H$, take $J\in\mJ_D$ and a stable $J$-representative $\Sigma_E$ of $E$.  Since $D\cdot E<0$, $\Sigma_E$ must contain a component which is a multiple cover of $D$.  Therefore, $\Sigma_E$ also contains at least one component $C$ of $\Sigma_E$ such that $C\cdot H<0$.  By Lemma \ref{a<0}, $[C]=-kH+(k+1)E_1-\sum_{k_j\neq1}E_{k_j}$.  But such a curve has negative virtual dimension, and we may avoid the existence of such $C$ by choosing $J$ generically outside of the $D$-representative, leading to a contradiction. 

  Therefore, the $H$ coefficient of $D$ is upper bounded by that of $E$, while it is lower bounded by zero from Lemma \ref{a<0}.  It is elementary to see that only finite $-2$ classes satisfy this constraint.

  Given $J\in\mJ_{\w}\setminus (\mJ_\w^0(E) \bigcup \mJ_\w^2(E))$, the stable $J$-curve of class $E$ is decomposed as more than one components
  \[
    E=\sum_{i=1}^l \lambda_id_i, \hskip 3mm \lambda_i\in\ZZ_+.
  \]

Here, we require each $d_i$ not to be a multiple class, hence represented by a somewhere injective curve. From Lemma \ref{l:Edecomposition} (1), we may assume $d_1^2\le \min_{1\le i\le l}\{-1,d_i^2\}$; and if $d_1^2=-1$ is represented by an embedded $J$-curve, then we assume $d_i$ should also be represented by an embedded curve if $d_i^2=-1$ for any other $1\le i\le l$ (in other words, we pick the most singular component as $d_1$).  

If $d_1$ is represented by a singular curve, we see that such $J$ is of codimension $4$ from Lemma \ref{l:Edecomposition} (2).  If $d_1^2=-1$ is represented by an embedded curve, then all other $d_i$ square non-negatively, or is an exceptional class.  For energy reasons, $d_i\neq E$, so $d_i\cdot E\ge0$ by Lemma \ref{l:Edecomposition} (3), leading to a contradiction.

The last case to be discussed is when $d_1^2\le-2$ is represented by an embedded rational curve.  If $d_1^2\le -3$, it is again clear from the dimension formula that $vd(d_1)\le -4$.   

We are left with the case when $d_1^2=-2$.  By our assumption, $J\notin\mJ^2_\w(E)$, therefore, $d_1\cdot E\ge0$.  We have the following possible cases for $d_i$ when $i>1$.

\begin{enumerate}
  \item If some $d_i^2=-1$, but is not represented by an embedded sphere (i.e. when $d_i$ is not exceptional), we are back in the already-discussed case and conclude that $J$ lies in a submanifold of codimension $4$ or higher.  

  \item If some $d_i^2=-2$ is represented by an embedded sphere, then from Proposition \ref{stratum}, $J$ is also of codimension $4$ or higher because it supports both $d_1$ and $d_i$.

  \item If $d_i^2\ge0$ or $d_i\in\mE$ holds for all $i>1$, then $E\cdot E=\sum_{i=1}^k \lambda_i(d_i\cdot E)\ge0$.  Here, $d_1\cdot E\ge0$ from our assumption, and $d_i\cdot E\ge0$ by Lemma \ref{l:Edecomposition} (3).  This leads to a contradiction. 
\end{enumerate}

The fact that $\mJ^0_\w(E)$ is open and that $\mJ^2_\w(E)$ is locally closed both follow from Proposition \ref{stratum}, which offers an open neighborhood near a given point of the respective strata.  This also implies $\mJ_\w^4(E)$ is closed, because its complement $\mJ^0_\w(E)\cup\mJ^2_\w(E)$ is open from the above argument.  For each $J\in\mJ_\w^4(E)$, it must be contained in some submanifold $\mJ_{d_1}$ as described above, or $\mJ_{d_1,d_i}$ in case (2), but any such submanifolds are of codimension 4 or higher.  Each such $\mJ_\mC\cap\mJ_\w^4(E)=\mJ_\mC\setminus(\mJ_\w^0(E)\cup\mJ_\w^2(E))$ is a closed subset in $\mJ_\mC$ by definition. \footnote{Note that this does not automatically give a stratification of $\mJ_\w^4(E)$ using $\mJ_\mC$, because there are examples when $\ov\mJ_{\mC_1}\cap\mJ_{\mC_2}\neq0$ but $\ov\mJ_{\mC_2}\nsubseteq\mJ_{\mC_2}$.  See \cite{LLW21}.} 

Lastly, we need to show that $\mJ_\w^4(E)$ is covered by closed subsets of finitely many submanifolds of the shape $\mJ_\mC$.  If $J\in\mJ_\w^4(E)$ admits a curve with negative $H$-coefficient, it is covered by submanifolds $\mJ_S$, where $S$ is a class of the shape described in Lemma \ref{a<0}.  There are only finitely many such classes with $\w(S)>0$.  

If $J$ admits no such classes with $d\cdot H<0$, it must be covered by some $\mJ_d$ of codimension 4 or higher, where $0\le d\cdot H<E\cdot H$.  Such a $J$-curve can be smoothed into an embedded symplectic representative (see \cite{McSingularity}).  From \cite[Lemma 3.3]{Chen20}, such classes satisfies $d\cdot E_i\ge0$ for all $E_i$ in the standard basis.  There are also only finitely many such classes $d$ for energy reasons.  Since $J\notin\mJ^0_\w(E)\cup\mJ^2_\w(E)$, it lies in a closed subset in the relative topology.

\end{proof}

\begin{dfn}\label{d:stratificationA}
   Given $J\in\mA_\w$ and an exceptional class $E$, one may take a symplectic form $\w'$ isotopic to $\w$, such that $J\in\mJ_{\w'}$.  There is a decomposition
   $$\mA_\w=\mA_\w^{0}(E)\sqcup\mA_\w^2(E)\sqcup\mA_\w^4(E)$$
   such that $J\in \mA^\dagger_\w(E)\Longleftrightarrow J\in\mJ_{\w'}^\dagger(E)$, where $\dagger=0, 2, 4$.
\end{dfn}

\begin{lma}\label{l:Aw}
    The stratification $\mA_{\w}^\dagger(E)$ is well-defined and we have the following properties similar to \ref{estable2}:

    1) $\mA^{0}_\w(E):=\mA_{E}$ is an open and dense subset of $\mA_{\w}$.\\
  2) $\mA^2_\w(E):=\bigcup_{D\in\mD(E,\w)} U_{D}$ is a finite union of codimension 2 submanifolds in $\mA_\w$, which is locally closed.\\
  3)  $\mA^{4}_\w(E):=\mA_{\w}\setminus (\mA^0_\w(E) \cup \mA^2_\w(E))$ is a closed subset in $\mA_\w$.  It is the union of closed subsets in finitely many locally closed submanifolds of $\mA_\w$, and each of these submanifolds has codimension at least $4$.

  Here, $U_D:=\{J\in\mA_\w(E):D\text{ admits an embedded }J\text{-representative}\}.$
\end{lma}

\begin{proof}
      The stratification $\mJ_{\w'}^\dagger$ depends only on the existence of embedded curves of the class $E$ or $D$, so the choice of $\w'$ is irrelevant to which stratum $J$ belongs to.  The rest of the claims are proved as in \ref{estable2} verbatim, while Proposition \ref{stratum} has a $\mA_\w$ counterpart, proved in \cite[Lemma 2.6]{McDacs}.

\end{proof}

In later sections, we will use the following stratification of $\mA_\w$ for a set of exceptional classes.

\begin{dfn}\label{d:setStratification}
   Given a subset $\Pi\subset \mE_K$, define a $\Pi$-stratification of $\mA_\w$ as
$$\mA_\w=\mA_\w^{0}(\Pi)\sqcup\mA_\w^2(\Pi)\sqcup\mA_\w^4(\Pi),$$
where $$\mA_\w^0(\Pi)=\bigcap_{E\in\Pi}\mA_\w(E),$$
 $$\mA_\w^2(\Pi)=\bigcup_{E\in\Pi}\mA_\w^2(E),$$
  $$\mA_\w^4(\Pi)=\mA_\w\backslash(\mA_\w^0(\Pi)\cup\mA_\w^2(\Pi)).$$

Occasionally, we use the notation $\mA_\w^{0}(\Pi)$ for $\Pi\subset H_2(X;\mathbb{Z})$ which is not a subset of $\mE_K$.  This will denote the set of almost complex structures $J\in\mJ_{\w'}$, where $\w'$ is isotopic to $\w$, and that each class $A\in\Pi$ is represented by an embedded $J$-curve.

\end{dfn}

Extra attention should be paid to the second stratum $\mA^2_\w(E)$.  There is one useful distinction between $U_D$ and $\mJ_D$: while we do not know whether $\mJ_D$ is connected for a given symplectic form $\w$, $U_D$ can be proved connected for any $\w$.

\begin{lma}\label{l:connected}
    $U_D$ is connected.
\end{lma}

\begin{proof}
  Given two almost complex structures $J_0, J_1\in U_D$, we have a $J_i$-curve $\Sigma_i$ with $[\Sigma_i]=D$ for $i=0,1$.  Let $J_i$ be compatible with $\w_i$, there is a path of diffeomorphisms $\phi_t$ starting from the identity such that $(\phi_1)_*(\w_0)=\w_1$.  Consider $\phi_1(\Sigma_0)$.  This curve of self-intersection $(-2)$ is homologous to $\Sigma_1$, and hence we can find a homologically trivial $\w_1$-symplectomorphism $f$, such that 
  $$f\circ\phi_1(\Sigma_0)=\Sigma_1$$
   from \cite[Proposition 2.1]{BLW12}.  The details of construction of $f$ largely mimic the proof of \cite[Corollary 1.2]{BLW12}, and we only provide a sketch here.

  First of all, the classification in \cite{DLW18} allows us to assume $D$ to be $E_1-E_2$.  From \cite{MO15}, one may find a complete set of exceptional curves $C_i^l$ with $[C_i^l]=E_l$, for $i=0,1$ and $l=3,\cdots,n$, and $[C_i^{n+1}]=H-E_1-E_2$, where $C_i^l$ are disjoint from $\Sigma_i$.  One may then blow down these exceptional curves and obtain a minimal symplectic surface $M_i$ with a collection of embedded balls $\mB_i$, where both $M_0$ and $M_1$ are symplectomorphic to $S^2\times S^2$ with a symplectic class determined by $\w_1(H)$ and $\w_1(E_1-E_2)$.  Use a symplectomorphism $g: M_0\to M_1$, and \cite[Corollary 2.8]{AM99}, \cite[Proposition 6.4]{LW12} yields a further symplectomorphism $h:M_1\to M_1$ which sends $g(\Sigma_0)$ to $\Sigma_1$ (or rather, their inverse proper transform under the blow down of $C_i^l$).  From \cite[Proposition 2.1]{BLW12}, one has another symplectomorphism of $\psi:M_1\to M_1$, such that $\psi\circ h\circ g(\mB_0)=\mB_1$ and $\psi(\Sigma_1)=\Sigma_1$.  This last composition allows us to blow up $\mB_i$ and obtain the desired symplectomorphism of $f: (X, \w_1)\to (X, \w_1)$.  From the construction, it is clear that $f(\Sigma_0)=\Sigma_1$ (and also $f(C_0^l)=C_1^l$), concluding our construction.

  From \cite[Theorem A.1]{LLW22}, we see that $f$ is smoothly isotopic to the identity.  Therefore, $f\circ\phi_1$ is smoothly isotopic to identity through a family $\theta_t$, and the push-forward $\theta_t(J_0)$ is an isotopy of almost complex structures in $U_D$ such that the $\theta_1(J_0)$ has the same $D$-curve as $J_1$.  $\theta_1(J_0)$ and $J_1$ can clearly be connected: one may isotope the almost complex structures on $\theta_1(\Sigma_0)=\Sigma_1$ first, by preserving the holomorphicity of this symplectic curve, then isotope the almost complex structure in the complement.  Both choices are contractible, and hence connected.
\end{proof}

We also have the following existence of curves for $U_D$.

\begin{prp}\label{-2res}
  Given a (-2) sphere class $D$, and $W$ a symplectic class such that $D\cdot W>0$.  Then the space of $J$ making $nW$ embedded is an open dense subset in $U_{D}$. \end{prp}
\begin{proof}
   From \cite[Lemma 2.2]{McD96}, $Gr(nW)\neq0$.  Therefore, by Theorem \ref{thm:1},  for any symplectic form $\w$ with $[\w]=W$ and embedded $\w$-sphere $\Sigma$ with $[\Sigma]=D$, there is a open dense subset of $\mJ_\w(\Sigma)$ such that there is an embedded holomorphic representative in class $nW$ for any symplectic form with $[\w]=W$.  Take the union of these subsets over all $\w\in\mS_\w$ and $\Sigma$, we have an open dense subset in $U_D$.
\end{proof}

As a corollary, we have the following inflation lemma for $U_D$.

\begin{cor}\label{Dinf}

Let $(X,\w)$ be a symplectic rational surface, and $D$ is represented by an embedded $\w$-symplectic sphere with $D^2=-2$.


  Then for any symplectic class $W$, there exists  an open dense subset $U_D^0\subset U_{\w, D}$ such that for any $J\in U_{\w, D}^0$, $J$ is also compatible with some $\w'$ such that $[\w']=[\w]+tW$ for any $t>0$.
\end{cor}

\begin{proof}
  Immediate from Theorem \ref{t:inflation} and Proposition \ref{-2res}.
\end{proof}

\subsection{$\mA_\w$ inflations} 
\label{sub:generic_}

The purpose of this section is to prove the following proposition for inflations inside $P_K^+$.

\begin{prp}\label{p:universal}

Take a deformation of symplectic forms $\{\tau_t\}_{t\in[0,1]}$ inside $P_K^+$ such that $\tau_t=(1-t)\tau_0+t\tau_1$.  Assume the following condition holds
\begin{equation}
    [\tau_1]\cdot D>0\Longrightarrow[\tau_0]\cdot D>0\text{ for any }D\in \Gamma_{K_0}.
    \label{e:condition}
\end{equation} 

Then there is a surjective homomorphism 
\begin{equation}\label{e:surjective}
    \varpi_{10}:\pi_1(\mA_{\tau_1},*)\to\pi_1(\mA_{\tau_0},*). 
\end{equation}

Moreover, if $[\tau_0]\cdot D>0$ also implies $[\tau_1]\cdot D>0$, then $\pi_1(\mA_{\tau_t},*)$ is invariant under this deformation.

If $\tau_0$ and $\tau_1$ are of either type $\mathbb{A}$ or type $\mathbb{D}$ and are reduced, the above assertions are true as long as condition \eqref{e:condition} holds for simple roots $D=l_i$.

\end{prp}

We start with the following lemma, which is a special case of the $J$-inflation in \cite[Lemma 1.1]{McDuffErratum}.

\begin{lma}\label{Combinf}

   Given a symplectic class $W= \sum_{X_i\in \Pi} a_i X_i, a_i\ge0$, where $\Pi\subset\mE_K$ is a finite subset.  Assume that $\w_0, W\in P_K^+$, and $J\in \mJ^0_{\w_0}(\Pi)$, then there exists some smooth family $\{\w_t\}_{t\ge0}$, such that $[\w_t]=[\w_0]+tW$, and $J\in\mJ^0_{\w_t}(\Pi)$ for any $t>0$.  In other words, $\mA^0_{\w_0}(\Pi)\subset\mA^0_{\w_t}(\Pi)$.


\end{lma}

\begin{proof}

Take $K=\min_{X_i\in\Pi}\w_0(X_i)$ and $0<\epsilon<\min_{i\in I}\frac{1}{a_i}$.  From Theorem \ref{t:inflation}, there is a $J$-compatible form $[\w_{t_1}]=[\w_0]+t_1\cdot X_1$, where $0\le t_1< (a_1\cdot\epsilon)K$.  Note that $\langle X_i, X_j\rangle\ge0$ for any $X_i\neq X_j$ from Lemma \ref{l:Edecomposition}.  This implies $[\w_{t_1}]\cdot X_i\ge[\w_{0}]\cdot X_i$ for all $i\neq 1$, and hence one may repeat the above process for $i\in I$ once each, and each $X_i$ can be inflated by the amount of $t_i<(a_i\cdot\epsilon)K$. This finite process will conclude that $[\w_0]+r\cdot W$ belongs to the $J$-compatible cone for $0\le r< K\epsilon$.  Since $W\in P_K^+$, it pairs with each $X_i$ positively.  Therefore, $[\w_0]+rW$ has bigger areas on each $X_i$, and the above iterative inflation can be repeated indefinitely.  Lastly, one may use the same argument as in Corollary \ref{cor:smoothALLP} to improve $\w_t$ into a smooth family.






\end{proof}

\begin{proof}[Proof of Proposition \ref{p:universal}]

  To set up the proof, we first need an appropriate coarse stratification of $\mA_{\tau_i}$.  Fix a large $N\in\mathbb{N}$.  We can ensure that $\tau_{0}-\frac{1}{N}\tau_1$ and $\tau_1-\frac{1}{N}\tau_0$ both lie inside $P_K^+$ when $N$ is large enough.  One then write $\tau_{
  0}-\frac{1}{N}\tau_1=\sum_i a_iX_i$, where $X_i\in \mE_K$.  Define exceptional classes $Y_j$ similarly for $\tau_1-\frac{1}{N}\tau_0$, and denote the union of all such $X_i$ and $Y_j$ as $\Pi$.  From Lemma \ref{Combinf}, $\mA^0_{\tau_1}(\Pi)\subset\mA_{\tau_0}^0(\Pi)$, since $ \tau_1 + N (\tau_0- \frac1N \tau_1) =N \tau_0$ and the latter has the same set and stratifications of compatible almost complex structures as $\tau_0$.  The other direction of inflation holds true for the same reason, hence we have 
  \begin{equation}\label{e:A0}
      \mA_{\tau_0}^0(\Pi)=\mA_{\tau_1}^0(\Pi). 
  \end{equation}
    Note that the choice of our stratification depends on the endpoints, but $\mA_{\tau_t}$ inherits the same set of $\Pi$ for this stratification purpose.  

  For simplicity, we will omit the dependence on $\Pi$ unless explicitly specified in the rest of the proof.






  We now consider the inflation of $\mA^{2}_{\tau_1}$ by applying Proposition \ref{-2res} and Corollary \ref{Dinf} to each $D_i\in\bigcup_{E_i\in\Pi}\mD(E_i,\tau_1)$ appearing in the decomposition $\mA^{2}_{\tau_1}= \cup_i U_{\tau_1, D_i}$.  Consider again $W:=\tau_0-\frac{1}{N}\tau_1\in P_K^+$.  Since $\tau_1\cdot D_i>0$ from the definition of $\mA^{2}_{\tau_1}$, we have $\tau_0\cdot D_i>0$ by our assumption.  Therefore, $W\cdot D_i>0$ for all $i$ when $N$ is sufficiently large, because $i$ varies within a finite set.
   Therefore, Proposition \ref{-2res} gives an open dense subset $\mT_1\subset\mA^2_{\tau_1}$, such that $nW$ admits a $J$-embedded representative for $J\in\mT_1$.  Combining \ref{Combinf} and \ref{Dinf}, we see that 
   \begin{equation}\label{e:10inclusion}
   \mA_{\tau_1}^0\cup\mT_1\subset\mA_{\tau_0}
   \end{equation}
    by $J$-inflation, since $\tau_1+\frac{N}{n}\cdot n(\tau_0-\frac{1}{N}\tau_1)=N\tau_0$.  On the other hand, we have the following lemma.


\begin{lma}\label{l:complementfund}
   The inclusion $\iota_1:\mA^0_{\tau_1}\cup\mT_1\hookrightarrow\mA_{\tau_1}$ induces an isomorphism of fundamental groups.
\end{lma}

\begin{proof}
  For dimension reasons, it is clear that removing $\mA_{\tau_1}^4$ does not change the fundamental groups of $\mA_{\tau_1}$, and any loop $\gamma\subset\mA_{\tau_1}$ can be isotoped into $\mA_{\tau_1}^0$.  This implies that $\iota_1$ induces a surjective map between fundamental groups.

   To see that $\iota_1$ is also injective, take a loop $\gamma$ in $\mA^0_{\tau_1}\cup\mT_1$ bounding a disk $w$ in $\mA_{\tau_1}$.  Again for dimension reasons, $w$ can be isotoped away from $\mA^4_{\tau_1}$.  Since $\mA^2_{\tau_1}$ consists of finitely many  codimension $2$ submanifolds in $\mA_{\tau_1}\setminus\mA_{\tau_1}^4$, all intersections between $w$ and $\mA_{\tau_1}^2$ can also be assumed transversal hence finite.  Because of the density of $\mT_1$, we can deform $w$ locally so that the new intersection remains transversal to $\mA_{\tau_1}^2$ and lies inside $\mT_1$ \footnote{We actually only made use of the fact that $\mT_1$ has non-empty intersection with each connected component of $\mA_{\tau_1}^2$.}.  This implies $\gamma$ is also contractible in $\mA_{\tau_1}^0\cup\mT_1$.



\end{proof}

   Therefore, \eqref{e:10inclusion} induces a homomorphism $\varpi_{10}:\pi_1(\mA_{\tau_1},\ast)\to\pi_1(\mA_{\tau_0},\ast)$.  $\varpi_{10}$ is surjective, since all loops in $\mA_{\tau_0}$ can be isotoped into $\mA_{\tau_0}^0=\mA_{\tau_1}^0$ by \eqref{e:A0}.

   If the signs on each $-2$ spherical class $D$ coincide for $\tau_0$ and $\tau_1$, the above argument goes both ways, hence we have the following lemma.



\begin{lma}\label{l:commonDense}
   There is a set of almost complex structure $\mT\subset\mA_{\tau_0}^2\cap\mA_{\tau_1}^2$ which is open and dense in both $\mA_{\tau_i}^2$.
\end{lma}

\begin{proof}
  Since the topology of either $\mA^2_{\tau_i}$ is inherited from the space of all almost complex structures supported by $X$, what we proved above shows that there exists an open dense subset $\mT_i\subset\mA_{\tau_i}^2$ which can be included as a subset of $\mA_{\tau_{1-i}}^2$.  This implies $\mT=\mT_0\cup\mT_1$ satisfies the desired property.
\end{proof}

  As a result, we have an embedding $\iota_i: \mA^0\cup\mT\hookrightarrow\mA_{\tau_i}$ as in Lemma \ref{l:complementfund} for each $i$.  The isomorphism of fundamental groups follows.

 Lastly, if both $\tau_0$ and $\tau_1$ are of type $\mathbb{A}$ or $\mathbb{D}$ and are reduced, then $\tau_i\cdot D\ge0$ ($\le0$, respectively) if $D\cdot H\ge0$ ($\le0$, respectively) from Lemma \ref{l:-2class}.  The reduction procedure of Lemma \ref{l:-2class} also decomposes any $D\in\Gamma_{K_0}$ into a positive combination of simple roots.  Therefore, if $\tau_1\cdot D>0$, it means the decomposition of $D$ contains a simple root with positive $\tau_1$-area, which implies $\tau_2 \cdot D>0$ as well.


\end{proof}

\begin{rmk}\label{rem:naturality}



   The above lemma establishes a partial order among the \textbf{level $2$ chambers} in the sense of \cite{ALLP}. In fact, it endows the ``polytope'' $S\mR_n$ a structure of semi co-simplicial set of $\pi_1(\mA_\tau,\ast)$.
   This family of surjective homomorphisms $\varpi_{10}$ can be considered natural once $\Pi$ is fixed for a pair of $\tau_0$ and $\tau_1$, since the choice of $\mT$ can easily be proved irrelevant.


    To cancel out the dependence of $\Pi$, one may give each $c_1$-positive reduced symplectic class a \textbf{framing}, i.e. a fixed decomposition of $[\w]=\sum_i a_iX_i$ for $X_i\in\mE_K$ and $a_i>0$, then the homomorphism between fundamental groups of a pair of framed symplectic classes is clearly natural.  The only ambiguity involved in $\varpi_{01}$ in Proposition \ref{p:universal} can be parametrized by the framings of the same form.  However, if we take the union of two different framings $\Pi=\Pi_0\cup\Pi_1$ for the same symplectic form $\w$, it is not hard to prove that the induced isomorphism is indeed the identity map, so $\varpi_{01}$ is independent of the choices of framings.

    Since $c_1$-positive reduced symplectic forms constitute a cone, there are even higher naturality properties for higher dimensional symplectic deformations, which we will not pursue here.
\end{rmk}


\subsection{ Stability of $Symp_h(X,\w)$ along 2 types of rays}\label{s:aray}

Given $u\in N\mR_n$, we consider two distinguished directions of symplectic deformation.

\begin{cor}\label{nonbalstab}\label{nonbalstab'}\label{smallsize'}

Given $[\w_1]=(1|c_1,\cdots,c_n)\in N\mR_n$ , the following types of deformation $\{\w_t\}_{t\in(0,1]}\subset N\mR_n$ exists, and $\pi_1(\mA_{\tau_t},*)$ is invariant under the these deformations.

\begin{enumerate}[(i)]
  \item \textbf{(A-extremal deformation)} If $\w_1$ is of type $\mathbb{A}$ or $\mathbb{D}$, $\w_t$ has the following cohomology classes:
  \begin{equation}\label{e:A-line}
    [\w_t]=(1| (c_1-1)t+1 , t c_2, \cdots tc_n ),  0<t<1.
\end{equation}
  \item \textbf{(Minimal deformation)} Assume that $[\w]=(1|c_1,\cdots,c_{n-m},\underbrace{\eta,\cdots,\eta}_m)$ is of type $\aA$, $\DD$ or $\mathbb{E}$, but not type $\DD_{n-1}$ nor $\mathbb{E}_n$.  Then
  \begin{equation}\label{e:cn-m}
        [\w_t]=(1|c_1,\cdots,c_{n-m},\underbrace{t\eta,\cdots,t\eta}_m), 0<t<1.
   \end{equation}
\end{enumerate}

\end{cor}

\begin{proof}

    For (i), by Lemma \ref{extray},  $\beta=H-E_1$ always has an embedded representative for any $J\in \mA_{\tau_1}$, therefore, the inflation family along $\beta$ has exactly the cohomology class given by the desired $[\w_t]$ up to a rescale.

    Note that the family $\{\w_t\}_{t\in[0,1]}$ forms the convex hull between $\w_1$ and $\w_0=H-E_1$.  It is elementary to check that $\w_t$ lie in the interior of the $N\mR_n$ for $t\in(0,1)$.

     To apply Proposition \ref{p:universal} for the type $\mathbb{A}$ and $\mathbb{D}$ cases, we only need to check that signs of simple roots are invariant for the whole $\w_t$-family, which is also elementary.

     For (ii), one inflates along the minimal exceptional classes $E_{n-m+1},\cdots, E_{n}$ by Lemma \ref{minemb}.  The fact that this family stays inside $N\mR_n$ and that they have the same signs on simple roots are also easy to check (which requires us to exclude the case of $\mathbb{D}_{n-1}$ and $\mathbb{E}_n$).  Therefore, Proposition \ref{p:universal} applies.

\end{proof}

\begin{cor}\label{01stab}
$\pi_0(Symp_h(X,\w_t))$ is invariant under the two types of deformations in Lemma \ref{nonbalstab}.
\end{cor}

\begin{proof}

Note that we have an inclusion of $\mA_{\w_1}\subset\mA_{\w_t}$ for all $t\in(0,1)$ under the above two types of deformations.  For the A-extremal deformation, this inclusion is induced by inflating $H-E_1$ by Lemma \ref{extray}; and for minimal deformation, by inflating $E_{n-m+1},\cdots,E_n$ from \ref{minemb}.  Therefore, one may choose $\mT_1=\mA_{\w_1}^2$ in Lemma \ref{l:complementfund} for both A-extremal and minimal deformations.  This yields the following commutative diagram well-defined up to homotopy from \eqref{homotopy fibration}.

 \begin{equation}\label{hcommt}
\begin{CD}
Symp_h(X, \w_1)@>>> \Diff_0(X)@>>> \mA_{\w_1}\\
  @VVV =@VVV   @VVV\\
Symp_h(X, \w_t)@>>>  \Diff_0(X)   @>>>  \mA_{\w_t}
\end{CD}
\end{equation}

  Consider the induced homomorphism between the two long exact sequences of homotopy groups.  Since the inclusion of $\mA_{\w_1}\hookrightarrow\mA_{\w_t}$ induces an isomorphism between fundamental groups, the five lemma implies that $\pi_0(Symp(X,\w_1))\to\pi_0(Symp(X,\w_t))$ is an isomorphism.

\end{proof}

\begin{rmk}
  
The geometry behind the $A$-extremal deformation is the following: when $t\to0$, the rational surface, when considered as a blow-up of $S^2\times S^2$, is taking a ``topological limit'' where the ratio between the two components goes zero.  This idea has been applied to the study of higher homotopy groups in \cite{McDacs,Buse11} for minimal ruled surfaces, and extended recently by Buse and the first author to blow-ups of ruled surfaces in \cite{BL1}.

\end{rmk}

\section{Symplectic Torelli groups for type $\aA$ symplectic forms}\label{s:aA}

\subsection{Symplectic Torelli groups under minimal blow-down}

Take $[\w]=(1|c_1,\cdots,c_{n-m})$, and $\vec{c}=(c, c, \cdots, c)$ is an $m$-dimensional vector with identical entries, where $c<c_{n-m}$. Assume that $(M,\w)$ admits a symplectic packing of $m$ symplectic balls of area $c$, and denote $(\widetilde{M}, \w_{\vec{c}})$ as the resulting blow-up, then $[\w_{\vec{c}}]=(1| c_1,\cdots,c_{n-m},c,\cdots, c)$.

Applying the argument in \cite[Theorem 2.5(ii)]{LP04}, we have the following fibration
\begin{equation}\label{e:LPsequence}
Symp(M, \sqcup_i B_i(c); \w)^{U(2)} \to Symp(M,\w) \to \text{Emb}_{\w}^*(B^4(\vec{c}), M),     
\end{equation}

where $\text{Emb}^*_{\w}(B^4(\vec{c}), M)$ is the space of (ordered, parametrized) embeddings from the $m$-disjoint standard balls of size $c$ to $M$, \emph{modulo} a $U(2)$-action of each ball.  Therefore, the homotopy fiber $Symp(M, \sqcup_i B_i(c); \w_{\vec{c}})^{U(2)}$ is the subgroup of $Symp(M,\w)$ which acts on each fixed parametrized embedded ball by an element of $U(2)$.  Note that in \cite{LP04}, the authors used the space of unparametrized ball-packings as the homotopy cofiber, which is the parametrized embedding modulo the whole $Symp(B^4(c))$.  By Gromov's theorem, $Symp(B^4(c))$ is homotopic equivalent to $U(2)$ by the natural inclusion, and hence the resulting sequence is equivalent to \eqref{e:LPsequence}.  

Given a ball-packing as above, one may blowup $M$ at the corresponding balls, which results in a blown-up rational surface $(\wt M,\w_{\vec{c}})$ as well as a sequence of exceptional curves $\Sigma_i$ of area $c$.  $Symp(M, \sqcup_i B_i(c); \w)^{U(2)}$ is homotopy equivalent to $Symp(\wt M, \sqcup_i\Sigma_i; \w_{\vec{c}})^{U(2)}$, the subgroup of $Symp(\wt M;\w_{\vec{c}})$ consisting of elements that acts on $\Sigma_i$ in by a $U(2)$-element.  We have the following further homotopy equivalences
\begin{multline}\label{e:homotopy}
     Symp(M, \sqcup_i B_i(c); \w)^{U(2)}\sim Symp(\wt M, \sqcup_i\Sigma_i; \w_{\vec{c}})^{U(2)}\\
     \sim Symp(\wt M, \sqcup_i\Sigma_i; \w_{\vec{c}})\sim Symp(\wt M, \w_{\vec{c}}).
\end{multline}
Here, $Symp(\wt M, \sqcup_i\Sigma_i; \w_{\vec{c}})$ is the subgroup of $Symp(\wt M,\w_{\vec{c}})$ which preserves the exceptional divisors $\Sigma_i$.  The second homotopy equivalence follows from the fact that $Symp(\Sigma_i)\sim U(2)$.  The third homotopy equivalence boils down to the contractibility of the space of $\Sigma_i$ embeddings: this follows from \ref{minemb}, where the non-bubbling establishes a homotopy equivalence between the space of almost complex structures and $\Sigma_i$ embeddings.  See \cite[Section 3]{LP04} and \cite[Section 5.3]{ALLP} for further details.

There is a natural {\bf restriction map}  $\text{Emb}_{\w}^*(B^4(\vec{c}), M) \overset{\phi}{\longrightarrow} \text{Emb}_{\w}^*(B^4(\vec{t}), M)$ by restricting to a smaller ball.  Note that this map is well-defined since $U(2)$ preserves the radius.  

In particular, we'll consider the following commutative diagram, which is a multi-ball version of Theorem 1.6 in \cite{LP04}.
  \begin{equation}\label{ballcomm}
\begin{CD}
Symp_h(\widetilde{M}, \w_{\vec{c}}) @>i>> Symp_h(M,\w) @>>> \text{Emb}_{\w}^*(B^4(\vec{c}), M)\\
  @VVV =@VVV  \phi_t @VVV\\
Symp_h(\widetilde{M}, \w_{\vec{t}}) @>i>> Symp_h(M,\w)   @>>> \text{Emb}_{\w}^*(B^4(\vec{t}), M).\\
\end{CD}
\end{equation}
Here $\phi$ is the restriction map, and $i$ is defined by combining \eqref{e:LPsequence} and \eqref{e:homotopy}.


 We now recall a slightly modified version of Theorem A.1 in \cite{LLW22}:
\begin{lma}\label{ball0}
  Given any loop of ball-embedding $\iota_t: \coprod_{i=1}^m B_i(c)\hookrightarrow M$ , $1\le i\le m$, there is always a positive $\delta$ so that the the restriction of $\iota_t$ to $\coprod_{i=1}^m B_i(\delta)$, denoted $(\iota_\delta)_t$, is homotopic to a constant loop in $\text{Emb}_{\w}(B^4(\vec{\delta}), M)^*$.
\end{lma}

\begin{proof}[Sketch of proof]

For any given loop $\iota_t$, one may assume that the image of centers $x_i\in B_i(c)$ is independent of $t$.  This is because one may always find a symplectic isotopy of the loops $\{\iota_t(x_i)\}$ to constant loops.

   Fix an arbitrary metric $g$ on $M$  and assume $\frac{1}{K}<|\iota_t(x_i)|_{C^2}<K$.
From compactness, one may choose a sufficiently small $0<\delta\ll \frac{1}{K}$, so that there is a ball centered at $\iota_t(x_i)$ of radius $r_i$ (measured by $g$), denoted as $D(x_i,r_i)_g\subset M$.  Such a ball satisfies for all $t$ that
$$\iota_t(B_i(\delta))\subset D(x_i,r_i)_g\subset \iota_0(B_i(c)).$$

 Now recall Lemma A.2 of \cite{LLW22}, which concludes that each loop of $\iota_t(B_i(\delta))$ inside $\iota_0(B_i(c))$ is contractible, hence concluding the lemma.

\begin{lma}\label{ballcontr}
The space $\text{Emb}(B(\delta),B(c))^*$ is weakly contractible, if $\delta\ll c$. In particular, the loops $\iota_t(B(\delta_i)))$ inside each  $\iota_0(B(c_i)))$ are homotopic to identity.
\end{lma}

\begin{proof}
  Lemma A.2 of \cite{LLW22} proved the contractibility of unparametrized ball embedding of $B(\delta)$ into $B(c)$, but we already explained that space is homotopic equivalent to $\text{Emb}(B(\delta),B(c))^*$ earlier.
\end{proof}

\end{proof}

  Lemma \ref{ball0} does not imply that $\text{Emb}_\w^*(B^4,\delta)$ is simply connected for any $\delta$.  However, it is sufficient to draw the following conclusion, reducing the Torelli group by minimal blow-downs.

\begin{lma}\label{l:blowdown}
  Consider the minimal deformation family as in Corollary \ref{nonbalstab} (ii).  When $c< c_{n-m}$, the induced map of mapping class group by $i$ as in \eqref{ballcomm}
\begin{equation}\label{e:cn-minduc}
     \pi_0(Symp_h(\wt M,\w_{\vec{c}}))\xrightarrow{i_*} \pi_0(Symp_h(M,\w))
\end{equation}
 is an isomorphism.
\end{lma}

\begin{proof}
  From \eqref{ballcomm}, the homomorphism $i_*$ fits into the commutative diagram \eqref{homotopyballcomm}.  Lemma \ref{ball0} shows that, for any $[\gamma]\in \pi_1(\text{Emb}^*_{\w}(B^4(\vec{c})), M)$, there exists $t_\gamma>0$, such that $\beta_{t_\gamma}\circ\alpha([\gamma])=0$, which implies $\phi_{t_\gamma}\circ\beta([\gamma])=0$ by commutativity. Since $\phi_t$ is an isomorphism from Corollary \ref{01stab}, this implies $\beta=0$.  The fact that the space of ball-packing is connected is well-known due to McDuff \cite{McD96}, which concludes that $i_*$ is an isomorphism.
\end{proof}

  \begin{equation}\label{homotopyballcomm}
\begin{CD}
 \pi_1(\text{Emb}^*_{\w}(B^4(\vec{c})), M) @>\alpha>>  \pi_1(\text{Emb}^*_{\w}(B^4(\vec{t})), M)  \\
  \beta@VVV \beta_t@VVV \\
     \pi_0(Symp_h(\widetilde{M}, \w_{\vec{c}})^*) @>\phi_t>>  \pi_0(Symp_h(\widetilde{M}, \w_{\vec{t}})^*) \\
  i_*@VVV @VVV \\
  \pi_0(Symp_h(M,\w)) @>\cong>> \pi_0(Symp_h(M,\w))\\
       @VVV @VVV \\
     \pi_0(\text{Emb}^*_{\w}(B^4(\vec{c})), M)\cong0 @>>> \pi_0(\text{Emb}^*_{\w}(B^4(\vec{t})), M)\cong0.\\
\end{CD}
\end{equation}

\begin{cor}\label{rem:typeE}
  Take any form $\w$  of type $\DD_{k-1}$ or $\mathbb{E}_k$ in $X_n$ for $k<n$, s.t. $[\w]$ is $c_1$-positive.  We can blow down $X_n$ to $(X_{k}, \bar{\w})$ such that $Symp(X_n,\w)$ and $Symp(X_{k}, \bar{\w})$ without changing the isomorphism type of the Torelli group.
\end{cor}

\begin{proof}
  By choosing a reduced basis, we may always assume $[\w]=(1|m_1,m_2,\cdots,m_n)$ is reduced.  If it is of type $\mathbb{D}_{k-1}$, $\w(l_0)=\w(l_2)=\cdots=\w(l_{k-1})=0$, and $\w(l_1)$, $\w(l_k)>0$.  This means $m_1>m_2=\cdots=m_k>m_{k+1}$.  One can apply Lemma \ref{l:blowdown} consecutively until $E_{k+1}$ is blown down.
\end{proof}

\subsection{Type $\aA$ form and Hamiltonian toric/circle actions}

We are now ready to prove the following

\begin{thm}\label{aA}
Any positive rational surface $(X,\w)$ of type $\mathbb{A}$ has a trivial symplectic Torelli group and a finite symplectic mapping class group.

\end{thm}

\begin{proof}

   From Definition \ref{d:typeForms}, if $[\w]=(1|m_1,\cdots,m_n)$ is of type $\mathbb{A}$, it satisfies one of the following inequalities: $m_1+m_2+m_3<1$, $m_2>m_3$, $m_3>m_4$ or $m_4>m_5$.  In any of these cases, we may apply Lemma \ref{l:blowdown} to the minimal exceptional class iteratively.  This process preserves the type of the form, and will always reduce the Torelli group to th case of $n\le4$.  The triviality of the Torelli group in this case was proved in \cite{LLW15}.

   For the symplectic mapping class groups, one only needs to prove the homological action is finite.  This follows from \cite[Theorem 1.8]{LW12}, which asserts such actions are generated by reflections of Lagrangian spherical classes, which is exactly the Weyl group of the Lagrangian root system $\mL_\w$.  From the definition of $\mathbb{A}$-type symplectic classes, such reflections forms an $\mathbb{A}$-type Weyl group and must be finite.

\end{proof}

\begin{cor}\label{toric}
A toric symplectic surface $(X_\Delta,\w_\Delta)$ is a type $\mathbb{A}$ positive rational surface.  Therefore, $\pi_0(Symp_h(X_\Delta,\w_\Delta))$ is trivial, and its symplectic mapping class group is finite.
\end{cor}

\begin{proof}

 Since log Calabi-Yau surfaces are positive rational surfaces, so is $(X_\Delta,\w_\Delta)$.

 Take the toric complex structure $J_\Delta$.  From Lemma \ref{minemb}, the minimal area exceptional class has a rigid embedded $J$-representative.  This implies it is equivariant with respect to the toric action, and hence a toric divisor.  This shows that the toric blow-down is compatible with the reduction process in Theorem \ref{aA} hence again reduces to $4$-blowups.

\end{proof}

\begin{rmk}
  Note that Corollary \ref{toric} does not hold for circle actions. For example, on 5-point blowup of $\CC P^2$, the form $(1| \frac23, \frac16, \frac16, \frac16, \frac16)$ admits a circle action, but its Torelli group is infinite. See \cite{LLW22} and \cite{ABP19}.
\end{rmk}

\begin{cor}
   Let $(M,\w)$ be a symplectic toric surface.  Then the set of Hamiltonian conjugacy classes of maximal 2-tori in $Ham(M,\w)$ is finite.
\end{cor}

\begin{proof}
  This is an immediate consequence of Corollary \ref{toric} and Theorem 1.3 in
  \cite{Pin08}.
\end{proof}

\section{Symplectic Torelli groups for type $\mathbb{D}$ forms}\label{s:typeD}

The goal of this section is to prove the type $\mathbb{D}$ case of Theorem \ref{genSMC}.  From Lemma \ref{l:blowdown}, one may blow-down the minimal exceptional classes if it is not part of the the $\mathbb{D}$ type Dynkin diagram in $\mL_\w$.  Therefore, we may assume the rational surface $(\CP^2\#n\ov\CP^2, \w)$ to be of type $\mathbb{D}_{n-1}$ and reduced.  Such classes have the form
\begin{equation}\label{e:typeD}
    [\w_a]=(1|a,\frac{1-a}{2}, \cdots, \frac{1-a}{2}),\hskip 2mm \frac{n-7}{n-3}<a<1,
\end{equation}
 from Lemma \ref{conecv}.  We may further reduce our situation by the $A$-extremal deformation.

\begin{lma}\label{l:dense}
   If Theorem \ref{genSMC} holds for a sequence of $a_k\to 1$, then it holds for all $\frac{n-7}{n-3}<a<1$.
\end{lma}

\begin{proof}
  Given any $\frac{n-7}{n-3}<a<1$, take $a_k>a$ and $t=\frac{1-a_k}{1-a}$ from case (i) of Corollary \ref{nonbalstab}, one sees that $\pi_0(Symp_h(\w_a))\cong \pi_0(Symp_h(\w_{a_k}))$.
\end{proof}

Therefore, we will focus on $\w_a$ with $a>\frac{n-3}{n-1}$ and $a\in\mathbb{Q}$ (see \eqref{cv}).  To this end, We rely on the following classical fact of spherical braid groups (see \cite[Lemma 3.4]{LLW22}).

\begin{lma}\label{Hopfian}
The pure and full braid groups on disks or spheres are Hopfian, i.e. every self-epimorphism is an isomorphism.
\end{lma}

In particular, we will prove two epimorphism statements, Lemma \ref{Dsurjbraid} and Proposition \ref{prp:surjective}, which will conclude Theorem \ref{genSMC}.

\subsection{From braid groups to Torelli: a divisorial decomposition} 
\label{sub:a_divisorial_decomposition}

Consider a configuration $C$ of embedded symplectic submanifolds for $X= \CC P^2  \# n\overline {\CC P^2}$ of homology classes shown in Figure \ref{Cn}, called a \textbf{filling divisor}.  The set of these homology classes is denoted by $\mC$.  Take a subset $\mC_0\subset\mC$, which are configurations that intersect symplectically orthogonally.  From Gompf's isotopy lemma \cite[Lemma 26]{Ev11}, $\mC_0\sim\mC$.  We then have the following collection of homotopy fibrations:

 \begin{equation}\label{summary}
\begin{CD}
Symp_c(U)=Stab^1(C) @>>> Stab^0(C) @>>> Stab(C) @>>> Symp_h(X, \omega) \\
@. @VVV @VVV @VVV \\
@. \mG(C) @. Symp(C) @. \mC_0 \simeq \mJ_{C}
\end{CD}
\end{equation}

 Here, $\mJ_{\mC}$ is the subspace of $\mJ_{\w}$ such that each class in $\mC$ has an embedded  $J$-holomorphic representative as in Proposition \ref{stratum}.  From a standard codimension argument as in Proposition \ref{stratum}, $\mJ_C$ is connected, hence $\mC_0$ is also connected (see more details in \cite[Appendix A]{Ev11}).  Here is the glossary of the terms in \eqref{summary}.

\begin{itemize}
  \item $\mC_0$: the space of embedded configurations isotopic to $C$ in the given homology classes.
  \item $Symp(C)$: automorphism of the union of curves $C$ (in the domain).
  \item $Stab(C)$: the subgroup in $Symp_h(X,\w)$ which preserves a given embedded configuration $C$.
  \item $Stab^0(C)$: the subgroup of $Stab(C)$ which fixes $C$ pointwise.
  \item $\mG(C)$: gauge group of the normal bundle of $C$, i.e. automorphisms of the normal bundle of $C$ induced by $Stab^0(C)$.
  \item $Stab^1(C)$: the subgroup of $Stab^1(C)$ which fixes the normal bundle of $C$.
  \item $Symp_c(U)$: the compactly supported symplectomorphisms of $U=X_n\setminus C$.
\end{itemize}

\begin{figure}[ht]
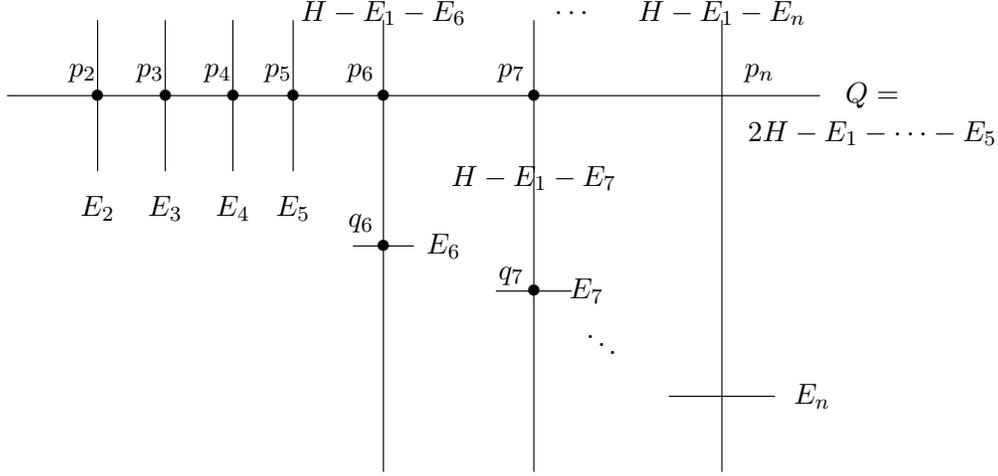

  \centering
\[
\xy
(0, -10)*{};(108, -10)* {}**\dir{-};
(115, -10)*{Q=};
(115,-15)*{2H-E_1-\cdots-E_5};
(50, 0)*{}; (50, -60)*{}**\dir{-};
(50, 1)*{H-E_1-E_6};
(12, 0)*{}; (12, -20)*{}**\dir{-};
(10, -7)*{p_2};
(12, -10)*{\bullet};
(12, -25)*{E_2};
(21, 0)*{}; (21, -20)*{}**\dir{-};
(19, -7)*{p_3};
(21, -10)*{\bullet};
(21, -25)*{E_3};
(30, 0)*{}; (30, -20)*{}**\dir{-};
(28, -7)*{p_4};
(30, -10)*{\bullet};
(30, -25)*{E_4};
(38, 0)*{}; (38, -20)*{}**\dir[red, ultra thick, domain=0:6]{-};
(36, -7)*{p_5};
(38, -10)*{\bullet};
(38, -25)*{E_5};
(46, -30)*{}; (54, -30)*{}**\dir{-};
(47, -7)*{p_6};
(50, -10)*{\bullet};
(50, -30)*{\bullet};
(47,-27)*{q_6};
(58, -30)*{E_6};
(70, 0)*{}; (70, -60)*{}**\dir{-};
(70, -21)*{H-E_1-E_7};(75, 1)*{\cdots};
(95, 0)*{}; (95, -60)*{}**\dir{-};
(95, 1)*{H-E_1-E_n};
(65, -36)*{}; (75, -36)*{}**\dir{-};
(67, -7)*{p_7};
(70, -10)*{\bullet};
(70, -36)*{\bullet};
(67,-34)*{q_7};
(77, -36)*{E_7};
(88, -50)*{}; (102, -50)*{}**\dir{-};
(107, -50)*{E_n};
(100, -7)*{p_n};
(79, -42)*{\ddots};
\endxy
\]
\caption{A filling divisor in $X_n$}
  \label{Cn}
\end{figure}

Each vertical map in \eqref{summary} is the fibration of a transitive action of a Lie group, and the horizontal maps are the inclusion of the corresponding isotropy group.  This follows the same line of argument as in Proposition \cite[Proposition 3.8]{LLW22}, and we only give a detailed justification of the fibration at the center for readers' convenience, since it is less standard (the argument in \cite{LLW22} is affected more essentially due to the different choice of symplectic divisors).

\begin{lma}\label{surj}
$Stab^0(C)\to Stab(C)\to Symp(C)$ is a homotopy fibration.
\end{lma}

\begin{proof}
We adapt the proof of Proposition 5.4 in \cite{LLW22}.

 Denote the spheres in class $H-E_1-E_i$ as $S_i$, the ones in class $E_i$ as $e_i$, and the one in class $2H-E_1-\cdots-E_5$ as $Q$. Let $p_i=e_i\cap Q$ for $2\le i\le 5$; and $p_i= S_i \cap Q$ for $i\ge 6$.  Also define $ q_i= S_i \cap e_i $ for $i\ge 6$. 

 Recall from \cite{Ev11} and \cite{LLW15}, $Symp(C)=\prod_{i=6}^{n} Symp(S_i;p_i,q_i)\times \prod_{i=2}^{5} Symp(S_i;p_i) \times Symp(Q,n-1)\times \prod_{i=6}^nSymp(e_i; q_i),$ where $Symp(S_i; p_i,q_i)$ is the symplectomorphism group of the sphere in class $H-E_1-E_i$ fixing the intersection points $p_i, q_i$, $Symp(Q,n-1)$ is the symplectomorphism group of $Q$ fixing $p_i$, $i=2,\cdots,n$, and $Symp(e_i; q_i)$ is the symplectomorphism of $e_i$ fixing $q_i$. Since $f:Stab(C)\to Symp(C)$ is a group homomorphism, we only need to show the projection to each factor is surjective.  This is clear for $Symp(e_i,q_i)$ and $Symp(S_i;p_i,q_i)$ factors (see Lemma 2.5 in \cite{LLW15}, for example).

The only thing we need to prove is the restriction map of $f$ being surjective on the factor $Symp(Q,n-1)$.
 This means for any given $ h^{(2)}\in Symp(Q,n-1)$, we need to find a symplectomorphism $h^{(4)}\in Stab(C) $  which fixes the whole configuration  $C$ as a set, whose restriction on $Q$ is $ h^{(2)}.$  To achieve this, we can blow down the exceptional spheres $e_i$ and $S_i$, and obtain
 a minimal symplectic rational surface 
 \begin{equation} \label{}
    \wt X=  \left\{ 
         \begin{aligned}
          &S^2\times S^2,  &\text{ when }n=2k  \\
           &\CP^2\#\ov\CP^2,  &\text{when }n=2k+1.        
         \end{aligned}
      \right.
 \end{equation}
  $\wt X$ also comes with a rational curve $\ov Q$, the proper (inverse) transform of $Q$, along with $n-1$ disjoint symplectic balls $\coprod_{i=2}^{n} B(i)$, and $n-5$ rational curves $D_i$ as the proper transforms of $e_i$ for $i\ge6$. Each intersection $B(i)\cap \ov Q$ is a disk $\Omega_i\subset\ov Q$. 

  This blow down process sends  $ h^{(2)}$ in $Symp(Q,n-1)$ to a unique
 $\overline {h^{(2)}} $ in $Symp(\ov Q,\coprod_{i=2}^{n} \Omega_i).$ It suffices to find a symplectomorphism $\overline {h^{(4)}}$ whose restriction is $\overline {h^{(2)}}$, and fixing the image of balls $\coprod_{i=2}^{n} B(i)$ and $D_i$.  Once we have that, blowing the balls $\coprod_{i=2}^{n} B(i)$ up and obtain a symplectomorphism $h^{(4)}\in Stab(C) $ whose restriction to $Q$ is
 the given $ h^{(2)}\in Symp(Q,n-1)$.

Now for a given  $\overline {h^{(2)}} $ in $ Symp(\ov Q,\coprod_{i=2}^{n} \Omega_i)$,
first find  $f^{(4)} \in  Symp(\wt X,\w)$ whose restriction on $\ov Q$ is   $\overline {h^{(2)}} $ in $Symp(\ov Q,\coprod_{i=2}^{n} D_i)$: $\overline {h^{(2)}}$ in $Symp(\ov Q,\coprod_{i=2}^{n} \Omega_i)$ is a  Hamiltonian diffeomorphism on $\ov Q$, therefore, the Hamiltonian function can be extended to a neighborhood so that the induced Hamiltonian diffeomorphism $f^{(4)}\in  Symp(\CC P^2  \# n{\overline {\CC P^2}},\w)$ equals $\overline{h^{(2)}}$ when restricted to $Q$.  $f^{(4)} $ clearly fixes the $(n-1)$ intersection disks $\coprod_{i=2}^{n} \Omega_i$.

Then we need another symplectomorphism  $g^{(4)} \in  Symp(\CC P^2,\w)$ so that $g^{(4)}$  move the $n-1$ symplectic  balls  back to their original position in
$\CC P^2.$    Namely, by Lemma 4.3 and Lemma 4.4  in \cite{Wu13}, the space of ball-packings relative to a divisor is connected, therefore, there exists a symplectomorphism $g^{(4)}\in Symp(\wt X,\w) $ such that  the composition $\overline {F^{(4)}}=g^{(4)}\circ f^{(4)}$ is a symplectomorphism fixing the $n-1$ balls.  After this, the isotopy from $\overline {F^{(4)}}(D_i)$ back to $D_i$ can be achieved by the usual $J$-holomorphic technique: pick an $\w$-compatible almost complex structure $J_0$ so that $\overline {F^{(4)}}(D_i)$ is $J_0$-holomorphic, and there is a path of $\w$-compatible $J_t, t\in[0,1]$ so that $D_i$ is $J_1$-holomorphic, and for all $J_t$, the $J_t$-holomorphic curves that passes through $D_i\cap \ov Q$ are embedded.  Note that we are in the situation of \cite{AM99} where $D_i$ is the fiber class in a minimal rational ruled surface, hence they cannot bubble for energy reasons.  The resulting isotopy of rational curves can therefore be extended to a Hamiltonian isotopy and hence be composed with $\overline {F^{(4)}}(D_i)$.

The end result of the above discussion is a symplectomorphism $\overline {h^{(4)}} $ of $\wt X$, which fixes the  $(n-1)$ balls and $D_i$.   Upon blowing up these balls we obtain an element $h^{(4)}$ in $Stab(C)$, which is  a ball swapping symplectomorphism whose restriction on $Symp(C)$
creates the group  $Symp(Q,n-1)$.  Hence this restriction map $Stab(C) \rightarrow  Symp(C)$ is surjective.

It is clear that the action of $Stab(C)$ on $ Symp(C)$ is transitive and by Theorem A in \cite{Mei02} $Stab(C) \rightarrow  Symp(C)$ is  a fibration. 

 \end{proof}

We first study $Symp_c(U)$ in \eqref{summary} for type $\mathbb{D}_{n-1}$ forms.  Consider a numerical condition 
\begin{equation}\label{cv}
  a >\frac{n-3}{n-1}.
\end{equation}

From Lemma \ref{conecv}, any $\w_a$ satisfying \eqref{cv} lies in the normalized reduced symplectic cone for any $n\ge5$.

\begin{lma}\label{CgenJ}
  Given a positive rational surface $(X_n,\w)$, $[\w]\in P_K^+$.  Then $\w$ is the K\"ahler form for some integrable complex structure $J$, which contains a $J$-holomorphic filling divisor as in Figure \ref{Cn}.  
\end{lma}

\begin{proof}
  Take a good generic complex structure as in \cite[Definition 2.1]{FM88}.  This can be obtained from the blow-up of $n$ distinct points generically chosen from a smooth cubic curve in $\CP^2$ (c.f. \cite[Proposition 2.6]{FM88}).  The genericity of blow-up position guarantees each irreducible component in Figure \ref{Cn} are embedded for this complex structure $J$.  

  On the other hand, the K\"ahler cone of a good generic complex structure is precisely $P_K^+$ by \cite[Proposition 3.4]{FM88}.  This means $J$ is compatible with some $\w'$ cohomologous to $\w$.  Therefore, $\w$ is K\"ahler with respect to some integrable complex structure by McDuff's uniqueness of symplectomorphism type \ref{t:DusaUniqueness}.



  
\end{proof}

\begin{lma}
  Let $(X,\w)$ be $\CC P^2 \# n\overline {\CC P^2}$, $n>5$, $\w\in P_K^+$.  Assume that $[\w]=(1|m_1,\cdots, m_n)$ with rational periods (i.e. all $m_i\in\mathbb{Q}$) with 
\begin{equation}\label{e:Stein}
     2m_1-1-\sum_{i\ge6}m_i>0
\end{equation}
Then there exists a good generic complex structure $J$, which is compatible with $\w$, such that $(U=X_n-C,\w|_U)$ is Stein.  Here $C$ is the unique $J$-configuration as in Figure \ref{Cn}. Therefore, there is a weak homotopy equivalence $Symp_c(U)\sim Symp_c((\bP^1\setminus\{p_1,\cdots p_{n-5}\})\times D^2,\w_{std})$, where the symplectic structure on $\bP^1\times D^2$ is a standard product structure equipped with appropriate symplectic areas on each factor.  

  In particular, this applies to $\w_a$ for  $a>\frac{n-3}{n-1}$.
\end{lma}

\begin{proof}

  Since classes satisfying $\eqref{e:Stein}$ is contained in $P_K^+$, any good generic complex structure $J'$ is compatible with some K\"ahler form $\w'$, where $[\w']=[\w]$.  Clearly, the $J'$-configuration given by Lemma \ref{CgenJ} is unique.

  Given an intersection $p$ between two distinct components of $C$, take a small neighborhood $B_p$, in which $C = \{z_1 z_2 = 0\}$.  By \cite[Lemma 1.7]{Se03}, we can perturb $\w'$ near $p$, such that the resulting symplectic form has $\w''=\frac{\sqrt{-1}}{2}(dz_1\wedge d\bar z_1+dz_2\wedge d\bar z_2)$ standard in $B_p$.  Therefore, we may assume that each intersection between components of $C$ are $\w'$-orthogonal.  Using the pushforward of the pair $(\w',J')$ by \ref{t:DusaUniqueness}, we obtain a K\"ahler pair $(\w,J)$ with a $J$-holomorphic filler divisor $C$, where the pairwise intersections between components are symplectic orthogonal.

 From Proposition 3.3 in \cite{LLW15}, if $PD[l\w]$ is a positive linear combination of homology classes of irreducible components in $C$ for some large positive integer $l\in \mathbb{Z}$, then $U$ is Stein.  The set of homology classes for these irreducible components are given by $\{2H-E_1-E_2-E_3-E_4-E_5, H-E_1-E_6,E_2,E_3, E_4, E_5, E_6; H-E_1-E_i, E_i, \}$, where $7\leq i\leq n$.  Let
\begin{align}
PD([l\omega])=aH-b_1E_1-b_2E_2-b_3E_3-\cdots -b_nE_n \nonumber\\
=d_0( 2H-E_1-E_2-E_3-E_4-E_5 )   \nonumber\\
+ d_2 E_2 +\cdots d_5 E_5 \nonumber \\
+ \sum_{i=6}^n d_iE_i +  \sum_{i=6}^n f_i(H-E_1-E_i) \nonumber
\end{align}

Comparing the coefficients, we have 
\begin{align}
  &d_0 +\sum_{i=6}^n f_i =b_1\\
  &2d_0 +  \sum_{i=6}^n f_i =a,\\
  &d_0-d_j=b_j,\hskip 2mm j=2, 3, 4, 5,\\
  &f_i-d_i=b_i,\hskip 2mm i\ge6.
\end{align} 

Therefore, $d_0=a-b_1$ and $d_j=a-b_1-b_j$ for $j=2, 3, 4, 5$.  Here $d_j>0$ because it is the area of $H-E_1-E_j$.   Since $\sum_{i\ge6}f_i=\sum_{i\ge6}b_i+\sum_{i\ge6}d_i=2b_1-a$, the condition $2b_1-a-\sum_{i\ge6}b_i>0$ ensures a set of positive (possibly rational) solution.  Such solutions can be taken integral when $l$ is sufficiently large.

To see that this complement is symplectomorphic to $(\bP^1\setminus\{p_1,\cdots p_{n-5}\})\times D^2$, take a Hirzebruch surface $F(n-2)$ along with a section $\Sigma_{n-2}$ of self-intersection $n-2$, as well as $n-2$ points of generic position $\{p_2,\cdots,p_n\}\subset\Sigma_{n-2}$.  Take a divisor $C_0$ consisting of the union of $\Sigma_{n-2}$ along with $n-5$ fibers pasing through $p_6,\cdots, p_n$.  Blowing up $p_2,\cdots,p_{n}$ along $\Sigma_{n-2}$, we denote the resulting exceptional classes as $e_2, \cdots, e_n$.  Denote the total transformation of $C_0$ as $C_1$.

As the blow-up of $F(n-2)$, the second homology of the resulting rational surface has a basis $\{b,f,e_2,\cdots,e_n\}$, where $b=[\Sigma_{n-2}]$ and $f$ is the fiber class.  As a rational surface, there is a standard basis consisting of line and exceptional classes $\{H,E_1,\cdots,E_n\}$.  There is a base change of the second homology class of this surface, such that $e_i\mapsto E_i$ for $i=2, 3, 4, 5$ and $f-e_j\mapsto E_j$ for $j\ge 6$, $b\mapsto2H-E_1-\cdots-E_5$ and $f\mapsto H-E_1$.  Such a transformation is unique, and $C_1$ gives exactly a configuration as in Figure \ref{Cn} after the base change.  Therefore, the complement of $C_1$ is biholomorphic to the complement of $C_0$ in $F(n-2)$, which is in turn biholomorphic to $(\bP^1\setminus\{p_1,\cdots p_{n-5}\})\times D^2$.  By \cite[Proposition 16]{Ev11}, the two compactly supported symplectomorphism groups are weakly homotopic equivalent.


\end{proof}




 \begin{lma}\label{sympcu}
$U=(\CC-\{p_1,p_2,\cdots, p_{n-6}\})\times \CC,$ $Symp_c(U,\w_{std})  $
 is weakly contractible.  In particular, it is connected.
\end{lma}
\begin{proof}

Let $m=n-6$.  From \cite[Proposition 15]{Ev11}, the compactly supported symplectomorphism group of $U$ can be identified with that of $(D^2-\{p_1,p_2,\cdots, p_{m}\})\times D^2$, while the two disks share the same area form.  Denote $U_m=(D^2-\{p_1,p_2,\cdots, p_{m}\})\times D^2$ with such a product symplectic form.   There is a compactification of  $U_m$ into $M^c=S^2 \times S^2$, where $U_m=M^c\setminus\{\Sigma_\infty, F_0,\cdots, F_{m}\}$.  Among the components of the compactifying divisor, $\Sigma_\infty$ is a section, $F_i$ is a collection of fibers of the trivial $S^2$-fibration, and we assume $F_i\cap\Sigma_\infty=q_i$ for $i=0,\cdots,m$.  Let us also denote the homology classes $[F_i]=F$ and $[\Sigma_\infty]=B$, and the collection $\{F_i\}_{1\le i\le m}$ as $Z_0$.

In the proof of Gromov's theorem (cf. \cite{MS04} proof of Theorem 9.5.1), $Symp_c(M^c\setminus(\Sigma_\infty\cup F_0))\sim Symp_c(D^2\times D^2)$ is contractible, which can be considered as a subgroup of $Symp_h(M^c)$  which fixes a neighborhood of $\Sigma_\infty\cup F_0$.  Consider the space $\mZ$ of configurations consisting of pairwise disjoint symplectic spheres $C=\{C_1, \cdots C_m\}$ in the homology class $F$, such that $C_i  \cap \Sigma_\infty= F_i  \cap \Sigma_\infty=q_i$.  

\noindent {\bf Claim:}
$\mZ$ is contractible.
\begin{proof}

The contractible space $\mJ_{M^c}$ fibers over $\mZ$ by sending $J$ to the configuration of rational curves of class $F$ passing through $q_i$.  The homotopy fiber of this map is given by another contractible space, the almost complex structures which make a fixed configuration pseudo-holomorphic.

\end{proof}

$Symp_c(M^c\setminus(\Sigma_\infty\cup F_0))$ acts transitively on a smaller space of configurations $\mZ_0\subset\mZ$, consisting of an $m$-tuple of rational curves which coincides with the fixed configuration $Z_0$ in a small neighborhood of $\Sigma_\infty$.  But $\mZ_0$ is homotopy equivalent to $\mZ$ by  Gompf's isotopy (see \cite[Lemma 2.3]{Gom95}, or \cite[Section 5.2.1]{Ev11}).

 Therefore, we have the following fibrations similar to \eqref{summary} (using the same set of notations for the various stabilizer groups)

 \begin{equation}\label{e:Um}
\begin{CD}
Symp_c(U_m)=Stab^1(Z_0) @>>> Stab^0(Z_0) @>>> Stab(Z_0) @>>> Symp_c(D^2\times D^2) \\
@. @VVV @VVV @VVV \\
@. \mG(Z_0) @. Symp(Z_0)^* @. \mZ_0
\end{CD}
\end{equation}

Here, $Symp(Z_0)^*$ is the product of the symplectomorphism groups of each $F_i$ which fixes a small neighborhood of $q_i\in F_i$ (here $F_i$ is considered as an abstract symplectic curve, and its embedding into $S^2\times S^2$ is irrelevant). This is homotopic to the compactly supported symplectomorphism of a disk, hence contractible.  The gauge group $\mG(Z_0)$ is again the product of the gauge group on each $F_i$.  Such gauge group measures the action of $Stab^0(Z_0)$ on the normal bundle of $F_i$.  Again, since the tangent space of $q_i$ is fixed, this gauge group is homotopic to a point.  Therefore, all terms in \eqref{e:Um} are weakly contractible.



\end{proof}

We are ready to come back to \eqref{summary} for a symplectic form $\w$ satisfying condition \eqref{cv}.  Diagram \eqref{summary} has the following form.
 \begin{equation}\label{e:Un6}
\begin{CD}
Symp_c(U_{n-6}) @>>> Stab^0(C) @>>> Stab(C) @>>> Symp_h(X, \omega) \\
@. @VVV @VVV @VVV \\
@. \ZZ^{2n-7} @. (S^1)^{2n-6}\times \Diff^+(S^2,n-1) @. \mC_0 \simeq \mJ_{C}
\end{CD}
\end{equation}

The terms $Symp(C)$ and $\mG(C)$ can be computed by the fact that

$$Symp(S^2,1)\cong S^1, Symp(X^2,2)\cong S^1, Symp(S^2,3)\cong *;$$
$$ Symp(S^2,k)\cong \Diff^+(S^2,k)\cong PB_k(S^2)/\langle \tau\rangle  \text{ for }k>1. $$
$$\mG(S^2,0)\cong S^1, \mG(S^2, 1)\cong *; \mG(S^2, k)\cong\ZZ^{k-1} \text{ for }k>1. $$

Here, $Symp(S^2, k)$ denotes the symplectomorphism groups of $S^2$ which fixes $k$ points, while $\tau$ represents full twist element in the braid group; and $\mG(S^2, k)$ denotes the gauge group of a $D^2$-bundle over $S^2$ which fixes $k$ fibers.  The above calculations were proved in \cite{Ev11} and summarized in \cite[Section 2.1]{LLW15} \cite[Section 3.1]{LLW22}.

Consider the various long exact sequence of homotopy groups associated to \eqref{e:Un6} starting from the leftmost fibration, yielding $Stab^0(C)\cong \ZZ^{2n-7}$ from Lemma \ref{sympcu}.  For the zero$^{th}$ homotopy group of $Stab(C)$, we also have from \cite[Lemma 2.9]{LLW15} that the connecting map from $\pi_1 (S^1)^{2n-6}$ to $\ZZ^{2n-7} $ is surjective.    Therefore, $ \pi_0(Stab(C))\simeq \pi_0(\Diff^+(S^2,n-1)).$  Summarizing the above discussions, we have

\begin{lma}\label{Dsurjbraid}

For any normalized reduced symplectic form $\w$ on $X_n$ satisfying  \eqref{cv}, there is a surjective map  $\pi_0(\Diff^+(S^2,n-1)) \to \pi_0(Symp(X,\w))$.

\end{lma}

\begin{proof}
The long exact sequence of homotopy groups on the rightmost fibration in \eqref{e:Un6} reads
\begin{equation}\label{psi}
\cdots\to  \pi_1 (\mC_0) \to \pi_0(\Diff^+(S^2,n-1)) \to \pi_0(Symp(X,\w))  \to 1.
\end{equation}

\end{proof}

\begin{rmk}\label{rem:generators}
   It seems useful in many occasions that we have an explicit description for a set of generators for $\pi_0(Symp(X,\w))$ from the above discussions, assuming \eqref{cv} and using the ball-swapping constructions as in Lemma \ref{surj}.

   Since $\pi_1((S^1)^{2n-6})\to \ZZ^{2n-7}$ is surjective, the map $Stab(C)\to Symp(C)\simeq(S^1)^{2n-6}\times\Diff^+(S^2, n-1)$ induces an isomorphism in $\pi_0$, which means that a symplectomorphism $\phi\in Stab(C)$ is not isotopic to identity (inside $Stab(C)$) if and only if it induces a non-trivial braid element in $\pi_0(\Diff^+(S^2,n-1))$ when restricted to the component with homology class $2H-E_1-\cdots-E_5$.  From Lemma \ref{surj}, one may construct such a symplectomorphism for each braid element by blowing-down the corresponding exceptional curve, and using an isotopy of the resulting ball embedding.
\end{rmk}


\subsection{From Torelli to braid groups: monodromies from complex moduli spaces}

The aim of this section is to construct an epimorphism $\pi_0(Symp(X,\w))\twoheadrightarrow\pi_0(\Diff^+(S^2,n-1))$.  We will construct a family of rational surfaces $\mY_n$ from the configuration space of $n-1$ points on $\CP^1$, for which each fiber $(\mY_n)_x$ is diffeomorphic to $\CP^2\#n\ov\CP^2$ for $n=2k+1$.  For ease of notations, we will adopt the convention that a point in $\Conf_{n-1}(\CP^1)$, the \textbf{ordered} configuration space, is given by $\{p_i\}_{2\le i\le n}$, where $p_i\in\CP^1$.

\begin{dfn}\label{d:qn}
   Take $F(1):=\CP^2\#\ov\CP^2$ and denote the canonical homology basis consisting of the line class and the exceptional class as $\{h,e_1\}$.  Moreover, fix two distinct holomorphic embeddings $u, v:\CP^1\to F(1)$ whose homology classes are $h$.  Let $(F(1),u(\CP^1))\times\Conf_{n-1}(\CP^1)$ be the trivial family of a first Hirzebruch surface with a nef divisor over $\Conf_{n-1}(\CP^1)$.  There are two types of sections over this trivial family

   \begin{itemize}
      \item The section $s_{p_i}$ for $i\ge4$ is given by $\{u(p_i)\}$ over $\underline{p}:=(p_1, \cdots, p_n)\in \Conf_{n-1}(\CP^1)$.
      \item For $i=2,3$, take $q_i$ to be the intersection between $v$ and the $h-e_1$ curve that passes through $u(p_i)$.  This defines sections $s_{q_i}$.

    \end{itemize}

  Blowing up $s_{p_i}$ and $s_{q_i}$, we have a smooth family of rational surfaces $\mY_n$.  Over a point $\underline{p}\in\Conf_{n-1}(\CP^1)$, we denote the fiber as $(\mY_n)_{\underline p}$, and the proper transform of $u$ as $(\Sigma_0)_{\underline{p}}$.  The union of $(\Sigma_0)_{\underline{p}}$ defines a divisor $\Sigma_0$ in $\mY_n$.  Sometimes the subscript $\underline{p}$ is omitted when the context is clear.
\end{dfn}


In the above constructed $\mY_n$, there is a canonical basis for $H_2((\mY_n)_x,\ZZ)$ given by $h$, $e_1$ and the fiber of the exceptional divisors blowing up $s_{p_i}$ and $s_{q_i}$ for all $x\in\Conf_{n-1}(\CP^1)$.  We denote the exceptional divisors coming from $s_{p_i}$ and $s_{q_i}$ as $e_i$, $i\ge2$.

\begin{lma}\label{l:existenceCurves}
For each fiber of $\mY_n$, $h-e_1-e_i$ and $e_i$ each have an embedded holomorphic representative for all $i\ge2$.

Moreover, exactly one of the following pair of curves admits embedded representatives

\begin{enumerate}
  \item $h-e_2-e_3$ and $h-\sum_{i=4}^{n} e_i$;
  \item $h-e_2-e_3-e_a$ and $h-\sum_{i=4}^{n} e_i$ for some $a\ge4$;
  \item $h-e_2-e_3$ and $h-\sum_{i=4}^{n} e_i-e_a$ for some $a=2$ or $3$.
\end{enumerate}

\end{lma}

\begin{figure}[tb]
  \centering
  \includegraphics[scale=0.85]{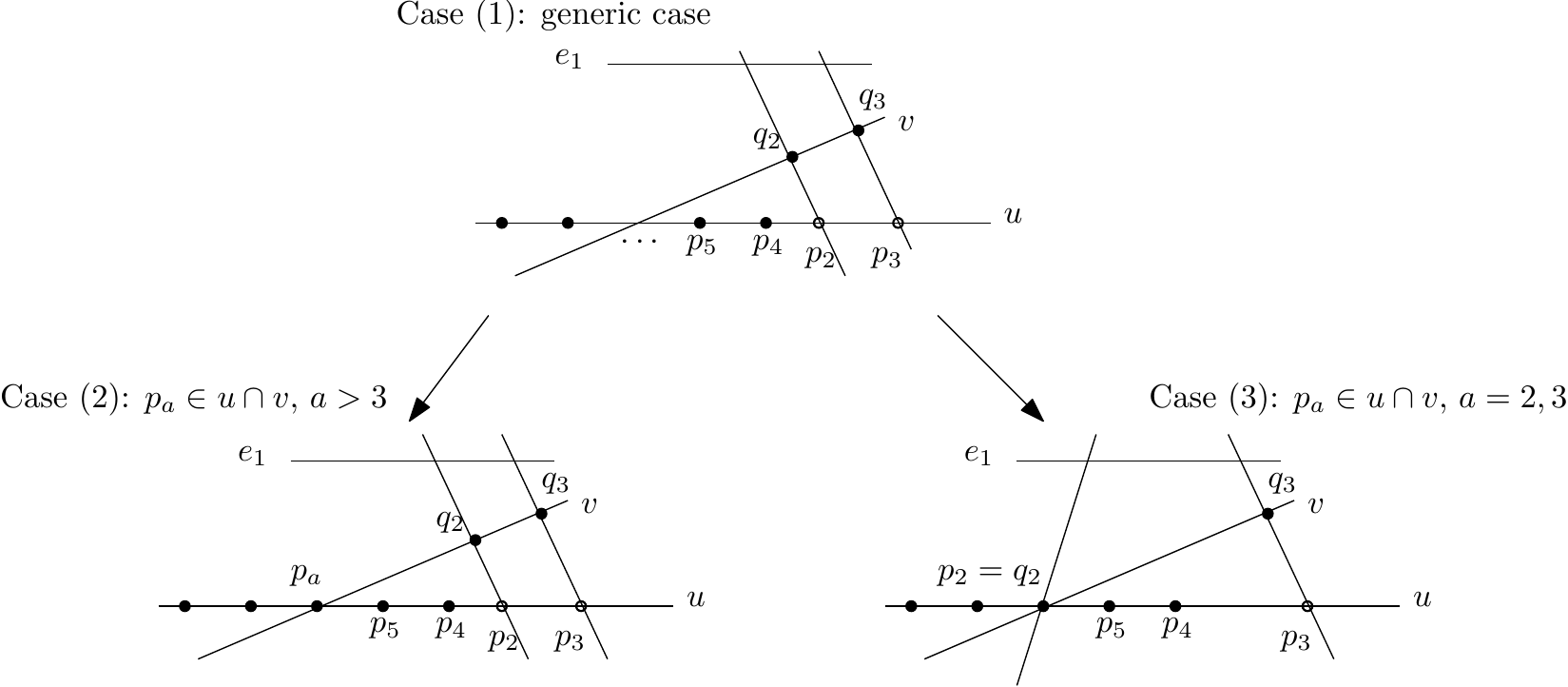}
  \caption{The three cases of Lemma \ref{l:existenceCurves}}
  \label{fig:threecases}
\end{figure}

\begin{proof}
  For the first assertion, note that before blowing up $u(p_i)$ and $q_i$, curves of the class $h-e_1$ intersect $u(\CP^1)$ and $v$ at exactly one point.  Therefore, if $h-e_1-e_i$ has a non-embedded representative, then either $u(p_i)=u(p_j)$ for some $i\neq j$, or some $q_i$ belongs to the $e_1$-curve, but neither is possible.  For $e_i$ to have a stable representative, some $u(p_i)$ and $q_j$ must coincide, which also contradicts our assumption.

  For the second part, case (1) corresponds to the case when none of $u(p_i)$ or $q_i$ lies at $u\cap v$, which is a single point; while case (2) corresponds to some $u(p_i)=u\cap v$, and case (3) happens when either $q_2$ or $q_3$ is $u\cap v$.
\end{proof}

 \begin{lma}\label{DkNM}

    Consider the family of rational surfaces with a divisor $(\mY_n,\Sigma_0)$.  For each fiber $(\mY_n)_x$, any holomorphic curve must have a non-negative coefficient of $h$ for all $x\in\Conf_{n-1}(\CP^1)$.  Moreover, there is a relatively ample line bundle $\mL\to \mY_n$, such that $\mL_x=\mathcal{O}(lh-(l-2)e_1-e_2-\cdots-e_n)$ over $(\mY_n)_x$ for all $l\ge n+5$.

       In particular, the family of complex structures are compatible with some K\"ahler form $\w_l$ whose class is dual to $lh-(l-2)e_1-\cdots-e_n$ for each $l\ge n+5$.
\end{lma}

\begin{proof}



  Take a fiber $(\mY_n)_x$, there is always a smooth representative of $h$ by taking a line in $F(1)$ avoiding a discrete set of points.  This proves the first assertion.  

  Let $D_A$ represent an irreducible rational curve of class $A$.  We construct a family of divisors as follows.  If $x\in\Conf_{n-1}(\CP^1)$ belongs to the first case in Lemma \ref{l:existenceCurves}, define
     \begin{equation} \label{e:Ddecomp}
              D_x:=(D_{h-e_{2}-e_3}+D_{e_{2}}+D_{e_{3}})+D_{h-\sum_{i=4}^{n} e_i}+\sum_{i=2}^3D_{h-e_1-e_i}+(l-4)(h-e_1-e_n)+(l-4)e_n,
     \end{equation}

  In case (2), replace $D_{h-e_{2}-e_3}$ by $D_{h-e_{2}-e_3-e_a}+D_{e_a}$; and in case (3), replace the second component $D_{h-\sum_{i=4}^{n-2} e_i}$ by $D_{h-\sum_{i=4}^{n-2} e_i-e_a}+D_{e_a}$.  
    By Lemma \ref{l:existenceCurves}, each irreducible component involved in $D_x$ exists and is unique. Therefore, this family of divisors defines a divisor $D\subset\mY_n$, giving a line bundle $\mL\to\mY_n$ (the degeneration of case (2) and (3) both give simple normal crossings).

  To prove $\mL\to\mY_n$ is relatively ample, we appeal to Nakai-Moishezon criteria by analyzing each fiber $\mathcal{L}_x\to(\mY_n)_x$, since our fibration is smooth and proper \cite[Theorem 1.7.8]{Positivity}.

  Note first that, any irreducible component involved in \eqref{e:Ddecomp} intersects $D_x$ positively.  Given any irreducible curve $\Sigma_x$ which is not one of the components of $D_x$, first suppose its $h$-coefficient is positive.  Then $\Sigma_x\cdot D_x\ge\Sigma_x\cdot (D_{h-e_{2}-e_3}+D_{e_{2}}+D_{e_{3}})>0$.  If $\Sigma_x$ has a vanishing $h$-coefficient, it must have at least one positive $e_i$-coefficient.  Since each $e_i$ has an embedded representative for each $i\ge2$, we have $i=1$ or $\Sigma=e_i$.  The latter case clearly intersects $D_x$ positively; and if $i=1$, $\Sigma$ must have the form $e_1-\sum_{i\in I\subset \{2,\cdots,n\}}c_ie_i$, where $c_i\ge0$ by positivity of intersections.  Since all $p_i$, $q_i$ are disjoint from $e_1$, $\Sigma_x=e_1$, which also intersects $D_x$ positively.  This shows that $D_x$ is an ample divisor and the rest of claim follows.


 \end{proof}

To get back to the topology of symplectomorphism groups, we need to consider a space of almost complex structures which have similar properties as those appearing in the family $\mY_n$.  Our goal is to show that the natural embedding the family $\mY_n$ into the universal family over this infinite dimensional almost complex structure is $1$-connected.

\begin{dfn}\label{JJ4}
Given a symplectic rational surface $(X,\w)$ of reduced $\mathbb{D}$ type, equipped with a basis $\{h,e_1,\cdots,e_n\}$ of $H_2(X)$ as before. $\mJ_{\w}^{s}$ is defined as the subspace of $\mJ_{\w}$ where each almost complex structure $J$

\begin{itemize}
  \item admits only holomorphic curves with non-negative $h$-coefficients;
  \item there is no subsets $I, K\subset\{2,\cdots n\}=I\cup K$, such that $e_1-\sum_{i\in I}e_i$ and $h-\sum_{k\in K}e_K$ each has an embedded $J$-representative.

\end{itemize}

\end{dfn}

To motivate this definition, the first conditio guarantees the simplicity of curve bubblings of curves that concerns us (Lemma \ref{l:bubble}).  We will prove the second condition prevents symmetries of the almost complex structures in Lemma \ref{4free}.  If $J\in\mJ_\w^2$ is integrable, the above conditions are equivalent to asking $J$ to be a blow-up complex structure of $F(1)$, and that not all such blow-ups belongs to the union of a line and the exceptional curve $e_1$.  The curve cone of almost complex structures $\mJ_\w^s$ has no essential difference from these integrable cases, as is shown below.  

We have the following basic properties of $\mJ_\w^s$.

\begin{lma}\label{l:codim}
	$\mJ_\w-\mJ_\w^s$ has at least codimension $4$ for any $n\ge5$.
\end{lma}

\begin{proof}
  Since we assume our symplectic form $\w$ is reduced, from \cite[Lemma 3.4]{Chen20}, a curve with negative $h$-coefficient must be of the shape $(a+1)e_1-ah-\sum_{i\in I}e_i$ for some $a\in\mathbb{Z}^+$.  Such almost complex structures contain only finitely many strata and each have at least codimension $4$ from \cite[Appendix B.1]{AP13}.

  If there is a decomposition $\{2,\cdots,n\}=I\cup L$ as in Definition \ref{JJ4}, since $|I|+|L|\ge n-1$, the existence of curves therein demands the codimension of such almost complex structures to have codimension at least 4 by Proposition \ref{stratum}.
\end{proof}

\begin{lma}\label{l:bubble}
   For any $J\in\mJ_\w^s$, each of $e_i, h-e_1-e_i$ is represented by at least one embedded $J$-holomorphic curve.  Also, $h-e_i-e_j$ and $2h-e_{i_1}-e_{i_2}-e_{i_3}-e_{i_4}-e_{i_5}$ each has at least one stable representative, whose irreducible components consist of curves in class $e_i$, $e_1-\sum_{i\in Q\subset \{2,\cdots,n\}}e_i$, $h-\sum_{i\in Q\subset [n]}e_i$ or $2h-\sum_{i\in Q\subset [n]}e_i$, where $[n]=\{1,\cdots,n\}$.

   Moreover, the stable representative of $e_1$ has a unique main component of type $e_1-\sum_{i\in Q}e_i$.  If $l,k\notin Q$, the stable representative of $h-e_l-e_k$ also has a unique main component of the shape $h-e_l-e_k-\sum_{i\in Q'}e_i$ where $1\notin Q'
   $.

\end{lma}

\begin{proof}
  The assertion of $h, e_{i\ge2}$ and $h-e_1-e_{i\ge2}$ follows from the definition of $\mJ^s_\w$, as well as the fact that $e_i$, $h-e_1-e_i$ have the smallest symplectic area among all exceptional curves (Lemma \ref{minemb}).

  For $h-e_i-e_j$ and $2h-e_{i_1}-e_{i_2}-e_{i_3}-e_{i_4}-e_{i_5}$, it follows from the definition of $\mJ_\w^s$ that if an irreducible component $[C]=ah-\sum_i b_ie_i$, then $2\ge a\ge0$. Also, $b_i\ge0$ unless $[C]=e_i$ when $i\ge2$ from the type $\mathbb{D}$ assumption.  If $a>0$, from the adjunction inequality, we know $c_1(C)-C^2\le 2$.  A straightforward computation shows that $a\ge b_i$, and the equality holds only when $b_i=1$ for all $i$.

  If $a=0$, again from a simple calculation based on adjunction inequality and the $\mathbb{D}$ type assumption that $m_1> m_2=\cdots=m_n$, one deduces that $b_1=-1$, $b_i=1$ for some $i\in Q\subset\{2,\cdots,n\}$ and $b_i=0$ for $i\notin Q$.

  For the second assertion, the $e_1$ stable components have the desired shape from the above calculation in the case of $a=0$.  For $h-e_l-e_k$ components, note that for any component with negative coefficients on $e_l$ or $e_k$, it must have a positive $e_1$ or $h$ coefficient (again it's a consequence of adjunction as above).  Since $l,k\notin Q$, both of them belongs to the unique component with positive $h$-coeffcient.  $1\notin Q'$ follows from the reduced condition and the symplectic area constraint.
\end{proof}

We consider next the moduli space of isomorphism classes of almost K\"ahler structures with an almost K\"ahler form $\w$, i.e. the space of compatible almost complex structure modulo reparametrization by symplectomorphisms.  The construction is given by \cite{FS88} Corollary 3.2, where Fujiki and Schumacher denotes $\mC^k_{a,\w}$ as the $H^k$-completion of smooth almost complex structures compatible with $\w$, then the moduli space $\mM^k_{a,\w}:=\mC^k_{a,\w}/Symp^{k+1}(X,\w)$, where $Symp^{k+1}(X,\w)$ is the $H^{k+1}$-completion of the smooth symplectomorphism group. For the moduli space of smooth almost complex structures, one can take the inverse limit on $k,$ and endow the inverse limit topology of the moduli space  $\mM_{a,\w}=\mJ_{\w}/Symp(X,\w)$ (see \cite[Section 6]{FS88} for further details). We can restrict the group action to make a moduli space $\mJ_{\w}/Symp_h(X,\w)$ since the two groups only differs by a finite extension.  

\begin{lma}\label{4free}
Given a reduced type $\mathbb{D}$ positive rational surface $(X,\w)$ with $\chi(X)>8$, the action of $Symp_h(X,\w)$ on $\mJ_{\w}^s$ is free, hence $\pi_i(Symp_h(X,\w))=\pi_{i+1}(\mJ_{\w}^s)/Symp_h(X,\w)$ for $i=0,1$.

\end{lma}

\begin{proof}

From our assumption on $\mJ_\w^s$, we know $I\neq\{2,\cdots,n\}$ as defined in \ref{JJ4}.
We claim that $\{2,\cdots,n\}\setminus I$ contains at least three elements.  Otherwise, assume $\{2\}$ is the only element that does not belong to $I$, then the moduli space of $h-e_2$ forms a two-dimensional family.  The only possible degenerations are rigid configurations $(h-e_2-e_i)+e_i$ or $(h-e_1-e_2)+(e_1-e_3-\cdots-e_n)+\sum_{i=3}^n e_i$ from Lemma \ref{l:bubble}.  Therefore, $h-e_2$ has an embedded representative, contradicting the second assumption of $\mJ_\w^s$.

Assume $\{2,3\}\cup I=\{2,\cdots,n\}$. The stable representative of $h-e_2-e_3$ must be $(h-e_2-e_3-\sum_{k'\in K'}e_{k'})+\sum_{k'\in K'}e_{k'}$ from Lemma \ref{l:bubble}.  Here $1\notin K'$
due to the lack of $h-e_1-e_2-e_3$ curves by area constraints.  But this again contradicts our assumption on $\mJ_\w^s$.

We therefore assume $\{2,3,4\}$ are disjoint from $I$, and that $h-e_2-e_3-e_4$ does not admit a stable representative.
Consider the stable representatives of $h-e_2-e_3$ and $h-e_2-e_4$.  Again by Lemma \ref{l:bubble} and a similar argument as above, their stable representatives each have a component
of the shape $h-e_2-e_3-\sum_{l\in L_0'}e_l$ and $h-e_2-e_4-\sum_{l\in L_1'}e_l$.  Moreover, $h-e_1-e_2$ has an embedded representative as a minimal area curve.  These three components each intersects $e_2$ once.

Assume $\sigma\in Symp_h(X,\w)$ which fixes some $J\in\mJ_\w^s$.  From the analysis above, $\sigma$ must fix the $e_2$-curve pointwise.  Furthermore, the main component of $h-e_3-e_4$ also has the form $h-e_3-e_4-\sum_{l\in L_2'}e_l$.  We then consider the intersections of the embedded sphere $C$ of class $h-e_1-e_2$ with three curves: $e_1-\sum_{i\in I}e_i$, $e_2$, and the main component of $h-e_3-e_4$.  These are again three distinct points on $C$ which were fixed by $\sigma$.  This implies $h-e_1-e_2$ is also fixed pointwise by $\sigma$, hence $\sigma$ fixes the whole tangent space of the intersection between $h-e_1-e_2$ and $e_2$.  This implies $\sigma$ must be a trivial action.

\end{proof}

\begin{rmk}\label{rem:dropJs}
  The second assumption on $\mJ^s_\w$ cannot be dropped: one may easily construct an $S^1$ action on a type $\mathbb{D}$ rational surface using Karshon's description, by blow-ups on the exceptional divisor $e_1$ and a line $h$.  The corresponding equivariant almost complex structures don't even have finite stabilizers.  The first assumption prevents a large amount of technical difficulty, but the authors do not know whether it is necessary.
\end{rmk}

Now we are ready to establish the following $\alpha$-map.

\begin{prp}\label{l:alpha}
    There is a continuous map $\alpha: \Conf_{n-1}(\CP^1)\to\mJ_\w^s/Symp_h(X_k,\w_l)$, where $PD[\w_l]=lh-(l-2)e_1-e_2-\cdots-e_n$ for $l\ge n+5$.
\end{prp}

\begin{proof}

Consider the family of rational surfaces polarized by $\mL$ in Definition \ref{d:qn} and \ref{DkNM}.  The fiberwise projective embedding of $(\mY_n)_{\underline{p}}$ given by the polarization (after raising $\mL$ to a certain power) induces a biholomorphism between a fiber $(\mY_n)_{\underline p}$ with an integrable complex structure of some symplectic manifold $(X,\w_{\underline p})$.  From the definition of $\mL$, $\w_{\underline p}$ is symplectomorphic to $(X,\w_l)$ after a rescaling of the form, therefore associates a map $\alpha: \Conf_{n-1}(\CP^1)\to \mJ_{\w_l}/Symp_h(X,\w_l)$.

 Note that $\mJ_{\w_l}^s$ is invariant under $Symp_h(X,\w_l)$ action, since Definition \ref{JJ4} imposes only constraints on the curve cones.  The $\alpha$-image clearly satisfies the first condition of Definition \ref{JJ4}, verified by Lemma \ref{DkNM}, and the second condition is also clear from the choice of $q_i$.

\end{proof}

Our last ingredient of the proof is a map 
$$\beta: \mJ^s_\w/Symp_h(X,\w)\to\Conf_{n-1}(\CP^1)/PSL(2,\CC).$$  Taking an almost complex structure $J\in \mJ^s_\w$, we consider the stable representative $\Sigma_b$ of the class $[\Sigma_b]=2h-e_1-e_2-e_3-e_4-e_5$.  There are three possible types of bubblings of $\Sigma_b$ if it is not an embedded curve from Lemma \ref{l:bubble}:

\begin{enumerate}[(i)]\label{dec2h}
  \item $(2h-e_1-e_2-e_3-e_4-e_5-\sum_{i\in I_0} e_i)+\sum_{i\in I_0} e_i$;

  \item $(h-\sum_{i_0\in I_0}e_{i_0})+(h-e_1-\sum_{i_1\in I_1}e_{i_1})+\sum_{i\in I_2}e_i$;

  \item $(e_1-\sum_{i_0\in I_0}e_{i_0})+(h-e_1-\sum_{i_1\in I_1}e_{i_1})+(h-e_1-\sum_{i_2\in I_2}e_{i_2})+\sum_{i_3\in I_3}e_{i_3}$.
\end{enumerate}

In the above decompositions $|I_1|=0$ or $1$ from the area restrictions.  All index sets $I_0$, $I_1$, $I_2$ and $I_3$ are subsets of $\{2,\cdots,n\}$.  In each case, we will call the irreducible component in the first parenthesis our \textbf{main component}.  For convenience, we will denote the $J$-representative of a class $C$ as $J(C)$ below.

Note that it is (only) possible for a stable curve listed above to have a multiply-covered component, which is in case (iii) when $I_1=I_2$.  In that case, the second and third components should be combined.



  Consider the stable curves consisting of the union of two embedded components: $h-e_1-e_i$ and $e_i$, for each $i\ge2$.  All of them are irreducible and intersect the main component of $\Sigma_b$ at a unique point, denoted as $p_i$, yielding a configuration in $\Conf_{n-1}/PSL(2,\CC)$.


\begin{prp}\label{prp:beta}
    The above procedure defines a  continuous map $\beta: \mJ^s_\w/Symp_h(X,\w)\to\Conf_{n-1}(\CP^1)/\PSL(2,\CC)$ by sending $\beta: J\mapsto [p_i]$.  
\end{prp}

\begin{proof}

   Note first that the image $\beta(J)$ is also the image of the forgetful map from $\ov\eM([\Sigma_b]; \gamma_2,\cdots, \gamma_n; J)$ to its marked point configuration of the domain curve, where $\ov\eM([\Sigma_b]; \gamma_2,\cdots, \gamma_n; J)$ is the compactified moduli space of $J$-holomorphic curves of class $[\Sigma_b]$ with $n-1$ marked points evaluated to the cycles $\gamma_i$ defined by the union of $h-e_1-e_i$ and $e_i$.  This map is clearly smooth in the open and dense subset of $\mJ_\w$ when $J$ admits an embedded representative of $\Sigma_b$.  We call this subset the \textbf{generic locus}.  Our aim is to prove $\beta$ extends continuously to the stable locus by analyzing the configuration of marked points.
   Even though the point configuration is well-defined only up to a $\PSL(2,\CC)$, the reparametrization group is irrelevant when we discuss continuity near a fixed configuration when we fix the parametrization of $\CP^1$.  Therefore, we will use a notation $p_i:=\beta(J)_i$ as the $i^{th}$ component in a configuration without loss of generality.


   We divide the rest of almost complex structures into types according to the bubblings of $2h-e_1-e_2-e_3-e_4-e_5$ listed above, and will use $J(A)$ to denote the $J$-holomorphic curves of class $A$ when such a curve is unique.

   Take a sequence of generic $\{J_r\}_{r\to\infty}$ approaching $J$ of type (i).  By the definition of Gromov convergence, we have $u_r:(\CP^1; p_2^r,\cdots,p_n^r)\to (X, J_r)$ converges in $C^\infty$-sense outside a small neighborhood of $\{p^r_2,\cdots,p^r_n\}\subset\CP^1$ to $u_\infty: (\CP^1; p_2,\cdots,p_n)\to (X,J)$ restricted to the complement of a neighborhood of $\{p_2,\cdots,p_n\}$.  Here $u_\infty$ is precisely the main component of $J([\Sigma_b])$.  Therefore, the configuration $(\CP^1; p_2^r,\cdots,p_n^r)$ clearly converges to $(\CP^1; p_2,\cdots,p_n)$ after shrinking the neighborhoods.

   The type (ii) case is similar: if $I_1=\emptyset$, the component of $h-e_1$ can be ignored.  Otherwise, $I_1$ must contain exactly one element $k$ by the energy constraint.  This means the bubble at $p_k$ is now $h-e_1-e_k$ instead of $e_k$ in the first case.  Regardless, the formation of the bubble of $h-e_1-e_k$ is obtained by removing a shrinking neighborhood of $\beta(J_r)_k$ on the $J_r([\Sigma_b])$.  Therefore, $\beta(J_r)_k\to\beta(J)_k$ when $r\to \infty$.

   For type (iii), if either of $I_1$ or $I_2$ is empty, we'll simply ignore that component since the bubbling happens away from the configuration.

   If $I_1=\{k\}$, the intersection between $J_r(h-e_1-e_k)$ or $J_r(e_k)$ and the irreducible $J_r([\Sigma_b])$ gives $\beta(J_r)_k$.  From the same argument as in type (ii), one may remove from the domain $\CP^1$ a small neighborhood of $\beta(J_r)_k$, then the $J_r([\Sigma_b])$ converges to $J(e_1-\sum_{i_0\in I_0}e_{i_0})$ uniformly outside a slightly larger neighborhood.  As the neighborhood shrinks, this establishes the convergence $J_r([\Sigma_b])_k\to J(e_1-\sum_{i_0\in I_0}e_{i_0})_k$.

   Lastly, since only the convergence in the main component is relevant to our discussion, which cannot be multiply-covered, the proof goes through without any changes when non-simple components are involved.  The independence of $\beta$ under the $Symp_h(X,\w)$ action is obvious.
\end{proof}

\begin{lma}\label{l:composition}
   Let $(X,\w_l)$ be a type $\mathbb{D}$ rational surface with $PD(\w_l)=lh-(l-2)e_1-e_2-\cdots-e_n$ for $l\ge n+5$.  Then $\beta\circ\alpha:\Conf_{n-1}(\CP^1)\to\Conf_{n-1}(\CP^1)/\PSL(2,\CC)$ induces a surjective map of the corresponding fundamental groups.
\end{lma}

\begin{proof}

  Considering the complex structure $\alpha(\underline p)$, we have a natural ruling given by curves of homology class $h-e_1$.  This ruling has $(n-1)$ special fibers consisting of $J_{\alpha(\underline p)}(h-e_i)$ and $J_{\alpha(\underline p)}(e_i)$.  From this perspective, $\beta\circ\alpha(\underline p)$ is given by the intersections between the main component of $J_{\alpha(\underline p)}(2h-e_1-\cdots-e_5)$
  and these special fibers.  Also, $[u(\underline p)]\in\Conf_{n-1}(\CP^1)/PSL(2,\CC)$ can be regarded as the intersection between $u$ and the special fibers.

  On the other hand, for any two sections $S_0$ and $S_1$ of this ruling, one can define a biholomorphism $\mu_{0,1}:S_0\to S_1$, which sends $S_0\ni p_0\mapsto p_1\in S_1$ when $p_0, p_1$ belong to the same fiber.  In particular, this defines a biholomorphism from the main component of $J_{\alpha(\underline p)}(2h-e_1-\cdots-e_5)$ to the $\alpha$-defining curve $u$ as in \ref{d:qn}.    Since the marked points of both $u$ and $J_{\alpha(\underline p)}(2h-e_1-\cdots-e_5)_{main}$ intersects the same ordered set of special fibers, this means $\beta\circ\alpha(\underline p)=[\underline p]\in \Conf_{n-1}(\CP^1)/\PSL(2,\CC)$ is the canonical projection, hence the desired surjectivity is obvious.

\end{proof}

\begin{prp}\label{prp:surjective}
    In the situation of \ref{l:composition}, we have a surjective group homomorphism $\pi_0(Symp_h(X,\w))\twoheadrightarrow PB_{n-1}(S^2)/\mathbb{Z}_2$.
\end{prp}

\begin{proof}
  From Lemma \ref{l:composition}, we have 
  $$\pi_1(\mJ^s_\w/Symp_h(X,\w))\twoheadrightarrow \pi_1(\Conf_{n-1}(\CP^1)/PSL(2,\CC))\cong PB_{n-1}(S^2)/\mathbb{Z}_2.$$  

  On the other hand, from Lemma \ref{l:codim} we have 
  $$\pi_1(\mJ^s_\w/Symp_h(X,\w))=\pi_1(\mJ_\w/Symp_h(X,\w)).$$ 

  From the exact sequence associated to $Symp_h(X,\w)\to\mJ_\w\to \mJ_\w/Symp_h(X,\w)$, $\pi_0(Symp_h(X,\w))\cong \pi_1(\mJ_\w/Symp_h(X,\w))$, hence the claim.

\end{proof}

Therefore, combining with the Hopfian property \ref{Hopfian}, the two surjective maps \ref{Dsurjbraid} and \ref{prp:surjective} concludes Theorem \ref{genSMC} for $[\w_l]=lh-(l-2)e_1-e_2-\cdots-e_n$.  From Lemma \ref{l:dense}, we conclude \ref{genSMC} for all type $\mathbb{D}$ classes which are $c_1$-positive.

\section{Generation of Lagrangian Dehn twists and parametrized Gromov-Witten invariants} 
\label{sec:kroheimer_mcduff_fibration_ball_swappings_and_dehn_twists}

 We first give a technical overview of the proof of Theorem \ref{t:generation}.  The proof is divided into two steps.  Section \ref{sub:meridians_of_} contains the main setup.  Given a symplectic form $\w$ of type $\mathbb{D}$ or $\mathbb{E}$, we consider $\mA_\w$ as a subset of $\mA_{\w'}$ for some type $\mathbb{A}$ perturbation $\w'$.  What we need to clarify is the kernel of \eqref{e:surjective} since we know it is not an isomorphism in general.  We will see that all Torelli elements come from $\pi_1(\mA_{\w})$, and we already knew that $\pi_1(\mA_{\w'})$ does not give any non-trivial Torelli elements.  We prove that the extra generators of $\pi_1(\mA_{\w})$ come from the deletion of certain $U_D$ from $\mA_{\w'}$, where such $D$-classes are $\w'$-symplectic and $\w$-Lagrangian.  

 The only thing left to verify is the connection between these meridians and Lagrangian Dehn twists.  We construct a loop around $U_D$ out of Dehn twists.  The main technical argument in Section \ref{sub:a_parametrized_gromov_witten_invariant} is the computation of a parametrized Gromov-Witten using symplectic field theory, which proves that the loop we constructed is indeed homotopic to a meridian.

\subsection{Meridians of $U_D$ and Lagrangian Dehn twists} 
\label{sub:meridians_of_}


 We start by considering the  following piece of long exact sequence associated to Kronheimer-McDuff fibration with a base point:

\begin{equation}\label{e:SES}
    \pi_1(\Diff_0(X), id)\xrightarrow{i_\w} \pi_1(\mS_\w, \w)\xrightarrow{j_\w} \pi_0(Symp_h(X,\w))\xrightarrow{k_\w} \pi_0(\Diff_0(X))
\end{equation}

We will omit the dependence on $\w$ for these $i, j, k$ maps when it is clear from the context.  We first notice the following elementary fact.

\begin{lma}\label{l:connectingMap}
  The image of a loop $[\{\w_t\}_{t\in[0,1]}]\in\pi_1(\mS_\w,\w)$ under the connecting map $j$ in \eqref{e:SES} can be given by Moser's method.  More specifically, $\w_t$ yields a family of diffeomorphisms $f_t:X\to X$ such that $(f_t)_*\w_0=\w_t$, and $j([\{\w_t\}])=[f_1]$.
\end{lma}

\begin{proof}
  The family of diffeomorphisms $\{f_t\}_{t\in[0,1]}$ given by Moser's technique is precisely a lift of $\w_t$ to $\Diff_0(X)$ over the (based) Serre fibration $(\Diff_0(X),id)\to(\mS_\w,\w)$, therefore, the statement simply unwraps the definition of the connecting map of homotopy exact sequences in our situation.
\end{proof}

We would like to point out two facts regarding Lemma \ref{l:connectingMap}.

\begin{itemize}
  \item When $X$ is a rational surface, it is a consequence of Theorem A.1 in \cite{LLW15} that $j$ is onto, because the image of $k$ is trivial in \eqref{e:SES}.
  \item Although it is an abstract property of Serre fibrations, we point out that the independence of various choices of the class $j([\{\w_t\}])=[f_1]\in\pi_0(Symp_h(X))$ is also a consequence of a family Moser's technique.  For example, two different primitives $\lambda_1(t)$ and $\lambda_1(t)$ satisfying $\frac{d}{dt}\w(t)=d\lambda_i(t)$ can be connected by a linear interpolation, which yields a smooth two-dimensional family of Moser vector fields that provides the needed isotopy.  We leave the details for interested readers.
\end{itemize}

Let $L\subset (X,\w)$ be a Lagrangian sphere with a chosen Weinstein neighborhood $\eW_L$.  We recall a construction from \cite[Proposition 1.1]{Sei08}, which gives an explicit smooth isotopy from $\tau_L^2$ to the identity, where the isotopy is supported in $\eW_L$ that supports $\tau_L$.  Consider the standard cotangent bundle $(T^*S^2,\w_{std})$, and denote its zero section as $Z$.   There is a symplectic deformation $\w^s$ of $\w_{std}$, which admits a circle action $\mu^s_\lambda, 0\le \lambda\le 1$.  Take a cut-off function $h(x)$ such that $h=0$ when $x\gg0$ and $h'(0)=\frac{1}{2}$.  For technical reasons, we take 
\begin{equation}\label{e:identity}
     h(x)=\frac{x}{2}\text{ when }|x|<\epsilon
\end{equation}
 for a small number $\epsilon$. Then $\tau_Z^2$ is defined as $\mu^0_{4\pi h'(||v||)}$, where $||v||$ is the magnitude of the cotangent vector in the round metric.  The desired smooth isotopy can be chosen as $\phi_s:=\mu^{1-s}_{s\cdot 4\pi h'(||v||)}$.  For each $\mu_\lambda^s$, $Z$ is a fixed point set, therefore, throughout this isotopy, $Z$ is point-wise fixed. By shrinking the support of $h(x)$ if necessary, both $\tau_Z$ and the above isotopy can be implanted inside $\eW_L$.

We summarize the above discussions as follows:

\begin{lma}\label{l:fixedIsotopy}
    $\tau_L^2$ is smoothly isotopic to identity through a compactly supported isotopy $\phi_t$ in any Weinstein neighborhood $\eW_L$ that supports $\tau_L^2$.  Moreover, this isotopy fixes the zero section pointwise.
\end{lma}

What plays a central role in our proof is a loop of compatible pairs $(\w_t^L,\gamma_t^{L,\w})$.  Take the isotopy $\phi_t^L$ from $\tau_L^2$ to $id$ as above, it induces a loop of cohomologous symplectic forms 
$$\w_t^L:=(\phi_t^L)_*(\w),$$ where $\w_1^L=\w_0^L=\w$.  It is tautological from Lemma \ref{l:connectingMap} that $j_\w[\{\w_t^L\}]=[\tau_L^2]\in\pi_0(Symp_h(X,\w))$.

We now construct a smooth loop of almost complex structures $\gamma^{L,\w}_t$ associated to any Lagrangian $2$-sphere $L$ in a closed symplectic four manifold $(X,\w)$.
Take an arbitrary almost complex structure $J_0$ compatible with $\w_0^L=\w$, and define $J'_t:=(\phi_t^L)_*(J_0)$.  This usually does not form a loop, but $J'_1$ is compatible with $\w_0^L$.  Therefore, one may concatenate the above path of $\{J'_t\}_{t\in[0,1]}$ with another arbitrary path in $\mJ_{\w_0^L}$, which connects $J'_1$ to $J_0$, and is constant outside the Weinstein neighborhood $\eW_L$.  Note that the choices of the second half of the loop form a contractible space.

  The above construction gives a loop of almost complex structures, denoted as $\{\gamma_t^{L,\w}\}_{t\in[0,1]}$.  Concatenating a constant path to $\w_t^L$, we may assume that $\gamma_t^{L,\w}$ is compatible with $\w_t^L$ for every $t$.  Denote this loop as $\gamma^{L,\w}:=\{\gamma^{L,\w}_t\}_{t\in[0,1]}$, and call $J_0$ the \textbf{base point} of $\gamma^{L, \w}$.  

  \begin{rmk}\label{rem:welldefinedgamma}
    As is readily seen from our definition, $\gamma_t^{L,\w}$ is only well-defined up to a variety of choices.  

    \begin{enumerate}
      \item Choices of concatenations from $J_1'$ to $J_0$.  We already pointed out the choices of such paths form a contractible space.  It is also clear from our construction that, one may choose this path of almost complex structures which is constant outside of $\eW_L$.  Therefore, as long as we fix such a choice of inside $T^*S^2$ for each base point, $\gamma_t^{L,\w}$ is well-defined.

      \item Choices of base points $J_0$.  When $J_0$ varies smoothly along some path $J_0^s$, there is also a choice of a family of paths connecting $(J_1^s)'$ to $J_0^s$, such that the resulting $\gamma^{L,\w}$ varies smoothly.  Indeed, for any choice of the simplicial family $J_0^\Delta$, one may choose a $\Delta$-family of paths varying smoothly, which connect $(J_1^\Delta)$ to $J_0^\Delta$, by the contractibility of $\mJ_\w$.  But we make no claim that there is a canonical choice of such a path for each $J_0\in\mJ_\w$, which is much stronger. \footnote{If one only requires continuity for this family, a canonical choice can be obtained as follows.  $\phi^L$ induces an automorphism of $\mJ_\w$, which is homotopic to identity by contractibility of $\mJ_\w$.  Composing the inverse of the contraction from the identity map, this induces a canonical choice of continuous path for every $J_0$.  From this perspective, what we did above is to create a smooth isotopy between two simplicial maps $\sigma_0,\sigma_1:\Delta\to \mJ_\w$ by smoothing theory of the finite-dimensional objects.  The canonical choice is hard because one needs to smooth a mapping of infinite dimension.} 

      In particular, the change of $J_0$ will not affect the free homotopy class of $\gamma^{L,\w}$, but only its base point.  Sometimes we simply write $\gamma^L$ when the context is clear.
    \end{enumerate}

  \end{rmk}

We have the following easy property for $\w_t^L$ and $\gamma^{L,\w}$.

\begin{lma}\label{l:loopDT}
   For each Lagrangian sphere $L\subset (X,\w)$, the above constructed $\{\w_t^L\}_{t\in[0,1]}$ satisfies $j_\w([\{\w_t^L\}])=[\tau_L^2]$.

   As a result, given any almost complex structure $J_0$ compatible with $\w$, under the map $\pi_1(\mA_{\w},J_0)\xrightarrow{\sim}\pi_1(\mS_{\w},\w)\xrightarrow{j_\w}\pi_0(Symp_h(X,\w))$, $[\gamma^{L,\w}]$ has an image  $[\tau_L^2]$.

\end{lma}

\begin{proof}

    From Lemma \ref{rem:AtoS} (2), the isomorphism $\pi_1(\mA_{\w},J_0)\xrightarrow{\sim}\pi_1(\mS_{\w},\w)$ sends exactly the homotopy class of $\{J_t\}_{t\in[0,1]}$ to that of $\{\w^L_t\}_{t\in[0,1]}$.  The claim then follows from Lemma \ref{l:connectingMap}.
\end{proof}

Although $\gamma^{L,\w}$ is defined for all symplectic four manifold, we will stick to the case when $X$ is a rational surface in the rest of the section.

Back to our previous setup when $L\subset\eW_L\subset (X,\w)$.   \eqref{e:identity} guarantees that $\tau_L^2=id$ in a small neighborhood of $L$.  One can then choose a deformation $\w^r$ with $\w^0=\w$, such that $L$ is $\w^r$-symplectic for $1\ge r>0$, and is supported in the subset of $\eW_L$ where $\tau_L^2$ is identity (an easy way to find such a family of $\w^r$ is to perform a symplectic cut at the level set $T^*_{\epsilon\cdot r}S^2$ for a small $\epsilon>0$).  When the deformation is sufficiently small, $\gamma^{L,\w}_t$ \emph{tames} $\w^r$ for all $s, t\in[0,1]$.   By Lemma \ref{rem:AtoS}(1), $\gamma^{L,\w}$ can be considered as a loop in $\mA_{\w^1}$.  In other words, $\gamma^{L,\w}\subset \mA_{\w^1}\cap\mA_{\w}$.

\begin{dfn}\label{d:wL}
  We define $\w_L:=\w^1$ as above.  In particular, $\w_L=\w$ outside of $\eW_L$ and $L$ is $\w_L$-symplectic.
\end{dfn}

We will use the notation $\gamma^{L,\w_L}_t$ to denote exactly the same loop of almost complex structure as $\gamma^{L,\w}_t$, but the change in superscript indicates we regard this loop as one in $\mA_{\w_L}$ and $\mA_\w$, respectively.  We note the following simple fact:

\begin{lma}\label{l:neighborhood}
   $\gamma^{L,\w_L}$ can be isotoped, as a free loop, to some $\wt\gamma^{L,\w_L}$ contained in a decentered tubular neighborhood $V\setminus U_{[L]}$ of $U_{[L]}$ by moving the base point.  Moreover, $\wt\gamma^{L,\w_L}$ is the boundary of a disk $\sigma: D^2\to \mA_{\w_L}$ in $V$, where $J_{\sigma(\zeta)}$ is independent of $\zeta\in D^2$ outside $\eW_L$.
\end{lma}

\begin{proof}
   Assume $J_0$ is the base point of $\gamma^{L,\w_L}$.  Since it lies in $\mA_{\w_L}\cap\mA_{\w}$, none of the almost complex structures on this loop admits a rational curve of class $[L]$, i.e. $\gamma^{L,\w_L}\cap U_{[L]}=\emptyset$.


   Since $U_{[L]}$ is a codimension 2 submanifold in $\mA_{\w_L}$, one may find a homotopy of base points $\{J^s_0\}_{s\in[0,1]}$, where $J^0_0=J_0$, and the homotopy is supported within a neighborhood of $L$ where $\tau_L^2$ is identity, so that $L$ is $J_0^1$-holomorphic (hence $J_0^1\in U_{[L]}$ by definition), and none of $J_s$, $s<1$ admits a curve in class $[L]$.  Remark \ref{rem:welldefinedgamma} therefore gives a family of loops $\{\gamma_{t,s}^{L,\w_L}\}_{s,t\in[0,1]}$ for each $s$, while $\gamma_{0,s}^{L,\w_L}=J_0^s$.  Note that when we apply our construction to $J^1_0\in U_{[L]}$, the sphere $L$ is a $\gamma_{t,1}^{L,\w_L}$-holomorphic curve for all $t\in[0,1]$.  This follows from Lemma \ref{l:fixedIsotopy}, since $\phi^L_t$ fixes $L$ pointwise for every $t$.  Therefore,  $\gamma_{t,1}^{L,\w_L}$ is a loop lying entirely in $U_{[L]}\subset\mA_{\w_L}^2$.  Therefore, $\wt\gamma^{L,\w_L}=\gamma_{t,\epsilon}^{L,\w_L}$ lies inside $V\setminus U_{[L]}$ when $\epsilon$ is sufficiently small.

   Furthermore, the loop $\gamma_{t,1}^{L,\w_L}$ constructed above can further be contracted to a point inside $U_{[L]}$.  Since each $\gamma_{t,1}^{L,\w_L}$ admits $L$ as a pseudo-holomorphic curve, one may first deform this loop of almost complex structures on $L$ so that $\gamma_{t,1}^{L,\w_L}|_L$ is constant as $t$ varies, then deform the resulting loop of almost complex structures to constant while fixing their restrictions on $L$.  All involved spaces of choices are contractible, and the deformation is entirely supported in $\eW_L$, while preserving $L$ as a pseudo-holomorphic curve.  Concatenating with the isotopy from $\gamma_{t,\epsilon}^{L,\w_L}$ to $\gamma_{t,1}^{L,\w_L}$, this gives a topological disk that bounds $\wt\gamma^{L,\w_L}$, which can be smooth to the desired $\sigma$.
\end{proof}

From now on, we assume the base point of $\gamma^{L,\w_L}$ has the property required in Lemma \ref{l:neighborhood}, hence the loop lies entirely inside some tubular neighborhood $V$ of $U_{[L]}$.  Since $U_{[L]}$ forms a connected codimension 2 submanifold in $\mA_{\w_L}$, its decentered tubular neighborhood $V\setminus U_{[L]}$ has an associated $S^1$ bundle.  The corresponding exact sequence of homotopy groups has the form
\begin{equation}\label{e:nbhLES}
       \cdots\longrightarrow\pi_1(S^1)\longrightarrow \pi_1(V\setminus U_{[L]})\longrightarrow \pi_1(U_{[L]})\longrightarrow\cdots
\end{equation}

FromLemma \ref{l:neighborhood}, $[\gamma^{L,\w_L}]\in\pi_1(V\setminus U_{[L]})$ has a vanishing image in $\pi_1(U_{[L]})$, hence itself is an image of some element in $\pi_1(S^1)$.  $U_{[L]}$ has a natural coorientation from the appendix of \cite{Abr98}, fixing a isomorphism $\pi_1(S^1)\cong\ZZ$, hence $[\gamma^{L,\w_L}]$ gives a well-defined integer.  We call this integer the \textbf{linking number of $\gamma^{L,\w_L}$}, which is equivalent to the intersection number between a disk $\sigma\subset V$ with $U_{[L]}$ when $\partial\sigma=\gamma^{L,\w_L}$.  If a free loop in $V$ has linking number $\pm1$ with $U_{[L]}$, we call it a \textbf{meridian} of $U_{[L]}$.  Since $U_{[L]}$ is connected by Lemma \ref{l:connected}, we have

\begin{lma}\label{l:meridianunique}
  If $\gamma$ and $\gamma'$ are both meridians of $U_{[L]}\subset V$, then they are homotopic as free loops up to a change of orientation.
\end{lma}



  The following result relies on a computation of a parametrized Gromov-Witten type invariant, whose proof will be postponed to the next section.

  \begin{prp}\label{c:meridian}
     $\gamma^{L,\w_L}$ is a meridian of $U_{[L]}\subset \mA_{\w_L}$.
  \end{prp}

We are now ready to prove Theorem \ref{t:generation} with the help of Proposition \ref{c:meridian}.  Consider $\w\in S\mR_n(X)$.  Perturb $\w$ to some $\w_0\in S\mR_n(X)$, where all $c_i>c_{i+1}$ and $c_1+c_2+c_3<1$ for $\w_0$.  This implies the Lagrangian root lattice $\mL_\w$ is empty, or equivalently, $\w_0(D)>0$ for all simple and positive roots in $\Gamma_{K_0}$ (see Lemma \ref{l:-2class}).  In particular, all simple and positive roots in $\Gamma_{K_0}$ can be represented as embedded $\w_0$-symplectic spheres by Lemma \ref{l:-2class}.

We will fix a coarse filtration $\Pi$ for $\w$ and $\w_0$, defined in the same way as in the proof of Proposition \ref{p:universal}.  The main question at hand is to compare $\mA_\w^0\cup\mA_\w^2$ and $\mA_{\w_0}^0\cup\mA_{\w_0}^2$.  

Since $\mA^0_\w=\mA^0_{\w_0}$ by \eqref{e:A0}, we only need to compare the second stratum by decomposing them into $U_D$ for $D\in\Gamma_{K_0}$ running over all positive and simple roots, with the additional assumption that $D\in\mD(E,\w)$ for some $E\in\Pi$ as in Lemma \ref{estable2}.  For convenience, we will consider $U_D=\emptyset$ when $D\notin\mD(E,\w)$ for any $E\in \Pi$. 

  Similar to Lemma \ref{l:commonDense}, if a simple or positive root $D$ pairs $\w$ positively, one can choose a common open dense set for $U_{\w_0,D}$ and $U_{\w,D}$ since $\w_0(D)>0$ always holds. Hence, $\mT_{\w,\w_0}$ can be taken as the union of all such dense sets over positive and simple roots satisfying $\w(D)>0$.  Therefore, we have an inclusion $\iota_{\omega,\omega_0}: \mA_\w^0\cup\mT_{\w,\w_0}\hookrightarrow(\mA_{\w_0}^0\cup\mA_{\w_0}^2)\setminus(\bigcup_{D\in\mL_\w}U_{\w_0,D})$.  The complement of the $\iota_{\w,\w_0}$-image is a meager set in $\bigcup_{\{D_i:\w(D_i)>0\}}U_{\w_0, D_i}$.  Hence a similar codimension argument as Lemma \ref{l:complementfund} implies two isomorphisms as follows

\begin{lma}\label{l:fundamentalgroup}
  The following two homomorphisms induced by inclusions are both isomorphisms
       $$(\iota_{\omega})_*:\pi_1(\mA_\w^0\cup\mT_{\w,\w_0},\ast)\xrightarrow{\sim}\pi_1(\mA_\w,\ast),$$
       $$(\iota_{\omega,\omega_0})_*: \pi_1(\mA_\w^0\cup\mT_{\w,\w_0},\ast)\xrightarrow{\sim}\pi_1\left((\mA_{\w_0}^0\cup\mA_{\w_0}^2)\setminus(\bigcup_{D\in\mL_\w}U_{\w_0,D}),\ast\right).$$

\end{lma}



\begin{dfn}\label{d:meridian}
    Fix a base point $*\in\mJ_\w^0\subset\mA_\w^0=\mA_{\w_0}^0$, a free loop $\{\gamma_t\}_{t\in[0,1]}\subset\mA_{\w_0}^0$ and a codimension $2$ submanifold $U_{\w_0,D}\subset \mA^0_{\w_0}\cup\mA_{\w_0}^2$.  Assume that $\gamma$ is a meridian of $U_{\w_0,D}$ in $\mA^0_{\w_0}$, and $\alpha$ is a path that connects $*$ and $\gamma_0$ in $\mA_{\w_0}^0$.  We call any based loop which is homotopic to the form of $\alpha^{-1}\circ\gamma\circ\alpha$ a \textbf{based meridian of $U_{\w_0,D}$}.
\end{dfn}

\begin{lma}\label{l:basedImage}
    Take a based meridian $\gamma_D$ of $U_{\w_0, D}$ for any $D\in\mL_\w$.  Consider the following composition of homomorphisms 
         \begin{multline}\label{e:generationFinal}
         \pi_1((\mA_{\w_0}^0\cup\mA_{\w_0}^2)\setminus(\bigcup_{D\in\mL_\w}U_{\w_0,D}),\ast)
         \xrightarrow{(\iota_{\w,\w_0})_*^{-1}}\pi_1(\mA_\w^0\cup\mT_\w,\ast)\\
         \xrightarrow{(\iota_\w)_*}\pi_1(\mA_\w,\ast)
         \xrightarrow{\sim}\pi_1(\mS_\w,\w)\xrightarrow{j_{\w}}\pi_0(Symp_h(X,\w)).
     \end{multline}

     Here, the based point $\ast$ can be taken as any almost complex structure in $\mJ_\w^0$.  Then the image of $[\gamma_D]$ in $\pi_0(Symp_h(X,\w))$ is the mapping class of a squared Lagrangian Dehn twist or its inverse.
\end{lma}

\begin{proof}
  Assume that $\gamma_D=\alpha^{-1}\circ\gamma\circ\alpha$.  As a free loop, $\gamma$ or $\gamma^{-1}$ is homotopic to some $\gamma^{L,\w_0}$.  Assume it is $\gamma$ without loss of generality.  By a further isotopy, we may assume that $\gamma_D=\alpha^{-1}\circ\gamma^{L,\w_0}\circ\alpha$.

  From Proposition \ref{c:meridian}, $\gamma^{L,\w_0}$ is a meridian of $U_{\w_0, D}$.  Using $\iota_{\w,\w_0}^{-1}$, one may also consider it as a based loop $\gamma^{L,\w}$ in $\mA_\w$ with base point $\gamma^{L,\w}_0\in\mJ_\w^0$.  Then $[\gamma^{L,\w}]$ induces a symplectic mapping class of a squared Dehn twist $[\tau_L^2]\in \pi_0(Symp(X,\w))$ from Lemma \ref{l:loopDT}.  The change of base point is given by the path of almost complex structures $\alpha_t$, which underlies a path of isotopic symplectic forms $\tau_t$ from the definition of $\mA_\w$, where $\tau_0=\tau_1=\w$.  From Moser's technique, this induces a family of diffeomorphisms $\{f_t\}_{t\in[0,1]}$ with the particular property that $(f_1)^*\w=\w$.  Therefore, the image of $[\gamma_D]$ via the map in \eqref{e:generationFinal} is precisely $[f_1^{-1}\circ\tau_L^2\circ f_1]=[\tau_{f_1(L)}^2]$: the first two maps are simply inclusions, and the third map gives a concatenation of $(f_t)_*(\w)\circ \w_t^L\circ(f_{1-t})_*(\w)$ from Lemma \ref{rem:AtoS} and \ref{l:loopDT}.  One then unrolls the definition of $j_\w$. 
\end{proof}

\begin{lma}\label{l:generationlma}
$\pi_1((\mA_{\w_0}^0\cup\mA_{\w_0}^2)\setminus(\bigcup_{D\in\mL_\w}U_{\w_0,D}))$ is generated by $i_{\w_0}(\pi_1(\Diff_0(X), id))$, along with the based meridians of $U_{\w_0,D_i}$ for $D_i\in\mL_\w$.


\end{lma}

\begin{proof}

  From Theorem \ref{aA}, $\pi_0(Symp_h(X,\w_0))=1$, so by the exact sequence \eqref{e:SES} and a codimension argument, $\pi_1(\mA^0_{\w_0}\cup \mA^2_{\w_0},\ast)\cong\pi_1(\mA_{\w_0},\ast)=i_{\w_0}(\pi_1(\Diff_0(X), id))$.

    To see the generation, suppose we have any based loop $\beta$ in $(\mA_{\w_0}^0\cup\mA_{\w_0}^2)\setminus(\bigcup_{D\in\mL_\w}U_{\w_0,D})$.  It is homotopic to the some loop $\beta'$ which gives an element in $i_{\w_0}(\pi_1(\Diff_0(X),id))$ in $\mA_{\w_0}^0\cup\mA_{\w_0}^2$.  Let this homotopy be smooth and transverse to each $U_{\w_0,D}$.  This homotopy can be modified, by adding a based meridian at each intersection with some $U_{\w_0,D}$ (see Figure \ref{fig:modification}).  Therefore, one obtains a homotopy from $\beta$ to a product of $\beta'$ and a series of based meridians.

\end{proof}

Corollary \ref{t:generation} is now an easy consequence of our previous discussions.

\begin{proof}[Proof of Corollary \ref{t:generation}]
  Since $k_\w$ in \eqref{e:SES} is trivial by \cite{LLW22}, we have a surjective homomorphism $j_\w: \pi_1(\mS_\w,\w)\twoheadrightarrow\pi_0(Symp_h(X,\w))$.  By Lemma \ref{rem:AtoS}, any loop $\alpha$ in $\mS_\w$ can be lifted by a choice of compatible loop $\alpha'$ in $\mA_\w$.  Lemma \ref{l:fundamentalgroup} guarantees that $\alpha'$ is homotopic to some loop $\alpha''$ in $(\mA_{\w_0}^0\cup\mA_{\w_0}^2)\setminus(\bigcup_{D\in\mL_\w}U_{\w_0,D})$.  \ref{l:basedImage} shows we can always decompose $\alpha''$ into loops generated from $i_{\w_0}\pi_1(\Diff_0(X), id)$ and based meridians.

  Lemma \ref{l:basedImage} asserts that the images of based meridians under \eqref{e:generationFinal} are squared Dehn twists (or their inverses).  For a loop $\beta\subset (\mA_{\w_0}^0\cup\mA_{\w_0}^2)\setminus(\bigcup_{D\in\mL_\w}U_{\w_0,D})$ whose homotopy class $[\beta]\in i_{\w_0}\pi_1(\Diff_0(X), id)$, we may assume that $\beta=\{\varphi_t(J_0)\}$, where $J_0\in (\mA_{\w_0}^0\cup\mA_{\w_0}^2)\setminus(\bigcup_{D\in\mL_\w}U_{\w_0,D})$ is the base point.  In fact, the base point $J_0$ is always taken in the generic part $\mA_{\w_0}^0=\mA_\w^0$, therefore, as an element in $\pi_1(\mA_\w,J_0)$, $[\beta]\in i_\w(\Diff(X),\id)$.  Therefore, $i_\w([\beta])=1\in\pi_0(Symp(X,\w))$, and $i_\w([\alpha])$ has the mapping class of a product of squared Dehn twists. 


\end{proof}

\begin{figure}[tb]
  \centering{}
  \includegraphics[scale=0.7]{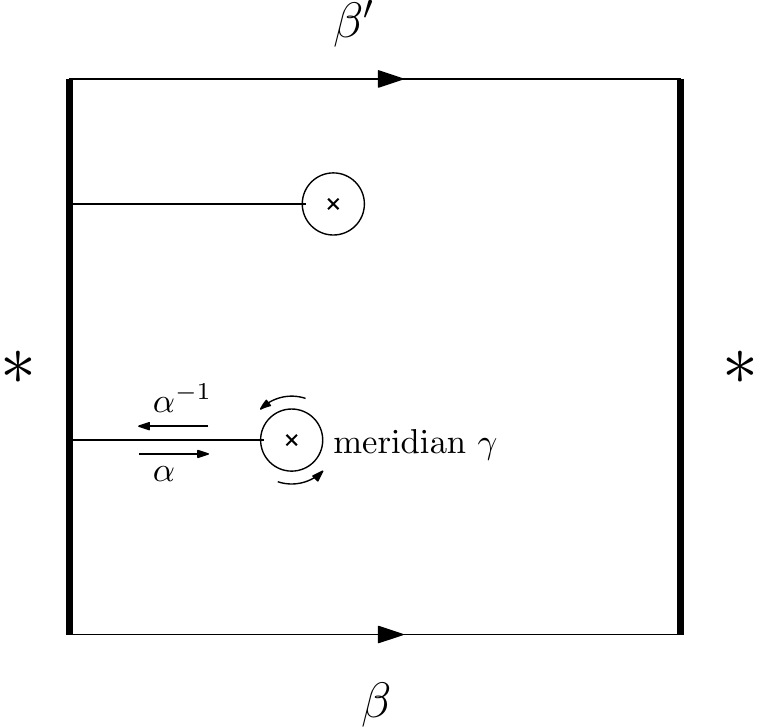}
  \caption{Modification of the homotopy}
  \label{fig:modification}
\end{figure}

\subsection{A parametrized Gromov-Witten invariant} 
\label{sub:a_parametrized_gromov_witten_invariant}

The goal of this section is to prove Proposition \ref{c:meridian}.

\subsubsection{Inputs from Symplectic Field Theory} 
\label{ssub:inputs_from_symplectic_field_theory}

We briefly review the basic setup and compactness results in symplectic field theory.  We keep the exposition minimum for our applications without proofs, and refer readers to the more comprehensive references for this topic, e.g. \cite{BEHWZ}.  For expositions closer to our applications, see  \cite{Hind04}, \cite{Eva}, \cite[Section 2.2]{LW12}, \cite[Section 4.1]{WuExotic}.

Given a Lagrangian sphere embedding $L\hookrightarrow (X^4,\w)$, we take its Weinstein neighborhood $\eW_L$ as in previous sections.  $\partial\eW_L\cong\partial T^*_\epsilon S^2$ obtains a standard contact structure $(\RP^3,\xi)$ and contact form $\lambda$ from the standard Liouville flow $\eta$ of $T^*S^2$.  Denote the Reeb vector field on $\RP^3$ as $R$, its periodic orbits come in an $S^2$-family and are of the same length. Each periodic orbit corresponds to a lift of a closed geodesic on $S^2$.

A symplectic neighborhood $N(\partial\eW_L)$ of $\partial\eW_L$ is symplectomorphic to $((-\epsilon,\epsilon)\times\RP^3, d(e^t\lambda))$ (this is a general property of \textit{contact type hypersurfaces}).  The first coordinate $t$ is induced by the flowtime along the Liouville flow $\eta$ from $\partial\eW_L$.  We say an almost complex structure $J$ is \textbf{$\lambda$-compatible} in $N(\partial\eW_L)$, if $d\lambda(\cdot, J\cdot)$ is a nondegenerate bundle metric on the contact structure along $\RP^3$.  $J$ is \textbf{adjusted to $\eta$}, or \textbf{cylindrical near the boundary} if it is

\begin{itemize}
  \item $\eta$-invariant,
  \item  $\lambda$-compatible, and
  \item $J(\partial_t)=R$ inside an open set $(-\epsilon',\epsilon')\times\RP^3\Subset N(\partial\eW_L)$.
\end{itemize}
 Any given almost complex structure can be isotoped into an adjusted one by a deformation supported in a slightly larger open set that contains $N(\partial\eW_L)$.

 There are three main ingredients from SFT that we will need.

\subsubsection*{Compactness theorem} 
\label{ssub:compactness_theorem}

The \textbf{neck-stretching process} introduced in \cite{EGHsft,BEHWZ} is a deformation of an almost complex structure $J_0=J$ adjusted to $\eta$.  More concretely, one may construct this family $\{J_s\}_{s\ge0}$ by prolonging the cylindrical part and extending by the $\eta$-invariance.  Alternatively, we consider $J_s=J$ outside of $N(\partial\eW_L)$, while in the cylindrical part, we modify the third condition of adjustness by $J_s(\partial_t)=sR$, and ask the $\eta$-invariance to hold only for $J|_\xi$.  When $s\to\infty$, $J_\infty$ defines an almost complex structure on the union of the following three pieces

\begin{itemize}
  \item $\ov W$ is the symplectic completion of $X\setminus N(\partial\eW_L)$;
  \item $\ov U$ is the symplectic completion of $N(\partial\eW_L)$, which is symplectomorphic to $T^*S^2$;
  \item $S(\partial\eW_L)$, the symplectization of the standard contact $\RP^3$.
\end{itemize}

Each of these pieces obtains a cylindrical collar ($S(\partial\eW_L)$ is cylindrical on its own) and has a natural $t$-parameter therein.  Then $J_\infty=J$ outside of $N(\partial\eW_L)$ (this complement has two connected components contained in $\ov W$ and $\ov U$, respectively), and it is $\partial_t$-invariant inside each cylindrical collar.

Given a neck-stretching family $\{J_s\}_{s\ge0}$, we have the following special case of the celebrated SFT compactness theorem from \cite{BEHWZ}:


\begin{thm}\label{t:SFTcompactness}
   Given $s_i\to \infty$, suppose that $u_i$ are $J_{s_i}$-holomorphic and $[u_i]=D$ has a fixed homology class.  Then $u_i$ converges in the SFT sense to a $J_\infty$-holomorphic building $u_\infty$.
\end{thm}

In particular, this yields a (possibly disconnected or empty) $J_\infty$-holomorphic curves with punctures that converges to Reeb orbits in the ideal contact boundary in each pieces $\ov W$, $\ov U$ and $S(\partial\eW_L)$ defined above.  Note that there might be more than one level of curves that are mapped to $S(\partial\eW_L)$.

We should point out that, one may glue up $\ov W$, $\ov U$ and $S(\partial\eW_L)$ smoothly back into $X$.  In doing that, the neck-stretching limit $u_\infty$ as above also gives a smooth closed curve in $X$, which has $[u_\infty]=D$.

\subsubsection*{Relative first chern class}

Given a holomorphic curve with asymptotic punctures at the ideal contact boundary, one would like to define a similar invariant as the Chern class.  For this purpose, one needs to fix a trivialization of the contact distribution near each Reeb orbit appearing as an asymptotic puncture, which we denote as $\Phi$.  This trivialization will also determine the Conley-Zehnder index of the relevant Reeb orbit.  Note that, since $\partial\eW_L$ is a standard contact $\RP^3$, any Reeb orbit therein is a multiple cover of a lift of a geodesic in the $S^2$ equipped with the round metric.  Denote $cov(\gamma)$ as the multiplicity of a Reeb orbit $\gamma$ over the simple one. Although we will not give the explicit definition of the Conley-Zehnder index or relative Chern class (interested readers can find an accessible exposition tailored to our situation in  \cite{Hind04,Eva,LW12}), the following calculations will suffice for our applications.

\begin{lma}[\cite{Hind04}, Lemma 7; \cite{Eva}, 6.3]\label{l:firstchern}
    Given a $J_\infty$-holomorphic building $u_\infty$ as the limit of a sequence of homologous $u_i$ stretched along $\partial\eW_L$, and assume that $u^W$ is the holomorphic curve in $\ov W$, $u^S$ is a collection of holomorphic curves in $S(\partial\eW_L)$, while $u^U$ is the holomorphic curve in $\ov U$.

    Then there exists a trivialisation of the contact type hyperplanes $\Phi$ near each orbit $\gamma$ on $(\RP^3,\xi_{std})$, such that their Conley-Zehnder index is given by $2cov(\gamma)$, and the relative Chern classes satisfies $c_1^\Phi(u^U)=c_1^\Phi(u^S)=0$ and $c_1^\Phi(u^W)=c_1([u_i])$.
\end{lma}

\subsubsection*{Dimension formula}

\begin{thm}\label{t:dimension}
   Suppose that $u^Z:\Sigma\to (Z,\w)$ is a finite energy punctured holomorphic curve, where $(Z,\w)$ is any of $\ov U$, $\ov W$ or $S(\partial\eW_L)$.  Suppose $\Sigma$ has genus zero, while $u^Z$ has $s^+$ positive punctures, converging to a series of Reeb orbits $\gamma_k^+$, $1\le k\le s^+$; and $s^-$ negative punctures, converging to $\gamma_k^-$, $1\le k\le s^-$.  Then
   \begin{equation}\label{e:index}
        \mbox{ind}(u^Z)=-2+2(s^++s^-)+2c_1^{\Phi}([u^Z])+\sum_{k=1}^{s^+}2cov(\gamma_k^+)-\sum_{k=1}^{s^-}2cov(\gamma_k^-).
   \end{equation}

\end{thm}

Here $c_1^\Phi(TZ)$ is the relative first Chern class of $(TZ,J_\infty)$ relative to the trivialization $\Phi$ along the ends, and $[u^Z]$ denotes the Borel-Moore homology class of $u$ as a non-compact cycle.

\subsubsection{A Gromov-Witten reformulation of meridians}

 We are now ready to circle back to the Definition of a meridian \eqref{e:nbhLES}.

 Given a rational surface $(X,\w)$ and a Lagrangian sphere $L\subset X$, we deform $\w$ to $\w_L$ as in \ref{sub:meridians_of_}.  Take the loop $\gamma^{L,\w_L}$ as in Proposition \ref{c:meridian}, and a disk $\sigma\subset\mA_{\w_L}\setminus \mA_{\w_L}^4$ with $\partial\sigma=\gamma^{L,\w_L}$.   Consider the moduli space defined by the following parametrized Cauchy-Riemann problem.

\begin{equation} \label{e:moduli}
     \left\{
        \begin{aligned}
          &\ov\partial_{J_{\sigma(\zeta)}}u=0\\
           &u:\CP^1\to X,\hskip 2mm \sigma: D^2\to \mA_{\w_L},\hskip 2mm\zeta\in D^2 \\
          &[u]=[L], \hskip2mm \sigma|_{\partial D^2}=\gamma^{L,\w_L},
        \end{aligned}
     \right.
\end{equation}

Denote the moduli space of parametrized holomorphic maps satisfying \eqref{e:moduli} as $\eM(\sigma, [L])$.  We emphasize that $\gamma^{L,\w}$ tames both $\w$ and $\w_L$, and we fix the standard orientation on $D^2$ as a unit disk in $\CC$.

The Fredholm property of this counting problem is similar to the usual counting \cite{MS04}, and has been studied in detailed in \cite{BuserelGW, LOrelGW, McDuffBlowup}.  \cite[Proposition 2.8]{BuserelGW} showed this problem has transversality if and only if $\sigma\pitchfork U_{[L]}$.  This is indeed a consequence of \cite[Lemma A.3.6]{MS04}, as well as the proof of \cite[Theorem 3.1.5(II)]{MS04}.  These theorems assert that $d\pi(u,J):T_{(u,J)}\eM([L];\mJ^l)\to T_J\mJ^l$ has its kernel and cokernel both isomorphic to that of $D_u$, the linearization the $\ov\partial$-operator.  Here $\mJ^l$ is the space of $\w$-tamed almost complex structures which are $l$-differentiable, and $\eM([L];\mJ^l)$ is the universal moduli of the $\ov\partial_J$-equation over all $J\in\mJ^l$.  Since $\ker(d\pi(u,J))|_{\sigma(\zeta)}=T_{\sigma(\zeta)}U_{[L]}$, and $\coker(d\pi(u,J))|_{\sigma(\zeta)}=N_{\sigma(\zeta)}U_{[L]}$ for every $\sigma(\zeta)\in U_{[L]}$.  Therefore, $T_\zeta\sigma$ surjects to onto $\coker(D_u)$ if and only if it intersects $U_{[L]}$ transversally when $\ov\partial_{J_{\sigma(\zeta)}}u=0$.

Regular solutions of this moduli space also inherit a natural orientation from that of $\sigma$, following the appendix of \cite{Abr98}.  The determinant line bundle of the linearized $\ov\partial$-operator can be written as $\det(D_u)=\Lambda^{\max}(\ker(D_u))\otimes\Lambda^{\max}(\coker(D_u))$.  Since $U_{[L]}$ is connected, and it contains an almost complex structure $J_{\inte}$, which is integrable in a neighborhood of the unique solution $[u_{\inte}]=[L]$ satisfying $\ov\partial_{J_{\inte}}u_{\inte}=0$.  The natural orientation of $det(D_{u_{\inte}})$ therefore gives a canonical orientation of $det(D_u)$ for all solutions $u$ satisfying $\ov\partial_J u=0$ where $J\in U_{[L]}$.  By positivity of intersections, all $\ov\partial$-solutions in class $[L]$ are unique.  Hence $\Lambda^{\max}\ker(D_u)$ also has a natural orientation as the orientation line bundle of $G=PSL(2,\CC)$.  Therefore, $\coker(D_u)$ is oriented naturally.  As mentioned above, the $\coker(D_u)$ is naturally isomorphic to $\coker(d\pi(u,J_{\sigma(\zeta)}))$, while the latter is identified with the normal fiber of the corresponding $J\in U_{[L]}$.  Therefore, if $(u,J_{\sigma(\zeta)})\in\eM(\sigma,[L])$ is a transversal solution, its sign can be defined as $+1$ if $d\sigma(T_\zeta\sigma)$ has the same orientation as $\coker(d\pi(u,J_{\sigma(\zeta)}))=N_{\sigma(\zeta)}U_{[D]}$, and $-1$ otherwise.  

Therefore, the linking number of the loop $\gamma^{L,\w_L}$ coincides with the algebraic count of solutions for \eqref{e:moduli}.  In summary, we have the following equivalent definition of a meridian.

\begin{lma}\label{l:enumerative}

   The enumerative count of $\#\eM(\sigma,[L])$ coincides with the algebraic intersection of $\sigma$ and $U_{[L]}$ inside the tubular neighborhood $V$.  Therefore, $\gamma^{L,\w_L}$ is a meridian if and only if $\#\eM(\sigma,[L])=\pm1$, for any $\sigma:D^2\to \mA_{\w_L}$, such that $\partial\sigma=\gamma^{L,\w_L}$ and $\sigma\subset V$ which is a tubular neighborhood of $U_{[L]}$.
\end{lma}


\subsubsection{Proof of Proposition \ref{c:meridian}} 
\label{sub:proof_of_corollary_c:meridian}

In view of Lemma \ref{l:enumerative}, Proposition \ref{c:meridian} boils down to the following proposition.

\begin{prp}\label{prp:dehnSolution}
   Given a positive rational surface $(X, \w)$ and a Lagrangian sphere $L\subset X$.  Take a Weinstein neighborhood $\eW_L$ of $L$.  Construct $\gamma^{L,\w_L}$ with base point $J_0$ in a decentered tubular neighborhood $V\setminus U_{[L]}$ of $U_{[L]}$, then perturb $\w$ to a form $\w_L$ as in Definition \ref{d:wL}.

   Then the algebraic count of $\#\eM(\sigma,[L])$ of the parametrized Gromov-Witten problem \eqref{e:moduli} is $\pm 1$.
\end{prp}

We first prove a special case when $X=X_5:=(\CP^2\#5\ov\CP^2,\w_{\mon})$, the five point blow-up with a monotone symplectic form.  Take the homology class $E_1-E_2$, which admits an embedded Lagrangian sphere $L_0$ (see \cite{LW12}).

Take a neighborhood $\eW_{L_0}$ of $L_0$, and an $\w_{\mon}$-tamed almost complex structure $I_0$ adjusted to $\partial\eW_{L_0}$ and construct $\w_{L_0}$ and $\gamma^{L_0,\w_{\mon}}$ accordingly.  Note that from \cite[Corollary 1.2]{BLW12}, the choice of $L_0$ is unique up to a symplectomorphism of $X_5$ so is not of our concern.


If $\#\eM(\sigma_{L_0},[L_0])=0$, we see from \eqref{e:nbhLES} that $\gamma^{L_0,\w_{\mon}}$ is contractible in $V\backslash U_{[L_0]}$, hence $\sigma_{L_0}$ can be isotoped so that there does not exist $\sigma_{L_0}(\zeta)$-curves inside $\eW_{L_0}$.  Therefore, the disk $\sigma_{L_0}\subset\mA^0_{\w_{L_0}}$ when perturbed generically to avoid also $\mA^4_{\w_{L_0}}$.  Moreover, from Lemma \ref{l:neighborhood}, one may assume $J_{\sigma_{L_0}}=I_0$ outside $\partial\eW_{L_0}$.  Therefore, a generic perturbation outside $\partial\eW_{L_0}$ of the base point $I_0$ also guarantees that $\sigma_{L_0}$ is disjoint from any other $U_D$ for $D\in\mL_{\w_{\mon}}$.

  Therefore, we may assume that $\sigma_{L_0}\subset\mA^0_{\w_{\mon}}$ from Lemma \ref{Combinf}.  But this implies that the squared Dehn twist $\tau_{L_0}^2$ is isotopic to identity: $[\gamma^{L_0,\w_\mon}]=1\in\pi_1(\mA_{\w_{\mon}},I_0)\cong\pi_1(\mS_{\w_{\mon}},\w_\mon)$ is mapped to $[\tau_{L_0}^2]\in\pi_0(Symp_h(X,\w_\mon))$ by the map $j_{\w_\mon}$ in \eqref{e:SES} from Lemma \ref{l:loopDT}.  This is a contradiction with \cite[Example 2.10]{Sei08}.  Therefore, $\#\eM(\sigma_{L_0},[L_0])\neq0$.

  Take a loop $\gamma'\subset\mA_{\w_{L_0}}^0$ based at $I_0$, whose linking number against $U_{[L_0]}$ is $1$.  By Lemma \ref{Combinf}, this is also a loop in $\mA^0_{\w_{\mon}}$.  Therefore, $j_{\w_{\mon}}([\gamma^{L_0,\w_\mon}])=(\#\eM(\sigma_{L_0},[L_0]))\cdot j_{\w_{\mon}}([\gamma'])\in\pi_0(Symp_h(X_5,\w_\mon))$.  However, from \cite{Ev11} we know that $\pi_0(Symp_h(X_5))\cong PB_5(S^2)/\mathbb{Z}_2$.  In \cite[Example 1.13]{Sei08}, Seidel constructed an injective homomorphism $\pi_1(\Conf_5(\CP^2)/\PSL_3(\CC))\to \pi_0(Symp_h(X_5,\w_\mon))$.  We also have $\pi_1(\Conf_5(\CP^2)/\PSL_3(\CC))\cong \pi_1(\Conf_5(\CP^1)/\PSL_2(\CC))=PB_5(S^2)/\ZZ_2$. Note that the image under Seidel's monomorphism of the standard generators of $PB_5(S^2)/\ZZ_2$ matches the generators in \cite{Ev11}, as explained in \cite[Section 5]{Wu13}.  Since all generators in $\pi_1(\Conf_5(\CP^2)/\PSL_3(\CC))$ are given by a meridian around the discriminant locus of colliding two points in $\Conf_5(\CP^2)$, their monodromies corresponds exactly to the squared Dehn twists (see \cite[Lemma 1.11]{Sei08}).  Going back to $\Conf_5(\CP^1)/PGL_2(\CC)$, this shows $\tau_{L_0}^2$ corresponds to a standard pure braid generator, up to a conjugation.

  A pure braid generator cannot be a non-trivial power of any other pure braid elements.  This can be seen from the proof of \cite[Theorem 5(e)]{GG13} (see equation (16) and the paragraph behind it): the abelianization of of $PB_n(S^2)$ is generated by $(n(n-3)+2)/2$ of the standard pure braid generators (different choices of this set of generators are conjugate therefore irrelevant).  Since $[\tau_{L_0}^2]\in \pi_0(Symp_h(X_5))\cong PB_5(S^2)/\mathbb{Z}_2$ is a generator, it yields a primitive element in the abelianization $\mathbb{Z}^6\cong Ab(PB_5(S^2)/\mathbb{Z}_2)$, hence cannot be a power of any other elements.  This forces $\#\eM(\sigma,[L_0])=\pm 1$.

In summary, we have

\begin{lma}\label{l:compactcount}

    For any choice of $\sigma: D^2\to\mA_{\w_\mon}$ with $\partial\sigma=\gamma^{L_0,\w_\mon}$, the algebraic count of solutions of parametrized Gromov-Witten problem \eqref{e:moduli} is $\pm1$.
\end{lma}

\begin{proof}[Proof of Proposition \ref{prp:dehnSolution}]

  We first consider a counting problem similar to \eqref{e:moduli} in the local setting of $T^*S^2$.  In this case, denote the zero section as $Z$, and $\w$ can be taken as the standard Liouville form, while the perturbation in a neighborhood $\eW_Z$ is denoted as $\w_Z$.  $\gamma^{Z,\w_Z}$ and $\sigma_Z$ can be defined similarly starting from an almost complex structure $J$ that is cylindrical outside $\eW_Z$.  We may then consider the parametrized Gromov-Witten invariant $\#\eM(\sigma_Z,[Z])$.

  We will apply an SFT argument on an arbitrary positive rational surface $(X,\w)$ containing a Lagrangian sphere $L\subset X$.  Take an almost complex structure such that $(\eW_{L}, J|_{\eW_{L}})\cong(\eW_{Z}, J_Z|_{\eW_{Z}})$ as an open almost complex manifold, while the isomorphism is required to respect the cylindrical structure near the boundary.  Construct a $\sigma$-family from of almost complex structures as before in $\eW_L$.  Since it is fixed outside the neighborhood $\eW_{L}$ and adjusted to its boundary, we may apply a neck-stretching operation and obtain a family of almost complex structures paramatrized by disks $\{\sigma_t\}_{t\in[0,\infty]}\subset\mA_{\w_L}$, where $\sigma_Z:=\sigma_\infty|_{T^*S^2}$ in the limit.  Clearly, along the stretching process, we preserve all solutions that are contained in $\eW_Z$, and by maximum principle, any solution to \eqref{e:moduli} inside $T^*S^2$ stays in the neighborhood $\eW_{Z}$.

  We claim that when $t$ is large enough, $\sigma_t$-solutions to \eqref{e:moduli} are contained entirely in $\eW_{L}$.  Otherwise, there is a sequence of $J_{\sigma_{t_i}(\zeta_{t_i})}$-holomorphic curves $u_i$, whose homology class $[u_i]=[L]$, and their SFT limit $u_\infty$ is a $\sigma_\infty(\zeta)$-holomorphic building for some $\zeta\in D^2$ with non-empty components in both $\eW_{L}$ and its complement.  Since the components in $X\setminus\eW_{L}$ has only negative ends, its index computed by \eqref{e:index} reads

  \begin{equation}\label{e:newindex}
       \mbox{ind}(u)=-2+2s^- +2c_1^{\Phi}(TX)([u])-\sum_{k=1}^{s_1} 2\cov(\gamma_k).
  \end{equation}

By Lemma \ref{l:firstchern}, the total $c_1^\Phi$ of all components in $X\setminus\eW_{L}$ vanishes since it equals $c_1([L])$.  This means at least one of these connected components in $X\setminus\eW_L$ would have $c_1^\Phi(u)\le0$.  Since $\cov(\gamma_k)\ge1$ for each asymptotic, we have $\mbox{ind}(u)<0$ for such a complement.  However, this can always be avoided by choosing $J$ generically in $X\setminus\eW_{L}$. Therefore, all curves contributing to $\#\eM(\sigma_Z,[Z])$ must come from solutions in $\eW_{L}$, which is unaffected by the neck-stretching including their transversality properties.  Therefore, we have

\begin{equation}\label{e:countEqual}
     \#\eM(\sigma_Z,[Z])=\#\eM(\sigma,[L]),
\end{equation}

where the right hand side is the counting in $(X,\w_L)$.  Since the left hand side of \eqref{e:countEqual} is independent of $(X,L)$, we know the count for any $(X,L)$ is $\pm1$ from Lemma \ref{l:compactcount} by taking $X=X_5$.  This concludes Proposition \ref{prp:dehnSolution}.

\end{proof}

\begin{rmk}\label{rem:general}
   As is clear from the proof of Proposition \eqref{c:meridian}, the parametrized Gromov-Witten problem \ref{e:moduli} does not even depend on the fact that $X$ is a rational surface.  Indeed, the whole proof carries over to any symplectic four manifold, except one only has a codimension $2$ submanifold $U_{[L]}$ but not the coarse stratification any more.

   Our computation is similar in many ways to \cite{SmirnovDehn}, where Smirnov established a similar non-vanishing result in the Seiberg-Witten setting.  Indeed, the loop $\gamma^{L,\w_L}$ as we constructed should exactly correspond to the boundary of the Atiyah flop family, and it is possible to approach the non-compact count this way, but we will still need SFT to show the count is independent of $(X,L)$.
\end{rmk}

\printbibliography

\end{document}